\documentclass{amsart}
\usepackage{amssymb}
\usepackage{amsfonts}
\usepackage{geometry}

\newtheorem{theorem}{Theorem}
\theoremstyle{plain}

\newtheorem{claim}{Claim}

\newtheorem{corollary}{Corollary}

\newtheorem{definition}{Definition}
\newtheorem{example}{Example}

\newtheorem{lemma}{Lemma}

\newtheorem{proposition}{Proposition}
\newtheorem{remark}{Remark}

\numberwithin{equation}{section}
\input{tcilatex}
\begin{document}
\title{Weighted inequalities for product fractional integrals}
\author{Eric Sawyer}
\address{McMaster University}
\author{Zipeng Wang}
\address{McMaster University}

\begin{abstract}
We investigate one and two weight norm inequalities for product fractional
integrals 
\begin{equation*}
I_{\alpha ,\beta }^{m,n}f\left( x,y\right) =\diint\limits_{\mathbb{R}%
^{m}\times \mathbb{R}^{n}}\left\vert x-u\right\vert ^{\alpha -m}\left\vert
y-t\right\vert ^{\beta -n}f\left( u,t\right) dudt
\end{equation*}%
in $\mathbb{R}^{m}\times \mathbb{R}^{n}$. We show that in the one weight
case, most of the $1$-parameter theory carries over to the $2$-parameter
setting - the one weight inequality%
\begin{equation*}
\left\Vert \left( I_{\alpha ,\beta }^{m,n}f\right) w\right\Vert _{L^{q}}\leq
N_{p,q}^{\left( \alpha ,m\right) ,\left( \beta ,n\right) }\left\Vert
fw\right\Vert _{L^{p}},\ \ \ \ \ f\geq 0,
\end{equation*}%
is equivalent to finiteness of the rectangle characteristic 
\begin{equation*}
A_{p,q}^{\left( \alpha ,\beta \right) ,\left( m,n\right) }\left( w\right)
=\sup_{I,J}\ \left\vert I\right\vert ^{\frac{\alpha }{m}-1}\left\vert
J\right\vert ^{\frac{\beta }{n}-1}\left( \diint\limits_{I\times
J}w^{q}\right) ^{\frac{1}{q}}\left( \diint\limits_{I\times J}w^{-p^{\prime
}}\right) ^{\frac{1}{p^{\prime }}},
\end{equation*}%
which is in turn equivalent to the diagonal and balanced equalities%
\begin{equation*}
0<\frac{\alpha }{m}=\frac{\beta }{n}=\frac{1}{p}-\frac{1}{q}.
\end{equation*}%
Moreover, the optimal power of the characteristic that bounds the norm is $%
2+2\max \left\{ \frac{p^{\prime }}{q},\frac{q}{p^{\prime }}\right\} $.
However, in the two weight case, apart from the trivial case of product
weights, the rectangle characteristic fails to control the operator norm of $%
I_{\alpha ,\beta }:L^{p}\left( v^{p}\right) \rightarrow L^{q}\left(
w^{q}\right) $ in general.

On the other hand, in the half-balanced case $\min \left\{ \frac{\alpha }{m},%
\frac{\beta }{n}\right\} =\frac{1}{p}-\frac{1}{q}$, we prove that the
rectangle characteristic is sufficient for the weighted norm inequality in
the presence of a side condition - either $w^{q}$ is in product $A_{1}$ or $%
v^{-p^{\prime }\text{ }}$ is in product $A_{1}$.

Moreover, the Stein-Weiss extension of the classical Hardy - Littlewood -
Sobolev inequality to power weights carries over to the $2$-parameter
setting with \emph{nonproduct} power weights using a `sandwiching'
technique, providing our main positive result in two weight $2$-parameter
theory.
\end{abstract}

\maketitle
\tableofcontents

\section{Introduction}

The theory of weighted norm inequalities for product operators, i.e. those
operators commuting with a multiparameter family of dilations, has proved
challenging since the pioneering work of Robert Fefferman \cite{Fef} in the
1980's involving covering lemmas for collections of rectangles. The purpose
of the present paper is to settle some of the basic questions arising in the
weighted theory for the special case of product \emph{fractional} integrals,
in particular the relationship between their norm inequalities and their
associated rectangle characteristics. Our four main results can be split
into two distinct parts, which are presented largely independent of each
other:

\begin{enumerate}
\item In the first part of the paper, we consider the general two weight
norm inequality $I_{\alpha ,\beta }^{m,n}:L^{p}\left( v^{p}\right)
\rightarrow L^{q}\left( w^{q}\right) $ for \textbf{product} fractional
integrals $I_{\alpha ,\beta }^{m,n}$ when the indices are \emph{subbalanced}:%
\begin{equation*}
\min \left\{ \frac{\alpha }{m},\frac{\beta }{n}\right\} \geq \frac{1}{p}-%
\frac{1}{q}>0.
\end{equation*}%
We focus separately on the three subcases where the indices are \emph{%
balanced} $\frac{\alpha }{m}=\frac{\beta }{n}=\frac{1}{p}-\frac{1}{q}>0$, 
\emph{half} \emph{subbalanced} $\min \left\{ \frac{\alpha }{m},\frac{\beta }{%
n}\right\} =\frac{1}{p}-\frac{1}{q}>0$ with $\frac{\alpha }{m}\neq \frac{%
\beta }{n}$, and finally \emph{strictly} \emph{subbalanced}, $\min \left\{ 
\frac{\alpha }{m},\frac{\beta }{n}\right\} >\frac{1}{p}-\frac{1}{q}>0$. It
turns out that the one weight theory in $1$-parameter carries over to the
product setting in the balanced case, some of the familiar two weight theory
in $1$-parameter carries over in the half balanced case, and finally, the
rectangle characteristic is not sufficient for the norm inequality in the
strictly subbalanced case, without assuming additional side conditions on
the weights. More precisely, we prove:

\begin{enumerate}
\item In the balanced case, the one weight inequality in the product setting
is equivalent to\ finiteness of the product Muckenhoupt characteristic.

\item In the half balanced case, we use the one weight inequality to show
that the two weight inequality holds if in addition to the Muckenhoupt
characteristic, we have one of the side conditions, $w^{q}\in A_{1}\times
A_{1}$ or $v^{-p^{\prime }}\in A_{1}\times A_{1}$.

\item In the strictly sub-balanced case, a simple construction shows that
the rectangle characteristic is not sufficient\footnote{%
We thank H. Tanaka for pointing out an error in our counterexample with
rectangle $A_{1}$\ weights in the previous version of this paper, and also
for bringing to our attention "The $n$-linear embedding theorem for dyadic
rectangles" by H. Tanaka and K. Yabuta, \texttt{arXiv 1710.08059v1}, which
obtains boundedness of certain product fractional integrals with reverse
doubling weights. The counterexample for the two-tailed characteristic in
the previous version of this paper also contained an error.}.
\end{enumerate}

\item In the second part of the paper, we give a sharp product version of
the Stein-Weiss extension of the classical Hardy-Littlewood-Sobolev theorem
for (nonproduct) power weights, which we establish by a (somewhat
complicated) method of iteration, something traditionally thought unlikely.
\end{enumerate}

Our positive results for one weight theory and two power weight theory are
obtained using the tools of iteration, Minkowski's inequality and a
sandwiching argument. Now we begin to describe these matters in detail.

Let $m,n\geq 1$. For indices $1<p,q<\infty $ and $0<\alpha <m$ and $0<\beta
<n$, we consider the weighted norm inequality 
\begin{equation*}
\left\{ \diint\limits_{\mathbb{R}^{m}\times \mathbb{R}^{n}}I_{\alpha ,\beta
}^{m,n}f\left( x,y\right) ^{q}\ w\left( x,y\right) ^{q}\ dxdy\right\} ^{%
\frac{1}{q}}\leq N_{p,q}^{\left( \alpha ,\beta \right) ,\left( m,n\right)
}\left( v,w\right) \left\{ \diint\limits_{\mathbb{R}^{m}\times \mathbb{R}%
^{n}}f\left( u,t\right) ^{p}\ v\left( u,t\right) ^{p}\ dudt\right\} ^{\frac{1%
}{p}}
\end{equation*}%
for the product fractional integral%
\begin{equation*}
I_{\alpha ,\beta }^{m,n}f\left( x,y\right) =\diint\limits_{\mathbb{R}%
^{m}\times \mathbb{R}^{n}}\left\vert x-u\right\vert ^{\alpha -m}\left\vert
y-t\right\vert ^{\beta -n}f\left( u,t\right) dudt,\ \ \ \ \ \left(
x,y\right) \in \mathbb{R}^{m}\times \mathbb{R}^{n},
\end{equation*}%
and characterize when one weight and two weight inequalities are equivalent
to the corresponding product fractional Muckenhoupt characteristic,%
\begin{equation}
A_{p,q}^{\left( \alpha ,\beta \right) ,\left( m,n\right) }\left( v,w\right)
\equiv \sup_{I\subset \mathbb{R}^{m},\ J\subset \mathbb{R}^{n}}\left\vert
I\right\vert ^{\frac{\alpha }{m}-1}\left\vert J\right\vert ^{\frac{\beta }{n}%
-1}\left( \diint\limits_{I\times J}w^{q}\right) ^{\frac{1}{q}}\left(
\diint\limits_{I\times J}v^{-p^{\prime }}\right) ^{\frac{1}{p^{\prime }}},
\label{A w v}
\end{equation}%
and its two-tailed variant,%
\begin{eqnarray}
&&\widehat{A}_{p,q}^{\left( m,n\right) ,\left( \alpha ,\beta \right) }\left(
v,w\right)  \label{A w v tail} \\
&\equiv &\sup_{I\times J\subset \mathbb{R}^{m}\times \mathbb{R}%
^{n}}\left\vert I\right\vert ^{\frac{\alpha }{m}-1}\left\vert J\right\vert ^{%
\frac{\beta }{n}-1}\left( \diint\limits_{\mathbb{R}^{m}\times \mathbb{R}%
^{n}}\left( \widehat{s}_{I\times J}w\right) ^{q}d\omega \right) ^{\frac{1}{q}%
}\left( \diint\limits_{\mathbb{R}^{m}\times \mathbb{R}^{n}}\left( \widehat{s}%
_{I\times J}v^{-1}\right) ^{p^{\prime }}d\sigma \right) ^{\frac{1}{p^{\prime
}}},  \notag
\end{eqnarray}%
where%
\begin{equation}
\widehat{s}_{I\times J}\left( x,y\right) \equiv \left( 1+\frac{\left\vert
x-c_{I}\right\vert }{\left\vert I\right\vert ^{\frac{1}{m}}}\right) ^{\alpha
-m}\left( 1+\frac{\left\vert y-c_{J}\right\vert }{\left\vert J\right\vert ^{%
\frac{1}{n}}}\right) ^{\beta -n}.  \label{def s hat}
\end{equation}

For the \emph{one weight} inequality when $v=w$, we show that the indices
must be balanced and diagonal, i.e.%
\begin{equation*}
\frac{1}{p}-\frac{1}{q}=\frac{\alpha }{m}=\frac{\beta }{n},
\end{equation*}%
and that $p<q$; then the finiteness of the operator norm $N_{p,q}^{\left(
m,n\right) }\left( w\right) $ is equivalent to finiteness of the
characteristic $A_{p,q}^{\left( m,n\right) }\left( w\right) $, where as is
conventional, we are suppressing redundant indices in the one weight
diagonal balanced case. In addition, we characterize the optimal power of
the characteristic that controls the operator norm. These one weight results
are proved by an iteration strategy using Minkowski's inequality that also
yields two weight results for the special case of \emph{product weights}.

For the general two weight case, we show that in the absence of any side
conditions on the weight pair $\left( v,w\right) $, the operator norm $%
N_{p,q}^{\left( \alpha ,\beta \right) ,\left( m,n\right) }\left( v,w\right) $
is \emph{never} controlled by the two-tailed characteristic $\widehat{A}%
_{p,q}^{\left( m,n\right) ,\left( \alpha ,\beta \right) }\left( v,w\right) $%
, not even the weak type operator norm of the much smaller dyadic fractional
maximal function $M_{\alpha ,\beta }^{\limfunc{dy}}$ (see below for
definitions). On the other hand, new two weight results can be obtained from
known norm inequalities, such as the one weight and product weight results
mentioned above, by the technique of `sandwiching'.

\begin{lemma}
\label{sandwich}If $\left\{ \left( V^{i},W^{i}\right) \right\} _{i=1}^{N}$
is a sequence of weight pairs `sandwiched' in a weight pair $\left(
v,w\right) $, i.e. 
\begin{equation*}
\frac{w\left( x,y\right) }{v\left( u,t\right) }\leq \sum_{i=1}^{N}\frac{%
W^{i}\left( x,y\right) }{V^{i}\left( u,t\right) }\ ,
\end{equation*}%
then\footnote{%
In the case $N=1$, the hypothesis is implied by \thinspace $w\leq W^{1}$ and 
$V^{1}\leq v$, hence the terminology `sandwiched'.} 
\begin{equation*}
N_{p,q}^{\left( \alpha ,\beta \right) ,\left( m,n\right) }\left( v,w\right)
\leq \sum_{i=1}^{N}N_{p,q}^{\left( \alpha ,\beta \right) ,\left( m,n\right)
}\left( V^{i},W^{i}\right) .
\end{equation*}
\end{lemma}

\begin{proof}
This follows immediately from setting $g=fV$ in the identity, 
\begin{equation*}
N_{p,q}^{\left( \alpha ,\beta \right) ,\left( m,n\right) }\left( V,W\right)
=\sup_{\left\Vert g\right\Vert _{L^{p}}\leq 1,\left\Vert h\right\Vert
_{L^{q^{\prime }}}\leq 1}\diint\limits_{\mathbb{R}^{m}\times \mathbb{R}%
^{n}}h\left( x,y\right) W\left( x,y\right) \left\vert x-u\right\vert
^{\alpha -m}\left\vert y-t\right\vert ^{\beta -n}\frac{g\left( u,t\right) }{%
V\left( u,t\right) }dxdydudt.
\end{equation*}
\end{proof}

We mention two simple examples of sandwiching, the first example sandwiching
a one weight pair, and the second example sandwiching two product weight
pairs:

\begin{enumerate}
\item In the case of diagonal and balanced indices $\frac{1}{p}-\frac{1}{q}=%
\frac{\alpha }{m}=\frac{\beta }{n}$, if there is a \emph{one weight} pair $%
\left( u,u\right) $ with $A_{p,q}^{\left( \alpha ,\beta \right) ,\left(
m,n\right) }\left( u\right) <\infty $ sandwiched in $\left( v,w\right) $,
then $N_{p,q}^{\left( \alpha ,\beta \right) ,\left( m,n\right) }\left(
v,w\right) \lesssim A_{p,q}^{\left( \alpha ,\beta \right) ,\left( m,n\right)
}\left( u\right) <\infty $.

\item If $w\left( x,y\right) =\left\vert \left( x,y\right) \right\vert
^{-\gamma }$ and $v\left( x,y\right) =\left\vert \left( x,y\right)
\right\vert ^{\delta }$ are power weights on $\mathbb{R}^{m+n}$, then $%
N_{p,q}^{\left( \alpha ,\beta \right) ,\left( m,n\right) }\left( v,w\right)
<\infty $ provided the indices satisfy product conditions corresponding to
those of the Stein-Weiss theorem in $1$-parameter (see below), that in turn
generalize the classical Hardy-Littlewood-Sobolev inequality. In this
example there are two weight pairs $\left\{ \left( V,W\right) ,\left(
V^{\prime },W^{\prime }\right) \right\} $ depending on the indices and
sandwiched in $\left( v,w\right) $ where each weight is an appropriate
product power weight.
\end{enumerate}

We begin by briefly recalling the $1$-parameter weighted theory of
fractional integrals.

\section{$1$-parameter theory}

Define $\Omega _{\alpha }^{m}\left( x\right) =\left\vert x\right\vert
^{\alpha -m}$ and set $I_{\alpha }^{m}g=\Omega _{\alpha }^{m}\ast g$. The
following one weight theorem for fractional integrals is due to Muckenhoupt
and Wheeden.

\begin{theorem}
\label{Muck and Wheed}Let $0<\alpha <m$. Suppose $1<p<q<\infty $ and%
\begin{equation}
\frac{1}{p}-\frac{1}{q}=\frac{\alpha }{m}.  \label{bal MW}
\end{equation}%
Let $w\left( x\right) $ be a nonnegative weight on $\mathbb{R}^{m}$. Then%
\begin{equation*}
\left\{ \int_{\mathbb{R}^{m}}I_{\alpha }^{m}f\left( x\right) ^{q}\ w\left(
x\right) ^{q}\ dx\right\} ^{\frac{1}{q}}\leq N_{p,q}\left( w\right) \left\{
\int_{\mathbb{R}^{m}}f\left( x\right) ^{p}\ w\left( x\right) ^{p}\
dx\right\} ^{\frac{1}{p}}
\end{equation*}%
for all $f\geq 0$ for $N_{p,q}\left( w\right) <\infty $ \emph{if and only if}%
\begin{equation}
A_{p,q}\left( w\right) \equiv \sup_{\text{cubes }I\subset \mathbb{R}%
^{m}}\left( \frac{1}{\left\vert I\right\vert }\int_{I}w\left( x\right) ^{q}\
dx\right) ^{\frac{1}{q}}\left( \frac{1}{\left\vert I\right\vert }%
\int_{I}w\left( x\right) ^{-p^{\prime }}\ dx\right) ^{\frac{1}{p^{\prime }}%
}<\infty .  \label{Apq one}
\end{equation}
\end{theorem}

In fact, assuming only that $1<p,q<\infty $, the balanced condition (\ref%
{bal MW}) is \emph{necessary} for the norm inequality (\ref{Apq one}). See
Theorem \ref{Muck and Wheed sharp} in the Appendix for this.

A two weight analogue for $1<p\leq q<\infty $ was later obtained by Sawyer 
\cite{Saw} that involved testing the norm inequality and its dual over
indicators of cubes times $w^{q}$ and $v^{-p^{\prime }}$ respectively,
namely for all cubes $Q\subset \mathbb{R}^{m}$,%
\begin{eqnarray*}
\left\{ \int_{\mathbb{R}^{m}}I_{\alpha }^{m}\left( \mathbf{1}%
_{Q}v^{-p^{\prime }}\right) \left( x\right) ^{q}\ w\left( x\right) ^{q}\
dx\right\} ^{\frac{1}{q}} &\leq &T_{p,q}^{\alpha ,m}\left( v,w\right) \
\left\vert Q\right\vert _{v^{-p^{\prime }}}^{\frac{1}{p}}, \\
\left\{ \int_{\mathbb{R}^{m}}I_{\alpha }^{m}\left( \mathbf{1}%
_{Q}w^{q}\right) \left( x\right) ^{p^{\prime }}\ v\left( x\right)
^{-p^{\prime }}\ dx\right\} ^{\frac{1}{p^{\prime }}} &\leq &T_{q^{\prime
},p^{\prime }}^{\alpha ,m}\left( \frac{1}{w},\frac{1}{v}\right) \ \left\vert
Q\right\vert _{w^{q}}^{\frac{1}{q^{\prime }}}.
\end{eqnarray*}%
Later yet, it was shown by Sawyer and Wheeden \cite{SaWh}, using an idea of
Kokilashvili and Gabidzashvili \cite{KoGa}, that in the special case $p<q$,
the testing conditions could be replaced with a two-tailed two weight
version of the $A_{p,q}$ condition (\ref{Apq one}):%
\begin{equation}
\sup_{I\subset \mathbb{R}^{m}}\left\vert I\right\vert ^{\frac{\alpha }{m}%
-1}\left( \frac{1}{\left\vert I\right\vert }\int_{I}\left[ \widehat{s}%
_{I}\left( x\right) w\left( x\right) \right] ^{q}\ dx\right) ^{\frac{1}{q}%
}\left( \frac{1}{\left\vert I\right\vert }\int_{I}\left[ \widehat{s}%
_{I}\left( x\right) v\left( x\right) ^{-1}\right] ^{p^{\prime }}\ dx\right)
^{\frac{1}{p^{\prime }}}\equiv \widehat{A}_{p,q}^{\alpha ,m}\left(
v,w\right) <\infty ,  \label{Apq hat}
\end{equation}%
where the tail $\widehat{s}_{I}$ is given by%
\begin{equation*}
\widehat{s}_{I}\left( x\right) \equiv \left( 1+\frac{\left\vert
x-c_{I}\right\vert }{\left\vert I\right\vert ^{\frac{1}{m}}}\right) ^{\alpha
-m},
\end{equation*}%
and $c_{I}$ is the center of the cube $I$. Note that in Sawyer and Wheeden 
\cite{SaWh}, this condition was written in terms of the rescaled tail $%
s_{I}=\left\vert I\right\vert ^{\frac{\alpha }{m}-1}\widehat{s}_{I}$, and it
was also shown that the two-tailed condition $\widehat{A}_{p,q}$ condition (%
\ref{Apq hat}) could be replaced by a pair of corresponding one-tailed
conditions.

\begin{theorem}
\label{Saw and Wheed}Suppose $1<p<q<\infty $ and $0<\alpha <m$. Let $w\left(
x\right) $ and $v\left( x\right) $ be a pair of nonnegative weights on $%
\mathbb{R}^{m}$. Then%
\begin{equation}
\left\{ \int_{\mathbb{R}^{m}}I_{\alpha }^{m}f\left( x\right) ^{q}\ w\left(
x\right) ^{q}\ dx\right\} ^{\frac{1}{q}}\leq N_{p,q}^{\alpha ,m}\left(
v,w\right) \left\{ \int_{\mathbb{R}^{m}}f\left( x\right) ^{p}\ v\left(
x\right) ^{p}\ dx\right\} ^{\frac{1}{p}}  \label{1 par 2 weight}
\end{equation}%
for all $f\geq 0$ \emph{if and only if }the $\widehat{A}_{p,q}^{\alpha ,m}$
condition (\ref{Apq hat}) holds, i.e. $\widehat{A}_{p,q}^{\alpha ,m}\left(
v,w\right) <\infty $. Moreover, the best constant $N_{p,q}^{\alpha ,m}\left(
w,v\right) $ in (\ref{1 par 2 weight}) is comparable to $\widehat{A}%
_{p,q}^{\alpha ,m}\left( w,v\right) $.
\end{theorem}

The special case of power weights $\left\vert x\right\vert ^{\gamma }$ had
been considered much earlier, and culminated in the following 1958 theorem
of Stein and Weiss \cite{StWe2}.

\begin{theorem}
\label{Stein-Weiss}Let $w_{\gamma }\left( x\right) =\left\vert x\right\vert
^{-\gamma }$ and $v_{\delta }\left( x\right) =\left\vert x\right\vert
^{\delta }$ be a pair of nonnegative power weights on $\mathbb{R}^{m}$ with $%
-\infty <\gamma ,\delta <\infty $. Suppose $1<p\leq q<\infty $ and $\alpha
\in \mathbb{R}$ satisfy the strict constraint inequalities,%
\begin{equation}
0<\alpha <m\text{ and }q\gamma <m\text{ and }p^{\prime }\delta <m,
\label{strict constraint}
\end{equation}%
together with the inequality%
\begin{equation*}
\gamma +\delta \geq 0,
\end{equation*}%
and the power weight equality,%
\begin{equation*}
\frac{1}{p}-\frac{1}{q}=\frac{\alpha -\left( \gamma +\delta \right) }{m}.
\end{equation*}%
Then (\ref{1 par 2 weight}) holds, i.e.%
\begin{equation*}
\left\{ \int_{\mathbb{R}^{m}}I_{\alpha }^{m}f\left( x\right) ^{q}\
\left\vert x\right\vert ^{-\gamma q}\ dx\right\} ^{\frac{1}{q}}\leq
N_{p,q}^{\alpha ,m}\left( w_{\gamma },v_{\delta }\right) \left\{ \int_{%
\mathbb{R}^{m}}f\left( x\right) ^{p}\ \left\vert x\right\vert ^{\delta p}\
dx\right\} ^{\frac{1}{p}}.
\end{equation*}
\end{theorem}

The previous theorem cannot be improved when the weights are restricted to
power weights. Indeed, $\alpha >0$ follows from the local integrability of
the kernel, both $q\gamma <m$ and $p^{\prime }\delta <m$ follow from the
local integrability of $w^{q}$ and $v^{-p^{\prime }}$, and then both $\gamma
+\delta \geq 0$ and $\frac{1}{p}-\frac{1}{q}=\frac{\alpha -\left( \gamma
+\delta \right) }{m}$ follow from the finiteness of the Muckenhoupt
condition $A_{p,q}\left( v,w\right) $ using standard arguments (see e.g. the
proof of Theorem \ref{A char mn} below). These conditions then yield 
\begin{equation*}
\alpha =m\left( \frac{1}{p}-\frac{1}{q}\right) +\gamma +\delta <m\left( 
\frac{1}{p}-\frac{1}{q}\right) +\frac{m}{q}+\frac{m}{p^{\prime }}=m.
\end{equation*}%
A routine calculation shows that the aforementioned conditions on the
indices $\alpha ,m,\gamma ,\delta ,p,q$ are precisely those for which the
characteristic $A_{p,q}^{\alpha ,m}\left( w_{\gamma },v_{\delta }\right) $
is finite. Finally, the necessity of the remaining condition $p\leq q$ is an
easy consequence of Maz'ja's characterization \cite{Maz} of the Hardy
inequality for $q<p$. See Theorem \ref{Stein-Weiss sharp} in the Appendix
for this. Altogether, this establishes the succinct conclusion that the
power weight norm inequality holds if and only if $p\leq q$ and the
characteristic is finite.

Next, in the one weight setting, we recall the solution to the `power of the
characteristic' problem for fractional integrals due to Lacey, Moen, Perez
and Torres in \cite{LaMoPeTo}. See the Appendix for a different proof of
this $1$-parameter theorem that reveals the origin of the number $1+\max
\left\{ \frac{p^{\prime }}{q},\frac{q}{p^{\prime }}\right\} $ to be the
optimal exponent in the inequality $\overline{A}_{p,q}\left( w\right) \leq
A_{p,q}\left( w\right) ^{1+\max \left\{ \frac{p^{\prime }}{q},\frac{q}{%
p^{\prime }}\right\} }$ where $\overline{A}_{p,q}\left( w\right) $ is a
one-tailed version of $A_{p,q}\left( w\right) $.

\begin{theorem}
Let $0<\alpha <m$. Suppose $1<p\leq q<\infty $ and%
\begin{equation*}
\frac{1}{p}-\frac{1}{q}=\frac{\alpha }{m}.
\end{equation*}%
Let $w\left( x\right) $ be a nonnegative weight on $\mathbb{R}^{m}$. Then%
\begin{equation*}
\left\{ \int_{\mathbb{R}^{m}}I_{\alpha }^{m}f\left( x\right) ^{q}\ w\left(
x\right) ^{q}\ dx\right\} ^{\frac{1}{q}}\leq N_{p,q}\left( w\right) \left\{
\int_{\mathbb{R}^{m}}f\left( x\right) ^{p}\ w\left( x\right) ^{p}\
dx\right\} ^{\frac{1}{p}}
\end{equation*}%
where%
\begin{equation*}
N_{p,q}\left( w\right) \leq C_{p,q,n}\ A_{p,q}\left( w\right) ^{1+\max
\left\{ \frac{p^{\prime }}{q},\frac{q}{p^{\prime }}\right\} }.
\end{equation*}%
The power $1+\max \left\{ \frac{p^{\prime }}{q},\frac{q}{p^{\prime }}%
\right\} $ is sharp, even when $w$ is restricted to power weights $w\left(
x\right) =\left\vert x\right\vert ^{\gamma }$.
\end{theorem}

Finally we recall the extremely simple proof of the equivalence of the
dyadic characteristic 
\begin{equation*}
\mathbb{A}_{p,q}^{\alpha ,\limfunc{dy}}\left( \sigma ,\omega \right) \equiv
\sup_{Q\subset \mathbb{R}^{n}\ \limfunc{dyadic}}\left\vert Q\right\vert ^{%
\frac{\alpha }{m}+\frac{1}{q}-\frac{1}{p}}\left( \frac{\left\vert
Q\right\vert _{\omega }}{\left\vert Q\right\vert }\right) ^{\frac{1}{q}%
}\left( \frac{\left\vert Q\right\vert _{\sigma }}{\left\vert Q\right\vert }%
\right) ^{\frac{1}{p^{\prime }}},
\end{equation*}%
where the supremum is taken over dyadic cubes, and the weak type $\left(
p,q\right) $ operator norm $\mathbb{N}_{p,q}^{\alpha ,\limfunc{dy}}\left(
\sigma ,\omega \right) $ of the dyadic fractional maximal operator $%
M_{\alpha }^{\limfunc{dy}}$ with respect to $\left( \sigma ,\omega \right) $:%
\begin{equation*}
\mathbb{N}_{p,q}^{\alpha ,\limfunc{dy}}\left( \sigma ,\omega \right) \equiv
\sup_{f\geq 0}\frac{\sup_{\lambda >0}\lambda \left\vert \left\{ M_{\alpha }^{%
\limfunc{dy}}\left( f\sigma \right) >\lambda \right\} \right\vert _{\omega
}^{\frac{1}{q}}}{\left( \int f^{p}d\sigma \right) ^{\frac{1}{p}}}.
\end{equation*}%
This proof illustrates the power of an effective covering lemma for weights,
something that is sorely lacking in the $2$-parameter setting, and accounts
for much of the negative nature of our results (c.f. Example \ref{simple}
below). Recall that the $1$-parameter dyadic fractional maximal operator $%
M_{\alpha }^{\limfunc{dy}}$ acts on a signed measure $\mu $ in $\mathbb{R}%
^{n}$ by%
\begin{equation*}
M_{\alpha }^{\limfunc{dy}}\mu \left( x\right) \equiv \sup_{Q\text{ dyadic}%
}\left\vert Q\right\vert ^{\alpha -n}\int_{Q}d\left\vert \mu \right\vert \ .
\end{equation*}

\begin{lemma}
\label{maximal}Let $\left( \sigma ,\omega \right) $ be a locally finite
weight pair in $\mathbb{R}^{n}$, and let $1<p\leq q<\infty $. Then $%
M_{\alpha }^{\limfunc{dy}}:L^{p}\left( \sigma \right) \rightarrow
L^{q,\infty }\left( \omega \right) $ if and only if $A_{p,q}^{\alpha ,%
\limfunc{dy}}\left( \sigma ,\omega \right) <\infty $.
\end{lemma}

\begin{proof}
Fix $\lambda >0$ and $f\geq 0$ bounded with compact support. Let $\Omega
_{\lambda }\equiv \left\{ x\in \mathbb{R}^{n}:M_{\alpha }^{\limfunc{dy}%
}f\sigma >\lambda \right\} $. Then%
\begin{equation*}
\Omega _{\lambda }=\overset{\cdot }{\dbigcup\limits_{k}}Q_{k}\ ,\ \ \ \ \
\left\vert Q_{k}\right\vert ^{\alpha -n}\int_{Q_{k}}fd\sigma >\lambda ,
\end{equation*}%
and we have%
\begin{eqnarray*}
\left( \lambda \left\vert \Omega _{\lambda }\right\vert _{\omega }^{\frac{1}{%
q}}\right) ^{p} &=&\lambda ^{p}\left( \sum_{k}\left\vert Q_{k}\right\vert
_{\omega }\right) ^{\frac{p}{q}}\leq \lambda ^{p}\sum_{k}\left\vert
Q_{k}\right\vert _{\omega }^{\frac{p}{q}}=\sum_{k}\lambda ^{p}\left\vert
Q_{k}\right\vert _{\omega }^{\frac{p}{q}} \\
&<&\sum_{k}\left( \left\vert Q_{k}\right\vert ^{\alpha
-n}\int_{Q_{k}}fd\sigma \right) ^{p}\left\vert Q_{k}\right\vert _{\omega }^{%
\frac{p}{q}}=\sum_{k}\left( \left\vert Q_{k}\right\vert ^{\alpha
-n}\left\vert Q_{k}\right\vert _{\sigma }^{\frac{1}{p^{\prime }}}\left\vert
Q_{k}\right\vert _{\omega }^{\frac{1}{q}}\right) ^{p}\left\vert
Q_{k}\right\vert _{\sigma }^{1-p}\int_{Q_{k}}fd\sigma \\
&\leq &\sum_{k}\left( A_{p,q}^{\alpha ,\limfunc{dy}}\left( \sigma ,\omega
\right) \right) ^{p}\int_{Q_{k}}f^{p}d\sigma \leq A_{p,q}^{\alpha ,\limfunc{%
dy}}\left( \sigma ,\omega \right) ^{p}\left\Vert f\right\Vert _{L^{p}\left(
\sigma \right) }^{p}\ ,
\end{eqnarray*}%
which gives%
\begin{eqnarray*}
\left\Vert M_{\alpha }^{\limfunc{dy}}f\right\Vert _{L^{q,\infty }\left(
\omega \right) } &=&\sup_{\lambda >0}\lambda \left\vert \Omega _{\lambda
}\right\vert _{\omega }^{\frac{1}{q}}\leq A_{p,q}^{\alpha ,\limfunc{dy}%
}\left( \sigma ,\omega \right) \left\Vert f\right\Vert _{L^{p}\left( \sigma
\right) }\ , \\
\text{for all }f &\geq &0\text{ bounded with compact support.}
\end{eqnarray*}%
The proof of the converse statement is standard, similar to but easier than
that of Lemma \ref{tails}\ below.
\end{proof}

\section{$2$-parameter theory}

Define the product fractional integral $I_{\alpha ,\beta }^{m,n}$ on $%
\mathbb{R}^{m}\times \mathbb{R}^{n}$ by the convolution formula $I_{\alpha
,\beta }^{m,n}f\equiv \Omega _{\alpha ,\beta }^{m,n}\ast f$, where the
convolution kernel $\Omega _{\alpha ,\beta }^{m,n}$ is a product function: 
\begin{equation*}
\Omega _{\alpha ,\beta }^{m,n}\left( x,y\right) =\left\vert x\right\vert
^{\alpha -m}\left\vert y\right\vert ^{\beta -n},\ \ \ \ \ \left( x,y\right)
\in \mathbb{R}^{m}\times \mathbb{R}^{n}.
\end{equation*}%
Let $v\left( x,y\right) $ and $w\left( x,y\right) $ be positive weights on $%
\mathbb{R}^{m}\times \mathbb{R}^{n}$. For $1<p,q<\infty $ and $0<\alpha <m$, 
$0<\beta <n$, we consider the two weight norm inequality for nonnegative
functions $f\left( x,y\right) $:%
\begin{equation}
\left\{ \int_{\mathbb{R}^{m}}\int_{\mathbb{R}^{n}}I_{\alpha ,\beta
}^{m,n}f\left( x,y\right) ^{q}\ w\left( x,y\right) ^{q}\ dxdy\right\} ^{%
\frac{1}{q}}\leq N_{p,q}^{\left( \alpha ,\beta \right) ,\left( m,n\right)
}\left( v,w\right) \left\{ \int_{\mathbb{R}^{m}}\int_{\mathbb{R}^{n}}f\left(
x,y\right) ^{p}\ v\left( x,y\right) ^{p}\ dxdy\right\} ^{\frac{1}{p}}.
\label{2 weight abs cont}
\end{equation}%
If we define absolutely continuous measures $\sigma ,\omega $ by%
\begin{equation}
d\sigma \left( x,y\right) =v\left( x,y\right) ^{-p^{\prime }}dxdy\text{ and }%
d\omega \left( x,y\right) =w\left( x,y\right) ^{q}dxdy,  \label{identif}
\end{equation}%
then the two weight norm inequality (\ref{2 weight abs cont}) is equivalent
to the norm inequality%
\begin{equation}
\left\{ \int_{\mathbb{R}^{m}}\int_{\mathbb{R}^{n}}I_{\alpha ,\beta
}^{m,n}\left( f\sigma \right) \left( x,y\right) ^{q}\ d\omega \left(
x,y\right) \right\} ^{\frac{1}{q}}\leq \mathbb{N}_{p,q}^{\left( \alpha
,\beta \right) ,\left( m,n\right) }\left( \sigma ,\omega \right) \left\{
\int_{\mathbb{R}^{m}}\int_{\mathbb{R}^{n}}f\left( x,y\right) ^{p}\ d\sigma
\left( x,y\right) \right\} ^{\frac{1}{p}},  \label{2 weight arb meas}
\end{equation}%
where now the measure $\sigma $ appears inside the argument of $I_{\alpha
,\beta }^{m,n}$, namely in $I_{\alpha ,\beta }^{m,n}\left( f\sigma \right) $%
. In this form, the norm inequality makes sense for arbitrary locally finite
Borel measures $\sigma ,\omega $ since for nonnegative $f\in L^{p}\left(
\sigma \right) $, the function $f$ is measurable with respect to $\sigma $
and the integral $\int \Omega _{\alpha ,\beta }^{m,n}\left( x-u,y-t\right)
f\left( u,t\right) d\sigma \left( u,t\right) $ exists. Note also that the
best constants $N_{p,q}^{\left( \alpha ,\beta \right) ,\left( m,n\right)
,}\left( v,w\right) $ and $\mathbb{N}_{p,q}^{\left( \alpha ,\beta \right)
,\left( m,n\right) ,}\left( \sigma ,\omega \right) $ coincide under the
standard identifications in (\ref{identif}), and this accounts for the use
of blackboard bold font to differentiate the two best constants.

A necessary condition for (\ref{2 weight arb meas}) to hold is the
finiteness of the corresponding product fractional characteristic $\mathbb{A}%
_{p,q}^{\left( \alpha ,m\right) ,\left( \beta ,n\right) }\left( \sigma
,\omega \right) $ of the weights,%
\begin{eqnarray}
\mathbb{A}_{p,q}^{\left( \alpha ,\beta \right) ,\left( m,n\right) }\left(
\sigma ,\omega \right) &\equiv &\sup_{I\subset \mathbb{R}^{m},\ J\subset 
\mathbb{R}^{n}}\left\vert I\right\vert ^{\frac{\alpha }{m}+\frac{1}{q}-\frac{%
1}{p}}\left\vert J\right\vert ^{\frac{\beta }{n}+\frac{1}{q}-\frac{1}{p}%
}\left( \frac{\left\vert I\times J\right\vert _{\omega }}{\left\vert I\times
J\right\vert }\right) ^{\frac{1}{q}}\left( \frac{\left\vert I\times
J\right\vert _{\sigma }}{\left\vert I\times J\right\vert }\right) ^{\frac{1}{%
p^{\prime }}}  \label{A fraktur} \\
&=&\sup_{I\subset \mathbb{R}^{m},\ J\subset \mathbb{R}^{n}}\left\vert
I\right\vert ^{\frac{\alpha }{m}-1}\left\vert J\right\vert ^{\frac{\beta }{n}%
-1}\left( \diint\limits_{I\times J}d\omega \right) ^{\frac{1}{q}}\left(
\diint\limits_{I\times J}d\sigma \right) ^{\frac{1}{p^{\prime }}},  \notag
\end{eqnarray}%
as well as the finiteness of the larger two-tailed characteristic,%
\begin{eqnarray}
&&\widehat{\mathbb{A}}_{p,q}^{\left( \alpha ,\beta \right) ,\left(
m,n\right) }\left( \sigma ,\omega \right)  \label{A hat fraktur} \\
&\equiv &\sup_{I\times J\subset \mathbb{R}^{m}\times \mathbb{R}%
^{n}}\left\vert I\right\vert ^{\frac{\alpha }{m}-1}\left\vert J\right\vert ^{%
\frac{\beta }{n}-1}\left( \diint\limits_{\mathbb{R}^{m}\times \mathbb{R}^{n}}%
\widehat{s}_{I\times J}\left( x,y\right) ^{q}d\omega \left( x,y\right)
\right) ^{\frac{1}{q}}\left( \diint\limits_{\mathbb{R}^{m}\times \mathbb{R}%
^{n}}\widehat{s}_{I\times J}\left( u,t\right) ^{p^{\prime }}d\sigma \left(
u,t\right) \right) ^{\frac{1}{p^{\prime }}},  \notag
\end{eqnarray}%
where $\widehat{s}_{I\times J}\left( x,y\right) $ is the `tail' defined in (%
\ref{def s hat}) above. When (\ref{identif}) holds, we write $%
A_{p,q}^{\left( \alpha ,\beta \right) ,\left( m,n\right) }\left( v,w\right) =%
\mathbb{A}_{p,q}^{\left( \alpha ,\beta \right) ,\left( m,n\right) }\left(
\sigma ,\omega \right) $ and $\widehat{A}_{p,q}^{\left( \alpha ,\beta
\right) ,\left( m,n\right) }\left( v,w\right) =\widehat{\mathbb{A}}%
_{p,q}^{\left( \alpha ,\beta \right) ,\left( m,n\right) }\left( \sigma
,\omega \right) $, \ and in the one weight case $v=w$ we simply write $%
A_{p,q}^{\alpha ,m}\left( w\right) $ and $\widehat{A}_{p,q}^{\alpha
,m}\left( w\right) $.

\begin{lemma}
\label{tails}For $1<p,q<\infty $ and $0<\alpha \leq m$, $0<\beta \leq n$, we
have 
\begin{equation}
\widehat{\mathbb{A}}_{p,q}^{\left( \alpha ,\beta \right) ,\left( m,n\right)
}\left( \sigma ,\omega \right) \leq \mathbb{N}_{p,q}^{\left( \alpha ,\beta
\right) ,\left( m,n\right) }\left( \sigma ,\omega \right) .  \label{first}
\end{equation}
\end{lemma}

\begin{proof}
To see this, we begin by noting that for any rectangle $I\times J$ we have%
\begin{eqnarray*}
\left\vert I\right\vert ^{\frac{1}{m}}\left\vert x-u\right\vert &\leq
&\left( \left\vert I\right\vert ^{\frac{1}{m}}+\left\vert x-c_{I}\right\vert
\right) \left( \left\vert I\right\vert ^{\frac{1}{m}}+\left\vert
u-c_{I}\right\vert \right) \text{ and }\left\vert J\right\vert ^{\frac{1}{n}%
}\left\vert y-t\right\vert \leq \left( \left\vert J\right\vert ^{\frac{1}{n}%
}+\left\vert y-c_{J}\right\vert \right) \left( \left\vert J\right\vert ^{%
\frac{1}{n}}+\left\vert t-c_{J}\right\vert \right) , \\
\text{i.e. }\left\vert x-u\right\vert &\leq &\left\vert I\right\vert ^{\frac{%
1}{m}}\left( 1+\frac{\left\vert x-c_{I}\right\vert }{\left\vert I\right\vert
^{\frac{1}{m}}}\right) \left( 1+\frac{\left\vert u-c_{I}\right\vert }{%
\left\vert I\right\vert ^{\frac{1}{m}}}\right) \text{ and }\left\vert
y-t\right\vert \leq \left\vert J\right\vert ^{\frac{1}{n}}\left( 1+\frac{%
\left\vert y-c_{J}\right\vert }{\left\vert J\right\vert ^{\frac{1}{n}}}%
\right) \left( 1+\frac{\left\vert t-c_{J}\right\vert }{\left\vert
J\right\vert ^{\frac{1}{n}}}\right) ,
\end{eqnarray*}%
and hence%
\begin{eqnarray*}
&&\left\vert x-u\right\vert ^{\alpha -m}\ \left\vert y-t\right\vert ^{\beta
-n} \\
&\geq &\left\vert I\right\vert ^{\frac{\alpha }{m}-1}\left( 1+\frac{%
\left\vert x-c_{I}\right\vert }{\left\vert I\right\vert ^{\frac{1}{m}}}%
\right) ^{\alpha -m}\left( 1+\frac{\left\vert u-c_{I}\right\vert }{%
\left\vert I\right\vert ^{\frac{1}{m}}}\right) ^{\alpha -m}\ \left\vert
J\right\vert ^{\frac{\beta }{n}-1}\left( 1+\frac{\left\vert
y-c_{J}\right\vert }{\left\vert J\right\vert ^{\frac{1}{n}}}\right) ^{\beta
-n}\left( 1+\frac{\left\vert t-c_{J}\right\vert }{\left\vert J\right\vert ^{%
\frac{1}{n}}}\right) ^{\beta -n} \\
&=&\left\vert I\right\vert ^{\frac{\alpha }{m}-1}\left\vert J\right\vert ^{%
\frac{\beta }{n}-1}\widehat{s}_{I\times J}\left( x,y\right) \widehat{s}%
_{I\times J}\left( u,t\right) \ .
\end{eqnarray*}%
Thus for $R>0$ and $f_{R}\left( u,t\right) \equiv \mathbf{1}_{B\left(
0,R\right) \times B\left( 0,R\right) }\left( u,t\right) \widehat{s}%
_{Q}\left( u,t\right) ^{p^{\prime }-1}$, we have%
\begin{eqnarray*}
I_{\alpha ,\beta }^{m,n}\left( f_{R}\sigma \right) \left( x,y\right)
&=&\diint\limits_{B\left( 0,R\right) \times B\left( 0,R\right) }\left\vert
x-u\right\vert ^{\alpha -n}\left\vert y-t\right\vert ^{\beta -n}\widehat{s}%
_{I\times J}\left( u,t\right) ^{p^{\prime }-1}d\sigma \left( u,t\right) \\
&\geq &\diint\limits_{B\left( 0,R\right) \times B\left( 0,R\right)
}\left\vert I\right\vert ^{\frac{\alpha }{m}-1}\left\vert J\right\vert ^{%
\frac{\beta }{n}-1}\widehat{s}_{I\times J}\left( x,y\right) \widehat{s}%
_{I\times J}\left( u,t\right) \widehat{s}_{I\times J}\left( u,t\right)
^{p^{\prime }-1}d\sigma \left( u,t\right) \\
&=&\left\vert I\right\vert ^{\frac{\alpha }{m}-1}\left\vert J\right\vert ^{%
\frac{\beta }{n}-1}\widehat{s}_{I\times J}\left( x,y\right)
\diint\limits_{B\left( 0,R\right) \times B\left( 0,R\right) }\widehat{s}%
_{I\times J}\left( u,t\right) ^{p^{\prime }}d\sigma \left( u,t\right) .
\end{eqnarray*}%
Substituting this into the norm inequality (\ref{2 weight arb meas}) gives%
\begin{eqnarray*}
&&\left\vert I\right\vert ^{\frac{\alpha }{m}-1}\left\vert J\right\vert ^{%
\frac{\beta }{n}-1}\left( \diint\limits_{B\left( 0,R\right) \times B\left(
0,R\right) }\widehat{s}_{I\times J}\left( u,t\right) ^{p^{\prime }}d\sigma
\left( u,t\right) \right) \left( \diint\limits_{\mathbb{R}^{m}\times \mathbb{%
R}^{n}}\widehat{s}_{I\times J}\left( x,y\right) ^{q}d\omega \left(
x,y\right) \right) ^{\frac{1}{q}} \\
&\leq &\left\{ \diint\limits_{\mathbb{R}^{m}\times \mathbb{R}^{n}}I_{\alpha
,\beta }^{m,n}\left( f_{R}\right) \sigma \left( x,y\right) ^{q}d\omega
\left( x,y\right) \right\} ^{\frac{1}{q}} \\
&\leq &\mathbb{N}_{p,q}^{\left( \alpha ,\beta \right) ,\left( m,n\right)
}\left( \sigma ,\omega \right) \left\{ \diint\limits_{\mathbb{R}^{m}\times 
\mathbb{R}^{n}}f_{R}\left( u,t\right) ^{p}d\sigma \left( u,t\right) \right\}
^{\frac{1}{p}} \\
&=&\mathbb{N}_{p,q}^{\left( \alpha ,\beta \right) ,\left( m,n\right) }\left(
\sigma ,\omega \right) \left( \diint\limits_{B\left( 0,R\right) \times
B\left( 0,R\right) }\widehat{s}_{I\times J}\left( u,t\right) ^{p^{\prime
}}d\sigma \left( u,t\right) \right) ^{\frac{1}{p}},
\end{eqnarray*}%
and upon dividing through by $\left( \diint\limits_{B\left( 0,R\right)
\times B\left( 0,R\right) }\widehat{s}_{I\times J}\left( u,t\right)
^{p^{\prime }}d\sigma \left( u,t\right) \right) ^{\frac{1}{p}}$, we obtain%
\begin{eqnarray*}
&&\left\vert I\right\vert ^{\frac{\alpha }{m}-1}\left\vert J\right\vert ^{%
\frac{\beta }{n}-1}\left( \diint\limits_{B\left( 0,R\right) \times B\left(
0,R\right) }\widehat{s}_{I\times J}\left( u,t\right) ^{p^{\prime }}d\sigma
\left( u,t\right) \right) ^{\frac{1}{p^{\prime }}}\left( \diint\limits_{%
\mathbb{R}^{m}\times \mathbb{R}^{n}}\widehat{s}_{I\times J}\left( x,y\right)
^{q}d\omega \left( x,y\right) \right) ^{\frac{1}{q}} \\
&&\ \ \ \ \ \ \ \ \ \ \ \ \ \ \ \ \ \ \ \ \ \ \ \ \ \ \ \ \ \ \leq \mathbb{N}%
_{p,q}^{\left( \alpha ,\beta \right) ,\left( m,n\right) }\left( \sigma
,\omega \right) ,\ \ \ \ \ \text{for all }R>0\text{ and all rectangles }%
I\times J.
\end{eqnarray*}%
Now take the supremum over all $R>0$ and all rectangles $I\times J$ to get (%
\ref{first}).
\end{proof}

\begin{remark}
We have the `duality' identities $\mathbb{N}_{p,q}^{\left( \alpha ,\beta
\right) ,\left( m,n\right) }\left( \sigma ,\omega \right) =\mathbb{N}%
_{q^{\prime },p^{\prime }}^{\left( \alpha ,\beta \right) ,\left( m,n\right)
}\left( \omega ,\sigma \right) $ and $\widehat{\mathbb{A}}_{p,q}^{\left(
\alpha ,\beta \right) ,\left( m,n\right) }\left( \sigma ,\omega \right) =%
\widehat{\mathbb{A}}_{q^{\prime },p^{\prime }}^{\left( \alpha ,\beta \right)
,\left( m,n\right) }\left( \omega ,\sigma \right) $.
\end{remark}

\begin{remark}
\label{disjoint support}Since $\widehat{s}_{I\times J}\approx 1$ on $I\times
J$, we have the inequality $\mathbb{A}_{p,q}^{\left( \alpha ,\beta \right)
,\left( m,n\right) }\left( \sigma ,\omega \right) \lesssim \widehat{\mathbb{A%
}}_{q^{\prime },p^{\prime }}^{\left( \alpha ,\beta \right) ,\left(
m,n\right) }\left( \omega ,\sigma \right) $. In particular, we see that in
the case 
\begin{equation*}
\frac{1}{p}-\frac{1}{q}>\min \left\{ \frac{\alpha }{m},\frac{\beta }{n}%
\right\} ,
\end{equation*}%
say $\frac{1}{p}-\frac{1}{q}-\frac{\alpha }{m}=\varepsilon >0$, we have $%
\frac{\alpha }{m}-1=-\varepsilon -\frac{1}{p^{\prime }}-\frac{1}{q}$, and so%
\begin{equation*}
\left\vert I\times J\right\vert ^{-\varepsilon }\left( \frac{1}{\left\vert
I\times J\right\vert }\diint\limits_{I\times J}d\sigma \right) ^{\frac{1}{%
p^{\prime }}}\left( \frac{1}{\left\vert I\times J\right\vert }%
\diint\limits_{I\times J}d\omega \right) ^{\frac{1}{q}}\leq \mathbb{A}%
_{p,q}^{\left( \alpha ,\beta \right) ,\left( m,n\right) }\left( \sigma
,\omega \right) \lesssim \widehat{\mathbb{A}}_{p,q}^{\left( \alpha ,\beta
\right) ,\left( m,n\right) }\left( \sigma ,\omega \right) ,
\end{equation*}%
for all rectangles $I\times J$. Thus the finiteness of $\widehat{\mathbb{A}}%
_{p,q}^{\left( \alpha ,\beta \right) ,\left( m,n\right) }\left( \sigma
,\omega \right) $ implies that the measures $\sigma $ and $\omega $ are
carried by \emph{disjoint} sets. We shall not have much more to say
regarding this case.
\end{remark}

\section{Statements of problems and theorems, and simple proofs}

The $2$-parameter questions we investigate in this paper are these.

\begin{enumerate}
\item Is the finiteness of the one weight characteristic $A_{p,q}\left(
w\right) $ in (\ref{Apq}) below sufficient for the one weight norm
inequality (\ref{2 weight abs cont}) with $w=v$ when the indices are
balanced, $\frac{1}{p}-\frac{1}{q}=\frac{\alpha }{m}=\frac{\beta }{n}$, and
if so what is the dependence of the operator norm $N_{p,q}\left( w\right) $
on the characteristic $A_{p,q}\left( w\right) $?

\item If the operator norm $\mathbb{N}_{p,q}^{\left( \alpha ,\beta \right)
,\left( m,n\right) }\left( \sigma ,\omega \right) $ fails to be controlled
by the characteristic $\mathbb{A}_{p,q}^{\left( \alpha ,\beta \right)
,\left( m,n\right) }\left( \sigma ,\omega \right) $, what additional side
conditions on the weights $\sigma ,\omega $ are needed for finiteness of the
characteristic $\mathbb{A}_{p,q}^{\left( \alpha ,\beta \right) ,\left(
m,n\right) }\left( \sigma ,\omega \right) $ to imply the norm inequality (%
\ref{2 weight arb meas})?
\end{enumerate}

\subsection{A one weight theorem}

The special `one weight' case of (\ref{2 weight abs cont}), namely when $v=w$%
, is equivalent to finiteness of the product characteristic, and we can
calculate the optimal power of the characteristic.

\begin{theorem}
\label{one weight product}Let $\alpha ,\beta >0$, and suppose $1<p,q<\infty $%
. Let $w\left( x,y\right) $ be a nonnegative weight on $\mathbb{R}^{m}\times 
\mathbb{R}^{n}$. Then the norm inequality 
\begin{equation}
\left\{ \int_{\mathbb{R}^{m}}\int_{\mathbb{R}^{n}}I_{\alpha ,\beta
}^{m,n}f\left( x,y\right) ^{q}\ w\left( x,y\right) ^{q}\ dxdy\right\} ^{%
\frac{1}{q}}\leq N_{p,q}\left( w\right) \left\{ \int_{\mathbb{R}^{m}}\int_{%
\mathbb{R}^{n}}f\left( x,y\right) ^{p}\ w\left( x,y\right) ^{p}\
dxdy\right\} ^{\frac{1}{p}}  \label{norm one weight}
\end{equation}%
holds for all $f\geq 0$ \emph{if and only if }both%
\begin{equation}
\frac{1}{p}-\frac{1}{q}=\frac{\alpha }{m}=\frac{\beta }{n},
\label{bala and diag}
\end{equation}%
and%
\begin{equation}
A_{p,q}\left( w\right) \equiv \sup_{I\subset \mathbb{R}^{m},\ J\subset 
\mathbb{R}^{n}}\left( \frac{1}{\left\vert I\right\vert \left\vert
J\right\vert }\int \int_{I\times J}w\left( x,y\right) ^{q}\ dxdy\right) ^{%
\frac{1}{q}}\left( \frac{1}{\left\vert I\right\vert \left\vert J\right\vert }%
\int \int_{I\times J}w\left( x,y\right) ^{-p^{\prime }}\ dxdy\right) ^{\frac{%
1}{p^{\prime }}}<\infty ,  \label{Apq}
\end{equation}%
\emph{if and only if} both (\ref{bala and diag}) and%
\begin{eqnarray}
\sup_{y\in \mathbb{R}^{n}}\left\{ \sup_{I\subset \mathbb{R}^{m}}\left( \frac{%
1}{\left\vert I\right\vert }\int_{I}w^{y}\left( x\right) ^{q}\ dx\right) ^{%
\frac{1}{q}}\left( \frac{1}{\left\vert I\right\vert }\int_{I}w^{y}\left(
x\right) ^{-p^{\prime }}\ dx\right) ^{\frac{1}{p^{\prime }}}\right\} &\leq
&A_{p,q}\left( w\right) ,  \label{iter Apq} \\
\sup_{x\in \mathbb{R}^{m}}\left\{ \sup_{J\subset \mathbb{R}^{n}}\left( \frac{%
1}{\left\vert J\right\vert }\int_{J}w_{x}\left( y\right) ^{q}\ dy\right) ^{%
\frac{1}{q}}\left( \frac{1}{\left\vert J\right\vert }\int_{J}w_{x}\left(
y\right) ^{-p^{\prime }}\ dy\right) ^{\frac{1}{p^{\prime }}}\right\} &\leq
&A_{p,q}\left( w\right) .  \notag
\end{eqnarray}%
Moreover 
\begin{equation}
N_{p,q}\left( w\right) \lesssim A_{p,q}\left( w\right) ^{2+2\max \left\{ 
\frac{p^{\prime }}{q},\frac{q}{p^{\prime }}\right\} },
\label{characteristic power}
\end{equation}%
and the exponent $2+2\max \left\{ \frac{p^{\prime }}{q},\frac{q}{p^{\prime }}%
\right\} $ is best possible.
\end{theorem}

There is a substitute for the case $\alpha =\beta =0$ due to R. Fefferman 
\cite[see page 82]{Fef}, namely that the strong maximal function 
\begin{equation}
\mathcal{M}f\left( x,y\right) \equiv \sup_{I,J}\frac{1}{\left\vert
I\right\vert \left\vert J\right\vert }\int_{I}\int_{J}\left\vert f\left(
x,y\right) \right\vert dxdy,  \label{def strong max}
\end{equation}%
is bounded on the weighted space $L^{p}\left( w^{p}\right) $ if and only if $%
A_{p}\left( w\right) =A_{p,p}\left( w\right) $ is finite. This is proved in 
\cite[see page 82]{Fef} as an application of a rectangle covering lemma, but
can also be obtained using iteration, specifically from Theorem \ref{iter op}
below as $\mathcal{M}$ is dominated by the iterated operators in (\ref%
{strong max}). Another substitute for the case $\alpha =\beta =0$ is the
boundedness of the double Hilbert transform%
\begin{equation*}
\mathcal{H}f\left( x,y\right) \equiv \int_{I}\int_{J}\frac{1}{x-u}\frac{1}{%
y-t}f\left( u,t\right) dudt,
\end{equation*}%
on $L^{p}\left( w^{p}\right) $ if and only if $A_{p}\left( w\right) $ is
finite. This result, along with corresponding results for more general
product Calder\'{o}n-Zygmund operators, can be easily proved using Theorem %
\ref{iter op} below, and are left for the reader.

\subsubsection{Application to a two weight half-balanced norm inequality}

Here is a two weight consequence of Theorem \ref{one weight product} when
the indices satisfy the half-balanced condition.

\begin{theorem}
\label{original A1}Suppose that $1<p,q<\infty $ and $0<\frac{\alpha }{m},%
\frac{\beta }{n}<1$ satisfy the half-balanced condition%
\begin{equation}
\frac{1}{p}-\frac{1}{q}=\min \left\{ \frac{\alpha }{m},\frac{\beta }{n}%
\right\} ,  \label{ba and di}
\end{equation}%
If one of the weights $w^{q}$ or $v^{-p^{\prime }}$ is in the product $%
A_{1}\times A_{1}$ class, then the following norm inequality holds,%
\begin{eqnarray}
&&\left\{ \int_{\mathbb{R}^{m}}\int_{\mathbb{R}^{n}}I_{\alpha ,\beta
}^{m,n}f\left( x,y\right) ^{q}\ w\left( x,y\right) ^{q}\ dxdy\right\} ^{%
\frac{1}{q}}  \label{A1 ineq} \\
&\leq &C_{p,q}A_{p,q}^{\left( \alpha ,\beta \right) ,\left( m,n\right)
}\left( v,w\right) \left\{ \int_{\mathbb{R}^{m}}\int_{\mathbb{R}^{n}}f\left(
x,y\right) ^{p}\ v\left( x,y\right) ^{p}\ dxdy\right\} ^{\frac{1}{p}}, 
\notag
\end{eqnarray}%
where $C_{p,q}$ depends also on either $\left\Vert w^{q}\right\Vert
_{A_{1}\times A_{1}}$ or $\left\Vert v^{-p^{\prime }}\right\Vert
_{A_{1}\times A_{1}}$.
\end{theorem}

This theorem shows that in the half-balanced case, the simple characteristic 
$A_{p,q}^{\left( \alpha ,\beta \right) ,\left( m,n\right) }$ controls the
two weight norm inequality (\ref{2 weight abs cont}) under either of the
side conditions (\textbf{i}) $w^{q}\in A_{1}\times A_{1}$ or (\textbf{ii}) $%
v^{-p^{\prime }}$ $\in A_{1}\times A_{1}$. This is in stark constrast to the
strictly subbalanced case $\frac{1}{p}-\frac{1}{q}<\min \left\{ \frac{\alpha 
}{m},\frac{\beta }{n}\right\} $ where not even both side conditions $%
v^{-p^{\prime }},w^{q}\in A_{1}\times A_{1}$ are sufficient for control of
the norm inequality by the two-tailed characteristc $\widehat{A}%
_{p,q}^{\left( \alpha ,\beta \right) ,\left( m,n\right) }\left( v,w\right) $%
. On the other hand, we cannot replace the side condition that $w^{q}\in
A_{1}\times A_{1}$ or $v^{-p^{\prime }}\in A_{1}\times A_{1}$ in the above
theorem with the smaller side condition $w^{q}\in A_{\infty }\times
A_{\infty }$ or $v^{-p^{\prime }}\in A_{\infty }\times A_{\infty }$, or even
with $w^{q}\in A_{q}\times A_{q}$ or $v^{-p^{\prime }}\in A_{p^{\prime
}}\times A_{p^{\prime }}$, as evidenced by the following family of examples,
whose properties we prove below.

\begin{example}
\label{half example}Let$1<p<q<\infty $ and $\frac{\alpha }{m}=\frac{1}{p}-%
\frac{1}{q}$. Set 
\begin{equation*}
v_{1}\left( y\right) =\left\vert y\right\vert ^{-\frac{m}{q}}\text{ and }%
w_{1}\left( x\right) =\left( 1+\left\vert x\right\vert \right) ^{-m}\text{
for }x,y\in \mathbb{R}^{m}.
\end{equation*}%
Then $v_{1}^{-p^{\prime }}\left( \mathbb{R}^{m}\right) \in A_{p^{\prime }}$
and $A_{p,q}^{\alpha ,m}\left( v_{1},w_{1}\right) <\infty =\overline{A}%
_{p,q}^{\alpha ,m}\left( v_{1},w_{1}\right) $. Now let $\left(
v_{2},w_{2}\right) $ be any weight pair in $\mathbb{R}^{n}$ satisfying both $%
v_{2}^{-p^{\prime }}\in A_{p^{\prime }}\left( \mathbb{R}^{n}\right) $ and $%
0<A_{p,q}^{\beta ,n}\left( v_{2},w_{2}\right) <\infty $. Then\ with $v\left(
y_{1},y_{2}\right) =v_{1}\left( y_{1}\right) v_{2}\left( y_{2}\right) $ and $%
w\left( x_{1},x_{2}\right) =w_{1}\left( x_{1}\right) w_{2}\left(
x_{2}\right) $, we have $v^{-p^{\prime }}\in \left( A_{p^{\prime }}\times
A_{p^{\prime }}\right) \left( \mathbb{R}^{m+n}\right) $ and%
\begin{eqnarray*}
A_{p,q}^{\left( \alpha ,\beta \right) ,\left( m,n\right) }\left( v,w\right)
&=&A_{p,q}^{\alpha ,m}\left( v_{1},w_{1}\right) \ A_{p,q}^{\beta ,n}\left(
v_{2},w_{2}\right) <\infty , \\
\overline{A}_{p,q}^{\left( \alpha ,\beta \right) ,\left( m,n\right) }\left(
v,w\right) &=&\overline{A}_{p,q}^{\alpha ,m}\left( v_{1},w_{1}\right) \ 
\overline{A}_{p,q}^{\beta ,n}\left( v_{2},w_{2}\right) =\infty .
\end{eqnarray*}%
In particular, the two weight norm inequality (\ref{2 weight abs cont})
fails to hold for the weight pair $\left( v,w\right) $ despite the fact that 
$v_{p^{\prime }}^{-p^{\prime }}$ belongs to the product $A_{p^{\prime
}}\times A_{p^{\prime }}$ class, and $A_{p,q}^{\left( \alpha ,\beta \right)
,\left( m,n\right) }\left( v,w\right) $ is finite.
\end{example}

\subsection{Two weight theorems - counterexamples}

Without any side conditions at all on the weights, the characteristic $%
\mathbb{A}_{p,q}^{\left( \alpha ,\beta \right) ,\left( m,n\right) }\left(
\sigma ,\omega \right) $ \emph{never} controls the operator norm $\mathbb{N}%
_{p,q}^{\left( \alpha ,\beta \right) ,\left( m,n\right) }\left( \sigma
,\omega \right) $ for the product fractional integral, and not even for the
smaller product dyadic fractional maximal operator $\mathcal{M}_{\alpha
,\beta }^{\limfunc{dy}}$ defined on a signed measure $\mu $ by%
\begin{equation*}
\mathcal{M}_{\alpha ,\beta }^{\limfunc{dy}}\mu \left( x,y\right) \equiv \sup 
_{\substack{ R=I\times J\text{ dyadic}  \\ \left( x,y\right) \in R}}%
\left\vert I\right\vert ^{\frac{\alpha }{m}-1}\left\vert J\right\vert ^{%
\frac{\beta }{n}-1}\diint\limits_{I\times J}d\left\vert \mu \right\vert .
\end{equation*}%
(note that $\mathcal{M}_{\alpha ,\beta }^{\limfunc{dy}}\mu \leq I_{\alpha
,\beta }^{m,n}\mu $ when $\mu $ is positive). Compare this to Lemma \ref%
{maximal} in the $1$-parameter setting.

\begin{example}
\label{simple}Let $0<\alpha ,\beta <1$ and $1<p,q<\infty $. Given $0<\rho
<\infty $, define a weight pair $\left( \sigma ,\omega _{\rho }\right) $ in
the plane $\mathbb{R}^{2}$ by%
\begin{equation}
\sigma \equiv \delta _{\left( 0,0\right) }\text{ and }\omega _{\rho }\equiv
\sum_{P\in \mathcal{P}}\delta _{P}\text{ where }\mathcal{P}=\left\{ \left(
2^{k},2^{-\rho k}\right) \right\} _{k=1}^{\infty }\ .  \label{def weights}
\end{equation}%
If $R=I\times J$ is a rectangle in the plane $\mathbb{R}\times \mathbb{R}$
with sides parallel to the axes that contains $\left( 0,0\right) $ and
satisfies $R\cap \mathcal{P}=\left\{ \left( 2^{k},2^{-\rho k}\right)
\right\} _{k=L}^{L+N}$ for some $L\geq 1$ and $N\geq 0$, then it follows
that $\diint\limits_{I\times J}d\omega _{\rho }\approx
\sum_{k=L}^{L+N}1\approx N+1$, and%
\begin{equation*}
\left\vert I\right\vert ^{\alpha -1}\left\vert J\right\vert ^{\beta
-1}\left( \diint\limits_{I\times J}d\omega _{\rho }\right) ^{\frac{1}{q}%
}\left( \diint\limits_{I\times J}d\sigma \right) ^{\frac{1}{p^{\prime }}%
}\lesssim \left( 2^{L+N}\right) ^{\alpha -1}\left( 2^{-\rho \left(
L+1\right) }\right) ^{\beta -1}\left( N+1\right) ^{\frac{1}{q}}1^{\frac{1}{%
p^{\prime }}}.
\end{equation*}%
If $\rho \leq \frac{1-\alpha }{1-\beta }$, then this latter expression is $%
2^{\rho \left( 1-\beta \right) }$ times $2^{-L\left[ \left( 1-\alpha \right)
-\rho \left( 1-\beta \right) \right] }\ 2^{N\left( \alpha -1\right) }\left(
N+1\right) ^{\frac{1}{q}}$, which is uniformly bounded, and hence $\mathbb{A}%
_{p,q}^{\left( \alpha ,\beta \right) ,\left( 1,1\right) }\left( \delta
_{\left( 0,0\right) },\omega \right) $ is finite. On the other hand, the
function $f\left( x,y\right) \equiv 1$ satisfies $f\in L^{p}\left( \sigma
\right) $, while the strong dyadic fractional maximal function satisfies%
\begin{equation*}
\mathcal{M}_{\alpha ,\beta }^{\limfunc{dy}}f\sigma \left( 2^{N},2^{-\rho
N}\right) \geq \left\vert \left( 0,2^{N}\right) \right\vert ^{\alpha
-1}\left\vert \left( 0,2^{-\rho N}\right) \right\vert ^{\beta
-1}\diint\limits_{\left[ 0,2^{N}\right) \times \left[ 0,2^{-\rho N}\right)
}fd\sigma =2^{N\left( \alpha -1-\rho \left( \beta -1\right) \right) }\geq 1,
\end{equation*}%
for all $N\geq 1$ provided $\rho \geq \frac{1-\alpha }{1-\beta }$, so that
the weak type operator norm of $\mathcal{M}_{\alpha ,\beta }^{\limfunc{dy}}f$
is infinite if $\rho \geq \frac{1-\alpha }{1-\beta }$:%
\begin{eqnarray*}
&&\left\Vert \mathcal{M}_{\alpha ,\beta }^{\limfunc{dy}}\right\Vert
_{L^{q,\infty }\left( \omega _{\rho }\right) }=\sup_{\substack{ \lambda >0 
\\ \left\Vert f\right\Vert _{L^{p}\left( \sigma \right) }=1}}\lambda
\left\vert \left\{ x\in \mathbb{R}^{n}:\mathcal{M}_{\alpha ,\beta }^{%
\limfunc{dy}}f\sigma >\lambda \right\} \right\vert _{\omega _{\rho }}^{\frac{%
1}{q}} \\
&\geq &\frac{1}{2}\left\vert \left\{ x\in \mathbb{R}^{n}:\mathcal{M}_{\alpha
,\beta }^{\limfunc{dy}}f\sigma >\frac{1}{2}\right\} \right\vert _{\omega
_{\rho }}^{\frac{1}{q}}=\frac{1}{2}\left( \sum_{N=1}^{\infty }\left\vert
\left( 2^{N},2^{-\rho N}\right) \right\vert _{\omega _{\rho }}\right) ^{%
\frac{1}{q}}\approx \left( \sum_{N=1}^{\infty }1\right) ^{\frac{1}{q}%
}=\infty .
\end{eqnarray*}
\end{example}

\begin{remark}
It is easy to verify that when $\sigma =\delta _{\left( 0,0\right) }$, the
operator norm $\mathbb{N}_{p,q}^{\left( \alpha ,\beta \right) ,\left(
m,n\right) }\left( \delta _{\left( 0,0\right) },\omega \right) $ is in fact
equivalent to the \emph{two-tailed} characteristic $\widehat{\mathbb{A}}%
_{p,q}^{\left( \alpha ,\beta \right) ,\left( m,n\right) }\left( \delta
_{\left( 0,0\right) },\omega \right) $ for all measures $\omega $. Indeed,
as shown in the Appendix below, $\mathbb{N}_{p,q}^{\left( \alpha ,\beta
\right) ,\left( m,n\right) }\left( \delta _{\left( 0,0\right) },\omega
\right) $ and $\widehat{\mathbb{A}}_{p,q}^{\left( \alpha ,\beta \right)
,\left( m,n\right) }\left( \delta _{\left( 0,0\right) },\omega \right) $ are
each equivalent to%
\begin{equation*}
\left\{ \int_{\mathbb{R}^{m}}\int_{\mathbb{R}^{n}}\left\vert x\right\vert
^{\left( \alpha -m\right) q}\left\vert y\right\vert ^{\left( \beta -n\right)
q}\ d\omega \left( x,y\right) \right\} ^{\frac{1}{q}}.
\end{equation*}%
Thus the previous example simply produces a weight pair $\left( \sigma
,\omega \right) $ for which $\mathbb{A}_{p,q}^{\left( \alpha ,\beta \right)
,\left( m,n\right) }\left( \sigma ,\omega \right) <\infty $ and $\widehat{%
\mathbb{A}}_{p,q}^{\left( \alpha ,\beta \right) ,\left( m,n\right) }\left(
\sigma ,\omega \right) =\infty $.
\end{remark}

\subsection{Two weight theorems - power weights}

We will see below that in the special case that \emph{both} weights are 
\emph{product} weights, i.e. $w\left( x,y\right) =w_{1}\left( x\right)
w_{2}\left( y\right) $ and $v\left( x,y\right) =v_{1}\left( x\right)
v_{2}\left( y\right) $, then the one parameter theory carries over fairly
easily to the multiparameter setting. Despite the negative nature of the
previous theorem, the 1958 result of Stein and Weiss on power weights \emph{%
does} carry over to the multiparameter setting using the sandwiching
technique of Lemma \ref{sandwich} - where here the power weights are \emph{%
not} product power weights (the theory for product power weights reduces
trivially to that of the $1$-parameter setting).

At this point it is instructive to observe that the kernel of the $1$%
-parameter fractional integral $I_{\alpha +\beta }^{m+n}$ is trivially
dominated by the kernel of the $2$-parameter fractional integral $I_{\alpha
,\beta }^{m,n}$, and hence the corresponding $1$-parameter conditions in
Theorem \ref{Stein-Weiss} are necessary for boundedness of $I_{\alpha ,\beta
}^{m,n}$ - namely $p\leq q$ and 
\begin{eqnarray*}
&&0<\alpha +\beta <m+n\text{ and }\gamma q<m+n\text{ and }\delta p^{\prime
}<m+n, \\
&&\ \ \ \ \ \ \ \ \ \ \ \ \ \ \ \ \ \ \ \ \gamma +\delta \geq 0, \\
&&\ \ \ \ \ \ \ \ \ \ \frac{1}{p}-\frac{1}{q}=\frac{\alpha +\beta -\left(
\gamma +\delta \right) }{m+n},
\end{eqnarray*}%
where the displayed conditions are equivalent to finiteness of the $1$%
-parameter Muckenhoupt characteristic $A_{p,q}^{\alpha +\beta ,m+n}\left(
v_{\delta },w_{\gamma }\right) $. The boundedness of $I_{\alpha ,\beta
}^{m,n}$ is instead given by finiteness of the rectangle Muckenhoupt
characteristic $A_{p,q}^{\left( \alpha ,\beta \right) ,\left( m,n\right)
}\left( v_{\delta },w_{\gamma }\right) $. Here now is our extension of the
classical Stein-Weiss theorem to the product setting.

\begin{theorem}
\label{power weights}Let $w_{\gamma }\left( x,y\right) =\left\vert \left(
x,y\right) \right\vert ^{-\gamma }=\left( \left\vert x\right\vert
^{2}+\left\vert y\right\vert ^{2}\right) ^{-\frac{\gamma }{2}}$ and $%
v_{\delta }\left( x,y\right) =\left\vert \left( x,y\right) \right\vert
^{\delta }=\left( \left\vert x\right\vert ^{2}+\left\vert y\right\vert
^{2}\right) ^{\frac{\delta }{2}}$ be a pair of nonnegative power weights on $%
\mathbb{R}^{m+n}=\mathbb{R}^{m}\times \mathbb{R}^{n}$ with $-\infty <\gamma
,\delta <\infty $. Let $1<p,q<\infty $ and $-\infty <\alpha ,\beta <\infty $%
. Then the following three conditions are equivalent:

\begin{enumerate}
\item the two weight norm inequality (\ref{2 weight abs cont}) holds, i.e.%
\begin{equation}
\left\{ \int_{\mathbb{R}^{m}\times \mathbb{R}^{n}}I_{\alpha ,\beta
}^{m,n}f\left( x,y\right) ^{q}\ \left\vert \left( x,y\right) \right\vert
^{-\gamma q}\ dxdy\right\} ^{\frac{1}{q}}\leq N_{p,q}^{\left( \alpha ,\beta
\right) ,\left( m,n\right) }\left( v_{\delta },w_{\gamma }\right) \left\{
\int_{\mathbb{R}^{m}\times \mathbb{R}^{n}}f\left( x,y\right) ^{p}\
\left\vert \left( x,y\right) \right\vert ^{\delta p}\ dxdy\right\} ^{\frac{1%
}{p}},  \label{pwni}
\end{equation}%
for all $f\geq 0$

\item the indices $p,q$ satisfy%
\begin{equation}
p\leq q,  \label{p at most q}
\end{equation}%
and the Muckenhoupt characteristic $A_{p,q}^{\left( \alpha ,\beta \right)
,\left( m,n\right) }\left( v_{\delta },w_{\gamma }\right) $ is finite, i.e.%
\begin{equation}
\sup_{I\subset \mathbb{R}^{m},\ J\subset \mathbb{R}^{n}}\left\vert
I\right\vert ^{\frac{\alpha }{m}-1}\left\vert J\right\vert ^{\frac{\beta }{n}%
-1}\left( \int \int_{I\times J}\left\vert \left( x,y\right) \right\vert
^{-\gamma q}\ dxdy\right) ^{\frac{1}{q}}\left( \int \int_{I\times
J}\left\vert \left( u,t\right) \right\vert ^{-\delta p^{\prime }}\
dudt\right) ^{\frac{1}{p^{\prime }}}<\infty ,  \label{finiteness of A}
\end{equation}

\item the indices satisfy\ (\ref{p at most q}) and%
\begin{equation}
{\frac{1}{p}}-{\frac{1}{q}}+{\frac{{\gamma +\delta }}{m+n}}={\frac{\alpha
+\beta }{m+n},}  \label{Formula}
\end{equation}%
and%
\begin{equation}
{\gamma +\delta }\geq 0,  \label{gamma+delta}
\end{equation}%
and%
\begin{eqnarray}
\beta -\frac{n}{p} &<&\delta \text{ and }\alpha -\frac{m}{p}<\delta \text{
when }\gamma \geq 0\geq \delta ,  \label{star'} \\
\beta -\frac{n}{q^{\prime }} &<&\gamma \text{ and }\alpha -\frac{m}{%
q^{\prime }}<\gamma \text{ when }\delta \geq 0\geq \gamma .  \notag
\end{eqnarray}
\end{enumerate}
\end{theorem}

In the absence of\ (\ref{p at most q}), a charactertization in terms of
indices of the finiteness of the Muckenhoupt characteristic $A_{p,q}^{\left(
\alpha ,\beta \right) ,\left( m,n\right) }\left( v_{\delta },w_{\gamma
}\right) $ is given by the following theorem, not used in this paper. For
any real number $t$ let $t_{+}\equiv \max \left\{ t,0\right\} $ and $%
t_{-}\equiv \max \left\{ -t,0\right\} $ be the positive and negative parts
of $t$. Note that $t=t_{+}-t_{-}$. For the statement of the next theorem, we
will use the notation $0_{+}\leq A$ to mean the inequality $0<A$.

\begin{theorem}
\label{A char mn}Suppose that 
\begin{equation*}
1<p,q<\infty \text{ and }-\infty <\alpha ,\beta <\infty ,
\end{equation*}%
and $w_{\gamma }\left( x,y\right) =\left\vert \left( x,y\right) \right\vert
^{-\gamma }$ and $v_{\delta }\left( x,y\right) =\left\vert \left( x,y\right)
\right\vert ^{\delta }$ for $\left( x,y\right) \in \mathbb{R}^{m}\times 
\mathbb{R}^{n}$ with $-\infty <\gamma ,\delta <\infty $. Let 
\begin{equation*}
\Gamma \equiv \frac{1}{p}-\frac{1}{q}=\frac{1}{q^{\prime }}-\frac{1}{%
p^{\prime }},
\end{equation*}%
and set%
\begin{eqnarray*}
\bigtriangleup _{p,q}^{\gamma ,\delta }\left( m\right) &\equiv &\left(
\gamma -\frac{m}{q}\right) _{+}+\left( \delta -\frac{m}{p^{\prime }}\right)
_{+}\ , \\
\bigtriangleup _{p,q}^{\gamma ,\delta }\left( n\right) &\equiv &\left(
\gamma -\frac{n}{q}\right) _{+}+\left( \delta -\frac{n}{p^{\prime }}\right)
_{+}\ .
\end{eqnarray*}%
Then the Muckenhoupt characteristic is finite, i.e. 
\begin{equation*}
A_{p,q}^{\left( \alpha ,\beta \right) ,\left( m,n\right) }\left( v_{\delta
},w_{\gamma }\right) <\infty ,
\end{equation*}%
\textbf{if and only if} the power weights $w_{\gamma }^{q}$ and $v_{\delta
}^{-p^{\prime }}$ are locally integrable,%
\begin{equation}
\gamma q<m+n\text{ and }\delta p^{\prime }<m+n,  \label{local integ}
\end{equation}%
and the following power weight equality and constraint inequalities for $%
\alpha $ and $\beta $ hold:%
\begin{eqnarray}
\Gamma &=&\frac{\alpha +\beta -\gamma -\delta }{m+n},  \label{3 lines} \\
\Gamma +\frac{\bigtriangleup _{p,q}^{\gamma ,\delta }\left( n\right) }{m}
&\leq &\frac{\alpha }{m}\leq \Gamma +\frac{\gamma +\delta }{m}-\frac{%
\bigtriangleup _{p,q}^{\gamma ,\delta }\left( m\right) }{m},  \notag \\
\Gamma +\frac{\bigtriangleup _{p,q}^{\gamma ,\delta }\left( m\right) }{n}
&\leq &\frac{\beta }{n}\leq \Gamma +\frac{\gamma +\delta }{n}-\frac{%
\bigtriangleup _{p,q}^{\gamma ,\delta }\left( n\right) }{n}.  \notag
\end{eqnarray}
\end{theorem}

For the proof of Theorem \ref{A char mn} see the Appendix.

\subsection{Two weight theorems - product weights}

When both $v\left( u,t\right) =v_{1}\left( u\right) v_{2}\left( t\right) $
and $w\left( x,y\right) =w_{1}\left( x\right) w_{2}\left( y\right) $ are
product weights, the one-parameter theory carries over fairly easily, and
also for product measures $\sigma =\sigma _{1}\times \sigma _{2}$ and $%
\omega =\omega _{1}\times \omega _{2}$. Recall that%
\begin{equation*}
\widehat{\mathbb{A}}_{p,q}^{\left( \alpha ,\beta \right) ,\left( m,n\right)
}\left( \sigma ,\omega \right) =\sup_{I\subset \mathbb{R}^{m},\ J\subset 
\mathbb{R}^{n}}\left\vert I\right\vert ^{\frac{\alpha }{m}-1}\left\vert
J\right\vert ^{\frac{\beta }{n}-1}\left( \diint\limits_{\mathbb{R}^{m}\times 
\mathbb{R}^{n}}\widehat{s}_{I\times J}^{q}d\omega \right) ^{\frac{1}{q}%
}\left( \diint\limits_{\mathbb{R}^{m}\times \mathbb{R}^{n}}\widehat{s}%
_{I\times J}^{p^{\prime }}d\sigma \right) ^{\frac{1}{p^{\prime }}},
\end{equation*}%
and so for product measures $\sigma =\sigma _{1}\times \sigma _{2}$ and $%
\omega =\omega _{1}\times \omega _{2}$, we have%
\begin{eqnarray*}
\widehat{\mathbb{A}}_{p,q}^{\left( \alpha ,\beta \right) ,\left( m,n\right)
}\left( \sigma ,\omega \right) &=&\sup_{I\subset \mathbb{R}^{m}}\left\vert
I\right\vert ^{\frac{\alpha }{m}-1}\left( \int_{I}\widehat{s}_{I}^{q}d\omega
_{1}\right) ^{\frac{1}{q}}\left( \int_{I}\widehat{s}_{I}^{p^{\prime
}}d\sigma _{1}\right) ^{\frac{1}{p^{\prime }}}\cdot \sup_{J\subset \mathbb{R}%
^{n}}\left\vert J\right\vert ^{\frac{\beta }{n}-1}\left( \int_{J}\widehat{s}%
_{J}^{q}d\omega _{2}\right) ^{\frac{1}{q}}\left( \int_{J}\widehat{s}%
_{J}^{p^{\prime }}d\sigma _{2}\right) ^{\frac{1}{p^{\prime }}} \\
&=&\widehat{\mathbb{A}}_{p,q}^{\alpha ,m}\left( \sigma _{1},\omega
_{1}\right) \cdot \widehat{\mathbb{A}}_{p,q}^{\beta ,n}\left( \sigma
_{2},\omega _{2}\right) .
\end{eqnarray*}

\begin{theorem}
Suppose that $1<p,q<\infty $ and $0<\frac{\alpha }{m},\frac{\beta }{n}<1$.
If both $\sigma =\sigma _{1}\times \sigma _{2}$ and $\omega =\omega
_{1}\times \omega _{2}$ are product measures on $\mathbb{R}^{m}\times 
\mathbb{R}^{n}$, then the norm inequality (\ref{2 weight arb meas}) is
characterized by the two-tailed Muckenhoupt condition (\ref{A hat fraktur}),
i.e.%
\begin{equation*}
\mathbb{N}_{p,q}^{\left( \alpha ,\beta \right) ,\left( m,n\right) }\left(
\sigma ,\omega \right) \approx \widehat{\mathbb{A}}_{p,q}^{\left( \alpha
,\beta \right) ,\left( m,n\right) }\left( \sigma ,\omega \right) \ .
\end{equation*}
\end{theorem}

As we will see later, the proof of this theorem follows immediately from the 
$1$-parameter version, Theorem \ref{Saw and Wheed}, together with the
measure version of the iteration Theorem \ref{iteration}.

\subsection{Two weight $T1$ or testing conditions}

Recall that a weight $u$ satisfies the product doubling condition if $%
\left\vert 2R\right\vert _{u}\leq C\left\vert R\right\vert _{u}$ for all
rectangles $R$, and that this condition implies the weaker\ product reverse
doubling condition. The theorem here shows that under certain side
conditions on the weights $\sigma $ and $\omega $, the characteristic $%
A_{p,q}^{\left( \alpha ,\beta \right) ,\left( m,n\right) }\left( \sigma
,\omega \right) $ controls the testing conditions for $I_{\alpha ,\beta
}^{m,n}$.

\begin{theorem}
\label{testing}Suppose that $1<p<q<\infty $ and $0<\frac{\alpha }{m},\frac{%
\beta }{n}<1$ satisfy%
\begin{equation*}
\frac{1}{p}-\frac{1}{q}<\min \left\{ \frac{\alpha }{m},\frac{\beta }{n}%
\right\} ,
\end{equation*}%
and suppose that the locally finite positive Borel measures $\sigma $ and $%
\omega $ satisfy the product doubling condition and have product reverse
doubling exponent $\left( \varepsilon ,\varepsilon ^{\prime }\right) $ that
satisfies%
\begin{equation}
1-\frac{\alpha }{m}<\varepsilon <\frac{1-\frac{\alpha }{m}}{\frac{1}{q}+%
\frac{1}{p^{\prime }}}\text{ and }1-\frac{\beta }{n}<\varepsilon ^{\prime }<%
\frac{1-\frac{\beta }{n}}{\frac{1}{q}+\frac{1}{p^{\prime }}}.
\label{rev doub exp}
\end{equation}%
Then if the two weight characteristic $A_{p,q}^{\left( \alpha ,\beta \right)
,\left( m,n\right) }\left( \sigma ,\omega \right) $ is finite, the following
testing (or $T1$) conditions hold: for all rectangles $R\subset \mathbb{R}%
^{m}\times \mathbb{R}^{n}$,%
\begin{eqnarray}
\left\{ \diint\limits_{\mathbb{R}^{m}\times \mathbb{R}^{n}}I_{\alpha ,\beta
}^{m,n}\left( \mathbf{1}_{R}\sigma \right) \left( x,y\right) ^{q}\ d\omega
\left( x,y\right) \right\} ^{\frac{1}{q}} &\leq &C_{p,q}\mathbb{A}%
_{p,q}^{\left( \alpha ,m\right) ,\left( \beta ,n\right) }\left( \sigma
,\omega \right) \left\vert R\right\vert _{\sigma }^{\frac{1}{p}},
\label{testing conditions} \\
\left\{ \diint\limits_{\mathbb{R}^{m}\times \mathbb{R}^{n}}I_{\alpha ,\beta
}^{m,n}\left( \mathbf{1}_{R}\omega \right) \left( x,y\right) ^{p^{\prime }}\
d\sigma \left( x,y\right) \right\} ^{\frac{1}{p^{\prime }}} &\leq &C_{p,q}%
\mathbb{A}_{p,q}^{\left( \alpha ,m\right) ,\left( \beta ,n\right) }\left(
\sigma ,\omega \right) \left\vert R\right\vert _{\omega }^{\frac{1}{%
q^{\prime }}}.  \notag
\end{eqnarray}
\end{theorem}

\begin{remark}
The testing condition in the first line of (\ref{testing conditions}) only
requires the reverse doubling assumption on $\sigma $, while the testing
condition in the second line only requires reverse doubling of $\omega $.
\end{remark}

The proofs of our positive results in the one weight case involve the
standard techniques of iteration, Lebesgue's differentiation theorem and
Minkowski's inequality, and the proof of the product version of the
Stein-Weiss two power weight extension involves the sandwiching technique as
well, and finally the derivation of the testing conditions from the
characteristic and reverse doubling assumptions on the weights requires a
quasiorthogonality argument. We begin with the simpler one weight norm
inequality, and the special case of the two weight inequality when the
weights are product weights. Some of this material generalizes naturally
from product fractional integrals to \emph{iterated operators} to which we
now turn.

\section{Iterated operators}

The product fractional integral $I_{\alpha ,\beta }^{m,n}f\equiv \Omega
_{\alpha ,\beta }^{m,n}\ast f$, where 
\begin{equation*}
\Omega _{\alpha ,\beta }^{m,n}\left( x,y\right) =\left\vert x\right\vert
^{\alpha -m}\left\vert y\right\vert ^{\beta -n},\ \ \ \ \ \left( x,y\right)
\in \mathbb{R}^{m}\times \mathbb{R}^{n},
\end{equation*}%
is an example of an \emph{iterated operator}. In order to precisely define
what we mean by an iterated operator, we denote the collection of
nonnegative measurable functions on $\mathbb{R}^{n}$ by 
\begin{equation*}
\mathcal{N}\left( \mathbb{R}^{n}\right) \equiv \left\{ g:\mathbb{R}%
^{n}\rightarrow \left[ 0,\infty \right] :g\text{ is Lebesgue measurable}%
\right\} ,
\end{equation*}%
and we refer to a mapping $T:\mathcal{N}\left( \mathbb{R}^{n}\right)
\rightarrow \mathcal{N}\left( \mathbb{R}^{n}\right) $ from $\mathcal{N}%
\left( \mathbb{R}^{n}\right) $ to itself as an \emph{operator} on $\mathcal{N%
}\left( \mathbb{R}^{n}\right) $, without any assumption of additional
properties. If $T_{1}$ is an operator on $\mathcal{N}\left( \mathbb{R}%
^{m}\right) $, we define its \emph{product extension} to an operator $%
T_{1}\otimes \delta _{0}$ on $\mathcal{N}\left( \mathbb{R}^{m}\times \mathbb{%
R}^{n}\right) $ by%
\begin{equation*}
\left( T_{1}\otimes \delta _{0}\right) f\left( x,y\right) =T_{1}f^{y}\left(
x\right) ,\ \ \ \ \ f\in \mathcal{N}\left( \mathbb{R}^{m}\times \mathbb{R}%
^{n}\right) ,
\end{equation*}%
and similarly, if $T_{2}$ is an operator on $\mathcal{N}\left( \mathbb{R}%
^{n}\right) $ we define its \emph{product extension} to an operator $\delta
_{0}\otimes T_{2}$ on $\mathcal{N}\left( \mathbb{R}^{m}\times \mathbb{R}%
^{n}\right) $ by%
\begin{equation*}
\left( \delta _{0}\otimes T_{2}\right) f\left( x,y\right) =T_{2}f_{x}\left(
y\right) ,\ \ \ \ \ f\in \mathcal{N}\left( \mathbb{R}^{m}\times \mathbb{R}%
^{n}\right) .
\end{equation*}

\begin{definition}
\label{iterated operator}If $T_{1}$ is an operator on $\mathcal{N}\left( 
\mathbb{R}^{m}\right) $ and $T_{2}$ is an operator on $\mathcal{N}\left( 
\mathbb{R}^{n}\right) $, then the composition operator 
\begin{equation*}
T=\left( \delta _{0}\otimes T_{2}\right) \circ \left( T_{1}\otimes \delta
_{0}\right)
\end{equation*}%
on $\mathcal{N}\left( \mathbb{R}^{m}\times \mathbb{R}^{n}\right) $ is called
an \emph{iterated operator}.
\end{definition}

To see that $I_{\alpha ,\beta }^{m,n}$ is an iterated operator, define $%
\Omega _{\alpha }^{m}\left( x\right) =\left\vert x\right\vert ^{\alpha -m}$
and $\Omega _{\beta }^{n}\left( y\right) =\left\vert y\right\vert ^{\beta
-n} $ and set $I_{\gamma }^{k}g=\Omega _{\gamma }^{k}\ast g$. Then extend
the operator $I_{\alpha }^{m}$ from $\mathbb{R}^{m}$ to the product space $%
\mathbb{R}^{m}\times \mathbb{R}^{n}$ by defining%
\begin{eqnarray*}
\left( I_{\alpha }^{m}\otimes \delta _{0}\right) f\left( x,y\right) &=& 
\left[ \Omega _{\alpha }^{m}\otimes \delta _{0}\right] \ast f\left(
x,y\right) \\
&=&\int_{\mathbb{R}^{n}}\int_{\mathbb{R}^{m}}\Omega _{\alpha }^{m}\left(
x-u\right) \delta _{0}\left( y-v\right) f\left( u,v\right) dudv \\
&=&\int_{\mathbb{R}^{m}}\Omega _{\alpha }^{m}\left( x-u\right) f\left(
u,y\right) du \\
&=&I_{\alpha }^{m}f^{y}\left( x\right) ,
\end{eqnarray*}%
where $f^{y}\left( u\right) \equiv f\left( u,y\right) $, and similarly define%
\begin{equation*}
\left( \delta _{0}\otimes I_{\beta }^{n}\right) f\left( x,y\right) =I_{\beta
}^{n}f_{x}\left( y\right) ,
\end{equation*}%
where $f_{x}\left( v\right) \equiv f\left( x,v\right) $. Then from $\Omega
_{\alpha ,\beta }^{m,n}\left( x,y\right) =\Omega _{\alpha }^{m}\left(
x\right) \Omega _{\beta }^{n}\left( y\right) $ we have%
\begin{eqnarray*}
I_{\alpha ,\beta }^{m,n}f\left( x,y\right) &=&\int_{\mathbb{R}^{n}}\int_{%
\mathbb{R}^{m}}\Omega _{\alpha ,\beta }^{m,n}\left( x-u,y-v\right) f\left(
u,v\right) dudv \\
&=&\int_{\mathbb{R}^{n}}\Omega _{\beta }^{n}\left( y-v\right) \left\{ \int_{%
\mathbb{R}^{m}}\Omega _{\alpha }^{m}\left( x-u\right) f^{v}\left( u\right)
du\right\} dv \\
&=&\int_{\mathbb{R}^{n}}\Omega _{\beta }^{n}\left( y-v\right) \left\{
I_{\alpha }^{m}f^{v}\left( x\right) \right\} dv \\
&=&\int_{\mathbb{R}^{n}}\Omega _{\beta }^{n}\left( y-v\right) \left\{ \left(
I_{\alpha }^{m}\otimes \delta _{0}\right) f\left( x,\nu \right) \right\} dv
\\
&=&\int_{\mathbb{R}^{n}}\Omega _{\beta }^{n}\left( y-v\right) \left\{ \left[
\left( I_{\alpha }^{m}\otimes \delta _{0}\right) f\right] _{x}\left( \nu
\right) \right\} dv \\
&=&I_{\beta }^{n}\left[ \left( I_{\alpha }^{m}\otimes \delta _{0}\right) f%
\right] _{x}\left( y\right) \\
&=&\left( \delta _{0}\otimes I_{\beta }^{n}\right) \circ \left( I_{\alpha
}^{m}\otimes \delta _{0}\right) f\left( x,y\right) .
\end{eqnarray*}%
Thus 
\begin{equation*}
I_{\alpha ,\beta }^{m,n}=\left( \delta _{0}\otimes I_{\beta }^{n}\right)
\circ \left( I_{\alpha }^{m}\otimes \delta _{0}\right) ,
\end{equation*}%
and similarly we have 
\begin{equation*}
I_{\alpha ,\beta }^{m,n}=\left( I_{\alpha }^{m}\otimes \delta _{0}\right)
\circ \left( \delta _{0}\otimes I_{\beta }^{n}\right) ,
\end{equation*}%
which expresses $I_{\alpha ,\beta }^{m,n}$ as an iterated operator, namely
the compostion of two commuting operators $I_{\alpha }^{m}\otimes \delta
_{0} $ and $\delta _{0}\otimes I_{\beta }^{n}$.

More generally, we can consider the operator%
\begin{equation*}
Tf\left( x,y\right) =\left( K\ast f\right) \left( x,y\right) =\int_{\mathbb{R%
}^{n}}\int_{\mathbb{R}^{m}}K\left( x-u,y-v\right) f\left( u,v\right) dudv\ ,
\end{equation*}%
where $K$ is a \emph{product kernel} on $\mathbb{R}^{m}\times \mathbb{R}^{n}$%
,%
\begin{equation*}
K\left( x,y\right) =K_{1}\left( x\right) K_{2}\left( y\right) ,\ \ \ \ \
\left( x,y\right) \in \mathbb{R}^{m}\times \mathbb{R}^{n},
\end{equation*}%
and obtain the factorizations%
\begin{equation*}
T=\left( K_{1}\otimes \delta _{0}\right) \circ \left( \delta _{0}\otimes
K_{2}\right) =\left( \delta _{0}\otimes K_{2}\right) \circ \left(
K_{1}\otimes \delta _{0}\right)
\end{equation*}%
of $T$ into iterated operators where%
\begin{eqnarray*}
\left( K_{1}\otimes \delta _{0}\right) \ast g\left( x,y\right) &=&\int_{%
\mathbb{R}^{n}}\int_{\mathbb{R}^{m}}K_{1}\left( x-u\right) \delta _{0}\left(
y-v\right) g\left( u,v\right) dudv \\
&=&\int_{\mathbb{R}^{m}}K_{1}\left( x-u\right) g\left( u,y\right) du \\
&=&K_{1}\ast g^{y}\left( x\right) ,
\end{eqnarray*}%
and%
\begin{equation*}
\left( \delta _{0}\otimes K_{2}\right) \ast h\left( x,y\right) =K_{2}\ast
g_{x}\left( y\right) .
\end{equation*}

As a final example, let $T_{1}=M_{\mathbb{R}^{m}}$ be the Hardy-Littlewood
maximal operator on $\mathbb{R}^{m}$ and let $T_{2}=M_{\mathbb{R}^{n}}$ be
the Hardy-Littlewood maximal operator on $\mathbb{R}^{n}$. Then the iterated
operator 
\begin{equation}
T=\left( \delta _{0}\otimes T_{2}\right) \circ \left( T_{1}\otimes \delta
_{0}\right) =\left( \delta _{0}\otimes M_{\mathbb{R}^{n}}\right) \circ
\left( M_{\mathbb{R}^{m}}\otimes \delta _{0}\right)  \label{strong max}
\end{equation}%
is usually denoted $M_{\mathbb{R}^{n}}\left( M_{\mathbb{R}^{m}}\right) $,
and the other iterated operator by $M_{\mathbb{R}^{m}}\left( M_{\mathbb{R}%
^{n}}\right) $. They both dominate the strong maximal operator $\mathcal{M}$
given in (\ref{def strong max}).

\subsection{One weight inequalities for iterated operators}

Define the iterated Lebesgue spaces $L_{m,n}^{p,q}$ on $\mathbb{R}^{m}\times 
\mathbb{R}^{n}$ by%
\begin{equation*}
\left\Vert F\right\Vert _{L_{m,n}^{p,q}}\equiv \left\Vert \left\Vert
F\right\Vert _{L^{p}\left( \mathbb{R}^{m}\right) }\right\Vert _{L^{q}\left( 
\mathbb{R}^{n}\right) }=\left\{ \int_{\mathbb{R}^{n}}\left\{ \int_{\mathbb{R}%
^{m}}F\left( x,y\right) ^{p}dx\right\} ^{\frac{q}{p}}dy\right\} ^{\frac{1}{q}%
}.
\end{equation*}%
In the proof of the next theorem we will use Minkowski's inequality for
nonnegative functions,%
\begin{equation*}
\left\Vert F\right\Vert _{L_{m,n}^{p,q}}\leq \left\Vert F\right\Vert
_{L_{n,m}^{q,p}},\ \ \ \ \ \text{for all }F\geq 0\text{ and }1\leq p\leq
q\leq \infty ,
\end{equation*}%
which written out in full is%
\begin{equation*}
\left\{ \int_{\mathbb{R}^{m}}\left\{ \int_{\mathbb{R}^{n}}F\left( x,y\right)
^{p}dy\right\} ^{\frac{q}{p}}dx\right\} ^{\frac{1}{q}}\leq \left\{ \int_{%
\mathbb{R}^{n}}\left\{ \int_{\mathbb{R}^{m}}F\left( x,y\right)
^{q}dx\right\} ^{\frac{p}{q}}dy\right\} ^{\frac{1}{p}}.
\end{equation*}

\begin{theorem}
\label{iter op}Let $1\leq p\leq q\leq \infty $ and $T_{1}:\mathcal{N}\left( 
\mathbb{R}^{m}\right) \rightarrow \mathcal{N}\left( \mathbb{R}^{m}\right) $
and $T_{2}:\mathcal{N}\left( \mathbb{R}^{n}\right) \rightarrow \mathcal{N}%
\left( \mathbb{R}^{n}\right) $. Suppose that $v\left( x,y\right) \geq 0$ on $%
\mathbb{R}^{m}\times \mathbb{R}^{n}$ satisfies%
\begin{equation}
\left\Vert T_{1}g\right\Vert _{L^{q}\left( \left( v^{y}\right) ^{q}\right)
}\leq C_{1}\left\Vert g\right\Vert _{L^{p}\left( \left( v^{y}\right)
^{p}\right) }\ ,\ \ \ \ \ \text{for all }g\geq 0,  \label{T1}
\end{equation}%
uniformly for $y\in \mathbb{R}^{n}$, and%
\begin{equation}
\left\Vert T_{2}h\right\Vert _{L^{q}\left( \left( v_{x}\right) ^{q}\right)
}\leq C_{2}\left\Vert h\right\Vert _{L^{p}\left( \left( v_{x}\right)
^{p}\right) }\ ,\ \ \ \ \ \text{for all }h\geq 0,  \label{T2}
\end{equation}%
uniformly for $x\in \mathbb{R}^{m}$. Then the iterated operator $T=\left(
\delta _{0}\otimes T_{2}\right) \circ \left( T_{1}\otimes \delta _{0}\right) 
$ satisfies%
\begin{equation*}
\left\Vert Tf\right\Vert _{L^{q}\left( v^{q}\right) }\leq
C_{1}C_{2}\left\Vert f\right\Vert _{L^{p}\left( v^{p}\right) }\ ,\ \ \ \ \ 
\text{for all }f\geq 0.
\end{equation*}
\end{theorem}

\begin{proof}
We have%
\begin{eqnarray*}
\left\Vert Tf\right\Vert _{L^{q}\left( v^{q}\right) } &=&\left\{ \int_{%
\mathbb{R}^{m}}\int_{\mathbb{R}^{n}}\left( \delta _{0}\otimes T_{2}\right)
\circ \left( T_{1}\otimes \delta _{0}\right) f\left( x,y\right) ^{q}v\left(
x,y\right) ^{q}dydx\right\} ^{\frac{1}{q}} \\
&=&\left\{ \int_{\mathbb{R}^{m}}\left[ \int_{\mathbb{R}^{n}}\left\vert T_{2}%
\left[ \left( T_{1}\otimes \delta _{0}\right) f\right] _{x}\left( y\right)
\right\vert ^{q}v_{x}\left( y\right) ^{q}dy\right] dx\right\} ^{\frac{1}{q}}
\\
&=&\left\{ \int_{\mathbb{R}^{m}}\left\Vert T_{2}\left[ \left( T_{1}\otimes
\delta _{0}\right) f\right] _{x}\right\Vert _{L^{q}\left( \left(
v_{x}\right) ^{q}\right) }^{q}dx\right\} ^{\frac{1}{q}} \\
&\leq &C_{2}\left\{ \int_{\mathbb{R}^{m}}\left\Vert \left[ \left(
T_{1}\otimes \delta _{0}\right) f\right] _{x}\right\Vert _{L^{p}\left(
\left( v_{x}\right) ^{p}\right) }^{q}dx\right\} ^{\frac{1}{q}} \\
&=&C_{2}\left\{ \int_{\mathbb{R}^{m}}\left\{ \int_{\mathbb{R}^{n}}\left(
T_{1}\otimes \delta _{0}\right) f\left( x,y\right) ^{p}v\left( x,y\right)
^{p}dy\right\} ^{\frac{q}{p}}dx\right\} ^{\frac{1}{q}},
\end{eqnarray*}%
where we have used $h=\left[ \left( T_{1}\otimes \delta _{0}\right) f\right]
_{x}\geq 0$ in (\ref{T2}). Then by Minkowski's inequality applied to the
nonnegative function $F=\left( T_{1}\otimes \delta _{0}\right) f\left(
x,y\right) \ v\left( x,y\right) $, this is dominated by%
\begin{eqnarray*}
&&C_{2}\left\{ \int_{\mathbb{R}^{n}}\left\{ \int_{\mathbb{R}^{m}}\left(
T_{1}\otimes \delta _{0}\right) f\left( x,y\right) ^{q}v\left( x,y\right)
^{q}dx\right\} ^{\frac{p}{q}}dy\right\} ^{\frac{1}{p}} \\
&=&C_{2}\left\{ \int_{\mathbb{R}^{n}}\left\{ \int_{\mathbb{R}%
^{m}}T_{1}f^{y}\left( x\right) ^{q}v^{y}\left( x\right) ^{q}dx\right\} ^{%
\frac{p}{q}}dy\right\} ^{\frac{1}{p}} \\
&=&C_{2}\left\{ \int_{\mathbb{R}^{n}}\left\Vert T_{1}f^{y}\right\Vert
_{L^{q}\left( \left( v^{y}\right) ^{q}\right) }^{p}dy\right\} ^{\frac{1}{p}}
\\
&\leq &C_{2}C_{1}\left\{ \int_{\mathbb{R}^{n}}\left\Vert f^{y}\right\Vert
_{L^{p}\left( \left( v^{y}\right) ^{p}\right) }^{p}dy\right\} ^{\frac{1}{p}%
}=C_{2}C_{1}\left\Vert f\right\Vert _{L^{p}\left( v^{p}\right) }\ ,
\end{eqnarray*}%
where we have used $g=f^{y}\geq 0$ in (\ref{T1}).
\end{proof}

The following porisms, or `corollaries of the proof', of Theorem \ref{iter
op} will find application in proving Theorem \ref{original A1}\ below.

\begin{description}
\item[Porism1] \label{porism}If we replace (\ref{T2}) with the more general 
\emph{two weight} inequality%
\begin{equation}
\left\Vert T_{2}h\right\Vert _{L^{q}\left( \left( w_{x}\right) ^{q}\right)
}\leq C_{2}\left\Vert h\right\Vert _{L^{p}\left( \left( v_{x}\right)
^{p}\right) }\ ,\ \ \ \ \ \text{for all }h\geq 0,  \label{T1'}
\end{equation}%
for some weight $w\left( x,y\right) \geq 0$ on $\mathbb{R}^{m}\times \mathbb{%
R}^{n}$, then the iterated operator $T=\left( \delta _{0}\otimes
T_{2}\right) \circ \left( T_{1}\otimes \delta _{0}\right) $ satisfies the
two weight inequality%
\begin{equation*}
\left\Vert Tf\right\Vert _{L^{q}\left( w^{q}\right) }\leq
C_{1}C_{2}\left\Vert f\right\Vert _{L^{p}\left( v^{p}\right) }\ ,\ \ \ \ \ 
\text{for all }f\geq 0.
\end{equation*}
\end{description}

To prove this Porism, we modify the first display in the proof of Theorem %
\ref{iter op} to this, 
\begin{eqnarray*}
\left\Vert Tf\right\Vert _{L^{q}\left( w^{q}\right) } &=&\left\{ \int_{%
\mathbb{R}^{m}}\int_{\mathbb{R}^{n}}\left( \delta _{0}\otimes T_{2}\right)
\circ \left( T_{1}\otimes \delta _{0}\right) f\left( x,y\right) ^{q}w\left(
x,y\right) ^{q}dydx\right\} ^{\frac{1}{q}} \\
&=&\left\{ \int_{\mathbb{R}^{m}}\left[ \int_{\mathbb{R}^{n}}\left\vert T_{2}%
\left[ \left( T_{1}\otimes \delta _{0}\right) f\right] _{x}\left( y\right)
\right\vert ^{q}w_{x}\left( y\right) ^{q}dy\right] dx\right\} ^{\frac{1}{q}}
\\
&=&\left\{ \int_{\mathbb{R}^{m}}\left\Vert T_{2}\left[ \left( T_{1}\otimes
\delta _{0}\right) f\right] _{x}\right\Vert _{L^{q}\left( \left(
w_{x}\right) ^{q}\right) }^{q}dx\right\} ^{\frac{1}{q}} \\
&\leq &C_{2}\left\{ \int_{\mathbb{R}^{m}}\left\Vert \left[ \left(
T_{1}\otimes \delta _{0}\right) f\right] _{x}\right\Vert _{L^{p}\left(
\left( v_{x}\right) ^{p}\right) }^{q}dx\right\} ^{\frac{1}{q}} \\
&=&C_{2}\left\{ \int_{\mathbb{R}^{m}}\left\{ \int_{\mathbb{R}^{n}}\left(
T_{1}\otimes \delta _{0}\right) f\left( x,y\right) ^{p}v\left( x,y\right)
^{p}dy\right\} ^{\frac{q}{p}}dx\right\} ^{\frac{1}{q}},
\end{eqnarray*}%
and then the remainder of the proof of Theorem \ref{iter op} applies
verbatim.

There is also the following symmetrical porism whose proof is left to the
reader.

\begin{description}
\item[Porism2] \label{porism2}If $\widehat{T}=\left( T_{1}\otimes \delta
_{0}\right) \circ \left( \delta _{0}\otimes T_{2}\right) $ is the
composition of the two iterated operators in the reverse order, and if%
\begin{equation*}
\left\Vert T_{1}g\right\Vert _{L^{q}\left( \left( w^{y}\right) ^{q}\right)
}\leq C_{1}\left\Vert g\right\Vert _{L^{p}\left( \left( v^{y}\right)
^{p}\right) }\ ,\ \ \ \ \ \text{for all }g\geq 0,
\end{equation*}%
uniformly for $y\in \mathbb{R}^{n}$, and%
\begin{equation*}
\left\Vert T_{2}h\right\Vert _{L^{q}\left( \left( v_{x}\right) ^{q}\right)
}\leq C_{2}\left\Vert h\right\Vert _{L^{p}\left( \left( v_{x}\right)
^{p}\right) }\ ,\ \ \ \ \ \text{for all }h\geq 0,
\end{equation*}%
uniformly for $x\in \mathbb{R}^{m}$, then%
\begin{equation*}
\left\Vert Tf\right\Vert _{L^{q}\left( w^{q}\right) }\leq
C_{1}C_{2}\left\Vert f\right\Vert _{L^{p}\left( v^{p}\right) }\ ,\ \ \ \ \ 
\text{for all }f\geq 0.
\end{equation*}
\end{description}

In the special case that%
\begin{equation}
\frac{1}{p}-\frac{1}{q}=\frac{\alpha }{m}=\frac{\beta }{n},  \label{star}
\end{equation}%
we can exploit the equivalence of the product $A_{p,q}$ condition (\ref{Apq}%
) with the iterated $A_{p,q}$ condition (\ref{Apq}) to obtain a
characterization of the one weight $L^{p}\rightarrow L^{q}$ inequality for $%
I_{\alpha ,\beta }^{m,n}$. In fact, condition (\ref{star}) is actually
necessary for the product $A_{p,q}$ condition (\ref{Apq}) to be finite.

\begin{definition}
We set%
\begin{equation*}
A_{p,q}^{\left( \alpha ,\beta \right) ,\left( m,n\right) }\left( w\right)
\equiv \sup_{I\subset \mathbb{R}^{m},\ J\subset \mathbb{R}^{n}}\left\vert
I\right\vert ^{\frac{\alpha }{m}+\frac{1}{q}-\frac{1}{p}}\left\vert
J\right\vert ^{\frac{\beta }{n}+\frac{1}{q}-\frac{1}{p}}\left( \frac{%
\left\vert I\times J\right\vert _{w^{q}}}{\left\vert I\times J\right\vert }%
\right) ^{\frac{1}{q}}\left( \frac{\left\vert I\times J\right\vert
_{w^{p^{\prime }}}}{\left\vert I\times J\right\vert }\right) ^{\frac{1}{%
p^{\prime }}}.
\end{equation*}
\end{definition}

\begin{claim}
\label{bal and diag}If $N_{p,q}^{\left( \alpha ,\beta \right) ,\left(
m,n\right) ,}\left( w\right) <\infty $ for some weight $w$, then (\ref{star}%
) holds.
\end{claim}

\begin{proof}
Following the proof of Theorem \ref{Muck and Wheed sharp} in the Appendix,
we apply H\"{o}lder's inequality with dual exponents $\frac{p^{\prime }+1}{%
p^{\prime }}$ and $p^{\prime }+1$ to obtain%
\begin{eqnarray*}
1 &=&\left\{ \frac{1}{\left\vert I\times J\right\vert }\int \int_{I\times
J}w^{\frac{p^{\prime }}{p^{\prime }+1}}w^{-\frac{p^{\prime }}{p^{\prime }+1}%
}\right\} ^{\frac{p^{\prime }+1}{p^{\prime }}} \\
&\leq &\left\{ \left( \frac{1}{\left\vert I\times J\right\vert }\int
\int_{I\times J}w\right) ^{\frac{p^{\prime }}{p^{\prime }+1}}\left( \frac{1}{%
\left\vert I\times J\right\vert }\int \int_{I\times J}w^{-p^{\prime
}}\right) ^{\frac{1}{p^{\prime }+1}}\right\} ^{\frac{p^{\prime }+1}{%
p^{\prime }}} \\
&=&\left( \frac{1}{\left\vert I\times J\right\vert }\int \int_{I\times
J}w\right) \left( \frac{1}{\left\vert I\times J\right\vert }\int
\int_{I\times J}w^{-p^{\prime }}\right) ^{\frac{1}{p^{\prime }}} \\
&\leq &\left( \frac{1}{\left\vert I\times J\right\vert }\int \int_{I\times
J}w^{q}\right) ^{\frac{1}{q}}\left( \frac{1}{\left\vert I\times J\right\vert 
}\int \int_{I\times J}w^{-p^{\prime }}\right) ^{\frac{1}{p^{\prime }}}.
\end{eqnarray*}%
Then we conclude that%
\begin{eqnarray*}
&&\left\vert I\right\vert ^{\frac{\alpha }{m}+\frac{1}{q}-\frac{1}{p}%
}\left\vert J\right\vert ^{\frac{\beta }{n}+\frac{1}{q}-\frac{1}{p}} \\
&\leq &\left\vert I\right\vert ^{\frac{\alpha }{m}+\frac{1}{q}-\frac{1}{p}%
}\left\vert J\right\vert ^{\frac{\beta }{n}+\frac{1}{q}-\frac{1}{p}}\left( 
\frac{1}{\left\vert I\times J\right\vert }\int \int_{I\times J}w^{q}\right)
^{\frac{1}{q}}\left( \frac{1}{\left\vert I\times J\right\vert }\int
\int_{I\times J}w^{-p^{\prime }}\right) ^{\frac{1}{p^{\prime }}} \\
&=&\left\vert I\right\vert ^{\frac{\alpha }{m}-1}\left\vert J\right\vert ^{%
\frac{\beta }{n}-1}\left( \int \int_{I\times J}w^{q}\right) ^{\frac{1}{q}%
}\left( \int \int_{I\times J}w^{-p^{\prime }}\right) ^{\frac{1}{p^{\prime }}%
}\leq N_{p,q}^{\alpha ,m}\left( w\right)
\end{eqnarray*}%
for all rectangles $I\times J$, which implies $\frac{\alpha }{m}=\frac{1}{p}-%
\frac{1}{q}=\frac{\beta }{n}$ as required.
\end{proof}

\subsection{Proof of Theorem \protect\ref{one weight product}}

Now we turn to the proof of Theorem \ref{one weight product}.

\begin{proof}[Proof of Theorem \protect\ref{one weight product}]
By Claim \ref{bal and diag} the balanced and diagonal condition (\ref{star})
holds. Thus we have the necessity of (\ref{Apq}) follows from Lemma \ref%
{tails} above:%
\begin{equation*}
\left( \frac{1}{\left\vert I\right\vert \left\vert J\right\vert }\int
\int_{I\times J}w\left( u,v\right) ^{-p^{\prime }}dudv\right) ^{\frac{1}{%
p^{\prime }}}\left( \frac{1}{\left\vert I\right\vert \left\vert J\right\vert 
}\int \int_{I\times J}w\left( x,y\right) ^{q}\ dxdy\right) ^{\frac{1}{q}%
}\leq N_{p,q}\left( w\right) .
\end{equation*}

By letting the cubes $J$ and $I$ shrink separately to points $y\in \mathbb{R}%
^{n}$ and $x\in \mathbb{R}^{m}$, we obtain (\ref{iter Apq}). Then from the
one weight theorem of Muckenhoupt and Wheeden above, Theorem \ref{Muck and
Wheed}, we conclude that both (\ref{T1}) and (\ref{T2}) hold with $%
T_{1}g=\Omega _{\alpha }^{m}\ast g$ and $T_{2}h=\Omega _{\beta }^{n}\ast h$.
Then the conclusion of Theorem \ref{iter op} proves the norm inequality (\ref%
{norm one weight}).

Finally, the estimate (\ref{characteristic power}) follows upon using the
estimate of Lacey, Moen, P\'{e}rez and Torres \cite[LaMoPeTo]{LaMoPeTo},
which gives%
\begin{equation*}
C_{1}\lesssim A_{p,q}\left( w^{y}\right) ^{q\max \left\{ 1,\frac{p^{\prime }%
}{q}\right\} \left( 1-\frac{\alpha }{m}\right) }\text{ and }C_{2}\lesssim
A_{p,q}\left( w_{x}\right) ^{q\max \left\{ 1,\frac{p^{\prime }}{q}\right\}
\left( 1-\frac{\beta }{n}\right) },
\end{equation*}%
uniformly in $x$ and $y$, upon noting that the characteristic $A_{p,q}$ used
in \cite[LaMoPeTo]{LaMoPeTo} is the $q^{th}$ power of that defined here.
Then using that%
\begin{equation*}
q\max \left\{ 1,\frac{p^{\prime }}{q}\right\} \left( 1-\frac{\alpha }{m}%
\right) =\max \left\{ q,p^{\prime }\right\} \left( \frac{1}{q}+\frac{1}{%
p^{\prime }}\right) =1+\max \left\{ \frac{p^{\prime }}{q},\frac{q}{p^{\prime
}}\right\} ,
\end{equation*}%
we obtain (\ref{characteristic power}). Sharpness of the exponent follows
upon taking product power weights and product power functions and then
arguing as in the previous work of Buckley \cite[Buc]{Buc} and Lacey, Moen, P%
\'{e}rez and Torres \cite{LaMoPeTo}.
\end{proof}

\subsection{Two weight inequalities for product weights}

If in addition we consider \emph{product} weights, then we can prove \emph{%
two weight} versions of the theorem and corollary above using essentially
the same proof strategy, namely iteration of one parameter operators.

\begin{theorem}
\label{iteration}Let $1\leq p\leq q\leq \infty $ and $T_{1}:\mathcal{N}%
\left( \mathbb{R}^{m}\right) \rightarrow \mathcal{N}\left( \mathbb{R}%
^{m}\right) $ and $T_{2}:\mathcal{N}\left( \mathbb{R}^{n}\right) \rightarrow 
\mathcal{N}\left( \mathbb{R}^{n}\right) $. Suppose that $w\left( x,y\right)
=w_{1}\left( x\right) w_{2}\left( y\right) $ and $v\left( x,y\right)
=v_{1}\left( x\right) v_{2}\left( y\right) $ are both product weights, and
that the weight pairs $\left( w_{1},v_{1}\right) $ and $\left(
w_{2},v_{2}\right) $ on $\mathbb{R}^{m}$ and $\mathbb{R}^{n}$ respectively
satisfy%
\begin{equation}
\left\Vert T_{1}g\right\Vert _{L^{q}\left( w_{1}^{q}\right) }\leq
C_{1}\left\Vert g\right\Vert _{L^{p}\left( v_{1}^{p}\right) }\ ,\ \ \ \ \ 
\text{for all }g\geq 0,  \label{two T1}
\end{equation}%
and%
\begin{equation}
\left\Vert T_{2}h\right\Vert _{L^{q}\left( w_{2}^{q}\right) }\leq
C_{2}\left\Vert h\right\Vert _{L^{p}\left( v_{2}^{p}\right) }\ ,\ \ \ \ \ 
\text{for all }h\geq 0.  \label{two T2}
\end{equation}%
Then the iterated operator $T=\left( \delta _{0}\otimes T_{2}\right) \circ
\left( T_{1}\otimes \delta _{0}\right) $ satisfies%
\begin{equation*}
\left\Vert Tf\right\Vert _{L^{q}\left( w^{q}\right) }\leq
C_{1}C_{2}\left\Vert f\right\Vert _{L^{p}\left( v^{p}\right) }\ ,\ \ \ \ \ 
\text{for all }f\geq 0.
\end{equation*}
\end{theorem}

\begin{proof}
We have%
\begin{eqnarray*}
\left\Vert Tf\right\Vert _{L^{q}\left( w^{q}\right) } &=&\left\{ \int_{%
\mathbb{R}^{m}}\int_{\mathbb{R}^{n}}\left( \delta _{0}\otimes T_{2}\right)
\circ \left( T_{1}\otimes \delta _{0}\right) f\left( x,y\right) ^{q}w\left(
x,y\right) ^{q}dydx\right\} ^{\frac{1}{q}} \\
&=&\left\{ \int_{\mathbb{R}^{m}}\left[ \int_{\mathbb{R}^{n}}\left\vert T_{2}%
\left[ \left( T_{1}\otimes \delta _{0}\right) f\right] _{x}\left( y\right)
\right\vert ^{q}w_{1}\left( x\right) ^{q}w_{2}\left( y\right) ^{q}dy\right]
dx\right\} ^{\frac{1}{q}} \\
&=&\left\{ \int_{\mathbb{R}^{m}}\left\Vert T_{2}\left[ \left( T_{1}\otimes
\delta _{0}\right) f\right] _{x}\right\Vert _{L^{q}\left( \left(
w_{2}\right) ^{q}\right) }^{q}w_{1}\left( x\right) ^{q}dx\right\} ^{\frac{1}{%
q}} \\
&\leq &C_{2}\left\{ \int_{\mathbb{R}^{m}}\left\Vert \left[ \left(
T_{1}\otimes \delta _{0}\right) f\right] _{x}\right\Vert _{L^{p}\left(
\left( v_{2}\right) ^{p}\right) }^{q}w_{1}\left( x\right) ^{q}dx\right\} ^{%
\frac{1}{q}} \\
&=&C_{2}\left\{ \int_{\mathbb{R}^{m}}\left\{ \int_{\mathbb{R}^{n}}\left(
T_{1}\otimes \delta _{0}\right) f\left( x,y\right) ^{p}v_{2}\left( y\right)
^{p}w_{1}\left( x\right) ^{p}dy\right\} ^{\frac{q}{p}}dx\right\} ^{\frac{1}{q%
}},
\end{eqnarray*}%
where we have used $h=\left[ \left( T_{1}\otimes \delta _{0}\right) f\right]
_{x}\geq 0$ in (\ref{two T2}). Then by Minkowski's inequality applied to the
nonnegative function $F=\left( T_{1}\otimes \delta _{0}\right) f\left(
x,y\right) \ v_{2}\left( y\right) w_{1}\left( x\right) $, this is dominated
by%
\begin{eqnarray*}
&&C_{2}\left\{ \int_{\mathbb{R}^{n}}\left\{ \int_{\mathbb{R}^{m}}\left(
T_{1}\otimes \delta _{0}\right) f\left( x,y\right) ^{q}v_{2}\left( y\right)
^{q}w_{1}\left( x\right) ^{q}dx\right\} ^{\frac{p}{q}}dy\right\} ^{\frac{1}{p%
}} \\
&=&C_{2}\left\{ \int_{\mathbb{R}^{n}}\left\{ \int_{\mathbb{R}%
^{m}}T_{1}f^{y}\left( x\right) ^{q}w_{1}\left( x\right) ^{q}dx\right\} ^{%
\frac{p}{q}}v_{2}\left( y\right) ^{p}dy\right\} ^{\frac{1}{p}} \\
&=&C_{2}\left\{ \int_{\mathbb{R}^{n}}\left\Vert T_{1}f^{y}\right\Vert
_{L^{q}\left( \left( w_{1}\right) ^{q}\right) }^{p}v_{2}\left( y\right)
^{p}dy\right\} ^{\frac{1}{p}} \\
&\leq &C_{2}C_{1}\left\{ \int_{\mathbb{R}^{n}}\left\Vert f^{y}\right\Vert
_{L^{p}\left( \left( v_{1}\right) ^{p}\right) }^{p}v_{2}\left( y\right)
^{p}dy\right\} ^{\frac{1}{p}}=C_{2}C_{1}\left\Vert f\right\Vert
_{L^{p}\left( v^{p}\right) }\ ,
\end{eqnarray*}%
where we have used $g=f^{y}\geq 0$ in (\ref{two T1}).
\end{proof}

\begin{corollary}
\label{product weights}Suppose $1<p<q<\infty $ and $\frac{1}{p}-\frac{1}{q}=%
\frac{\alpha }{m}=\frac{\beta }{n}$. Let $w\left( x,y\right) =w_{1}\left(
x\right) w_{2}\left( y\right) $ and $v\left( x,y\right) =v_{1}\left(
x\right) v_{2}\left( y\right) $ be a pair of nonnegative \emph{product}
weights on $\mathbb{R}^{m}\times \mathbb{R}^{n}$. Then%
\begin{equation}
\left\{ \int_{\mathbb{R}^{m}}\int_{\mathbb{R}^{n}}I_{\alpha ,\beta
}^{m,n}f\left( x,y\right) ^{q}\ w\left( x,y\right) ^{q}\ dxdy\right\} ^{%
\frac{1}{q}}\leq C_{p,q}\left( w,v\right) \left\{ \int_{\mathbb{R}^{m}}\int_{%
\mathbb{R}^{n}}f\left( x,y\right) ^{p}\ v\left( x,y\right) ^{p}\
dxdy\right\} ^{\frac{1}{p}}  \label{norm two weight}
\end{equation}%
for all $f\geq 0$ \emph{if and only if}%
\begin{eqnarray}
&&\sup_{I\subset \mathbb{R}^{m},\ J\subset \mathbb{R}^{n}}\left( \frac{1}{%
\left\vert I\right\vert \left\vert J\right\vert }\int \int_{\mathbb{R}%
^{m}\times \mathbb{R}^{n}}\left( \widehat{s}_{I,J}w\right) \left( x,y\right)
^{q}\ dxdy\right) ^{\frac{1}{q}}  \label{two weight Apq hat} \\
&&\ \ \ \ \ \ \ \ \ \ \ \ \ \ \ \ \ \ \ \ \ \ \ \ \ \times \left( \frac{1}{%
\left\vert I\right\vert \left\vert J\right\vert }\int \int_{\mathbb{R}%
^{m}\times \mathbb{R}^{n}}\left( \widehat{s}_{I,J}v^{-1}\right) \left(
x,y\right) ^{p^{\prime }}\ dxdy\right) ^{\frac{1}{p^{\prime }}}\equiv 
\widetilde{A}_{p,q}\left( w,v\right) <\infty ,  \notag
\end{eqnarray}%
where $\widehat{s}_{I,J}\left( x,y\right) =\widehat{s}_{I}\left( x\right) 
\widehat{s}_{J}\left( y\right) =\left( 1+\frac{\left\vert x-c_{I}\right\vert 
}{\left\vert I\right\vert ^{\frac{1}{m}}}\right) ^{\alpha -m}\left( 1+\frac{%
\left\vert y-c_{J}\right\vert }{\left\vert J\right\vert ^{\frac{1}{m}}}%
\right) ^{\beta -n}$. Moreover, 
\begin{equation}
C_{p,q}\left( w,v\right) \approx \widetilde{A}_{p,q}\left( w,v\right) .
\label{characteristic equiv}
\end{equation}
\end{corollary}

\begin{proof}
The necessity of (\ref{two weight Apq hat}) is again a standard exercise in
adapting the one parameter argument in Sawyer and Wheeden to the setting of
product weights.

Now we turn to the sufficiency of (\ref{two weight Apq hat}). Since our
weights $w$ and $v$ are product weights, the double integrals on the left
hand side of (\ref{two weight Apq hat}) each factor as a product of
integrals over $\mathbb{R}^{m}$ and $\mathbb{R}^{n}$ separately, e.g.%
\begin{eqnarray*}
&&\left( \frac{1}{\left\vert I\right\vert \left\vert J\right\vert }\int
\int_{\mathbb{R}^{m}\times \mathbb{R}^{n}}\left( \widehat{s}_{I,J}w\right)
\left( x,y\right) ^{q}\ dxdy\right) ^{\frac{1}{q}} \\
&=&\left( \frac{1}{\left\vert I\right\vert }\int_{\mathbb{R}^{m}}\left( 
\widehat{s}_{I}w_{1}\right) \left( x\right) ^{q}\ dx\right) ^{\frac{1}{q}%
}\left( \frac{1}{\left\vert J\right\vert }\int_{\mathbb{R}^{n}}\left( 
\widehat{s}_{J}w_{2}\right) \left( y\right) ^{q}\ dy\right) ^{\frac{1}{q}}.
\end{eqnarray*}%
As a consequence, the characteristic $\widehat{A}_{p,q}\left( w,v\right) $
defined in (\ref{two weight Apq hat}) can be rewritten as%
\begin{equation}
\widetilde{A}_{p,q}\left( w,v\right) =\widehat{A}_{p,q}\left(
w_{1},v_{1}\right) \widehat{A}_{p,q}\left( w_{2},v_{2}\right) .
\label{rewrite}
\end{equation}

From the two weight theorem of Sawyer and Wheeden above, Theorem \ref{Saw
and Wheed}, we conclude that (\ref{two T1}) holds with $T_{1}g=\Omega
_{\alpha }^{m}\ast g$ and constant $C_{1}=C\widehat{A}_{p,q}\left(
w_{1},v_{1}\right) $, and that (\ref{two T2}) holds with $T_{2}h=\Omega
_{\beta }^{n}\ast h$ and constant $C_{2}=C\widehat{A}_{p,q}\left(
w_{2},v_{2}\right) $. This precise dependence on $\widehat{A}_{p,q}\left(
w_{2},v_{2}\right) $ is not explicitly stated in \ref{Saw and Wheed}, but it
is easily checked by tracking the constants in the proof given there. See
also the detailed proof in the appendix below. Then the conclusion of the
theorem proves the norm inequality (\ref{norm two weight}), and also the
equivalence (\ref{characteristic equiv}), in view of (\ref{rewrite}).
\end{proof}

\begin{remark}
If we restrict the function $f$ in the norm inequality in the corollary
above to be a product function $f\left( x,y\right) =f_{1}\left( x\right)
f_{2}\left( y\right) $, then $\left( \Omega _{\alpha ,\beta }^{m,n}\ast
f\right) \left( x,y\right) $ is the product function $\Omega _{\alpha
}^{m}\ast f_{1}\left( x\right) \ \Omega _{\beta }^{n}\ast f_{2}\left(
y\right) $, and there is a particularly trivial proof of the norm bound for
such $f$:%
\begin{eqnarray*}
&&\left\{ \int_{\mathbb{R}^{m}}\int_{\mathbb{R}^{n}}\left( \Omega _{\alpha
,\beta }^{m,n}\ast f\right) \left( x,y\right) ^{q}\ w\left( x,y\right) ^{q}\
dxdy\right\} ^{\frac{1}{q}} \\
&=&\left\{ \int_{\mathbb{R}^{m}}\Omega _{\alpha }^{m}\ast f_{1}\left(
x\right) ^{q}\ w_{1}\left( x\right) ^{q}\ dx\right\} ^{\frac{1}{q}}\left\{
\int_{\mathbb{R}^{n}}\Omega _{\beta }^{n}\ast f_{2}\left( y\right) ^{q}\
w_{2}\left( y\right) ^{q}\ dy\right\} ^{\frac{1}{q}} \\
&\leq &C_{1}\left\{ \int_{\mathbb{R}^{m}}f_{1}\left( x\right) ^{p}\
w_{1}\left( x\right) ^{p}\ dx\right\} ^{\frac{1}{p}}C_{2}\left\{ \int_{%
\mathbb{R}^{n}}f_{2}\left( y\right) ^{p}\ w_{2}\left( y\right) ^{p}\
dy\right\} ^{\frac{1}{p}} \\
&=&C_{1}C_{2}\left\{ \int_{\mathbb{R}^{m}}\int_{\mathbb{R}^{n}}f\left(
x,y\right) ^{p}\ w\left( x,y\right) ^{p}\ dxdy\right\} ^{\frac{1}{p}}.
\end{eqnarray*}
\end{remark}

\begin{remark}
\label{pointwise}We have the following poinwise limit for $w_{1}$ above:%
\begin{equation}
\lim_{I\rightarrow x_{0}}\left( \frac{1}{\left\vert I\right\vert }\int_{%
\mathbb{R}^{m}}\left( \widehat{s}_{I}w_{1}\right) \left( x\right) ^{q}\
dx\right) ^{\frac{1}{q}}=Cw_{1}\left( x_{0}\right) ,\ \ \ \ \ \text{for a.e. 
}x_{0}\in \mathbb{R}^{m}.  \label{limit}
\end{equation}%
Indeed, with $I_{0}\equiv \left[ -\frac{1}{2},\frac{1}{2}\right] $, the
function 
\begin{equation*}
\widehat{s}_{I_{0}}\left( x\right) ^{q}=\left( 1+\left\vert x\right\vert
\right) ^{-\left( m-\alpha \right) q}=\left( 1+\left\vert x\right\vert
\right) ^{-m-m\frac{q}{p^{\prime }}},\ \ \ \ \ x\in \mathbb{R}^{m},
\end{equation*}%
is such that the family $\left\{ \frac{1}{r^{m}}\widehat{s}%
_{I_{0}}^{q}\left( \frac{\cdot }{r}\right) \right\} _{r>0}$ is an
approximate identity on $\mathbb{R}^{m}$, and thus (\ref{limit}) holds at
every Lebesgue point of $w_{1}^{q}$, since $w_{1}^{q}$ obviously satisfies
the growth condition,%
\begin{equation}
\int_{\mathbb{R}^{m}}\left( 1+\left\vert x\right\vert \right) ^{-m\left( 1+%
\frac{q}{p^{\prime }}\right) }w_{1}\left( x\right) ^{q}dx=\int_{\mathbb{R}%
^{m}}\left( 1+\left\vert x\right\vert \right) ^{-q\left( m-\alpha \right)
}w_{1}\left( x\right) ^{q}dx\leq C<\infty ,  \label{growth}
\end{equation}%
when $v_{1}$ is not the trivial weight identically infinity. Similar
pointwise limits hold for the remaining three functions $w_{2}$, $v_{1}$ and 
$v_{2}$.
\end{remark}

\section{Proof of Theorem \protect\ref{original A1} and Example \protect\ref%
{half example}}

We begin with this easy lemma.

\begin{lemma}
\label{easy}Suppose that $1<p,q<\infty $ and $0<\frac{\alpha }{m},\frac{%
\beta }{n}<1$ satisfy the half-balanced condition%
\begin{equation*}
\frac{1}{p}-\frac{1}{q}=\min \left\{ \frac{\alpha }{m},\frac{\beta }{n}%
\right\} .
\end{equation*}%
If $u$ is a positive weight on $\mathbb{R}^{m}\times \mathbb{R}^{n}$, then 
\begin{equation*}
A_{p,q}\left( u\right) \leq \min \left\{ \left\Vert u^{q}\right\Vert
_{A_{1}\times A_{1}},\left\Vert u^{-p^{\prime }}\right\Vert _{A_{1}\times
A_{1}}\right\} .
\end{equation*}
\end{lemma}

\begin{proof}[Proof of Lemma \protect\ref{easy}]
Suppose that $u^{-p^{\prime }}\in A_{1}\times A_{1}$. Then 
\begin{equation*}
\frac{1}{\left\vert I\times J\right\vert }\diint\limits_{I\times
J}u^{-p^{\prime }}\leq \left\Vert u^{-p^{\prime }}\right\Vert _{A_{1}\times
A_{1}}\ \left( \inf_{I\times J}u^{-p^{\prime }}\right)
\end{equation*}%
for all rectangles $I\times J$, and so%
\begin{eqnarray*}
&&\left( \frac{1}{\left\vert I\times J\right\vert }\diint\limits_{I\times
J}u\left( x,y\right) ^{q}\ dxdy\right) ^{\frac{1}{q}}\left( \frac{1}{%
\left\vert I\times J\right\vert }\diint\limits_{I\times J}u\left( x,y\right)
^{-p^{\prime }}\ dxdy\right) ^{\frac{1}{p^{\prime }}} \\
&\leq &\left( \frac{1}{\left\vert I\times J\right\vert }\diint\limits_{I%
\times J}u\left( x,y\right) ^{q}\ dxdy\right) ^{\frac{1}{q}}\left\Vert
u^{-p^{\prime }}\right\Vert _{A_{1}\times A_{1}}^{\frac{1}{p^{\prime }}%
}\left( \inf_{I\times J}v^{-1}\right) \\
&=&\left\Vert u^{-p^{\prime }}\right\Vert _{A_{1}\times A_{1}}^{\frac{1}{%
p^{\prime }}}\left( \frac{1}{\left\vert I\times J\right\vert }%
\diint\limits_{I\times J}\left( \frac{u\left( x,y\right) }{\sup_{I\times J}u}%
\right) ^{q}\ dxdy\right) ^{\frac{1}{q}}\leq \left\Vert u^{-p^{\prime
}}\right\Vert _{A_{1}\times A_{1}}^{\frac{1}{p^{\prime }}}\ ,
\end{eqnarray*}%
which proves the second assertion for $u^{-p^{\prime }}$. Similarly, if $%
u^{q}\in A_{1}\times A_{1}$, then%
\begin{eqnarray*}
&&\left( \frac{1}{\left\vert I\times J\right\vert }\diint\limits_{I\times
J}u\left( x,y\right) ^{q}\ dxdy\right) ^{\frac{1}{q}}\left( \frac{1}{%
\left\vert I\times J\right\vert }\diint\limits_{I\times J}u\left( x,y\right)
^{-p^{\prime }}\ dxdy\right) ^{\frac{1}{p^{\prime }}} \\
&\leq &\left\Vert u^{q}\right\Vert _{A_{1}\times A_{1}}^{\frac{1}{q}}\left(
\inf_{I\times J}u\right) \left( \frac{1}{\left\vert I\times J\right\vert }%
\diint\limits_{I\times J}u\left( x,y\right) ^{-p^{\prime }}\ dxdy\right) ^{%
\frac{1}{p^{\prime }}} \\
&=&\left\Vert u^{q}\right\Vert _{A_{1}\times A_{1}}^{\frac{1}{q}}\left( 
\frac{1}{\left\vert I\times J\right\vert }\diint\limits_{I\times J}\left( 
\frac{\inf_{I\times J}u}{u\left( x,y\right) }\right) ^{p^{\prime }}\
dxdy\right) ^{\frac{1}{p^{\prime }}}\leq \left\Vert u^{q}\right\Vert
_{A_{1}\times A_{1}}^{\frac{1}{q}}\ ,
\end{eqnarray*}%
and this completes the proof of Lemma \ref{easy}.
\end{proof}

\subsection{Theorem \protect\ref{original A1}}

We can now prove Theorem \ref{original A1} easily using Lemma \ref{easy}.

\begin{proof}[Proof of Theorem \protect\ref{original A1}]
Assume now that $v^{-p^{\prime }}\in A_{1}\times A_{1}$ and $\frac{1}{p}-%
\frac{1}{q}=\frac{\alpha }{m}<\frac{\beta }{n}$, the other cases being
similar. Then we have%
\begin{equation*}
A_{p,q}^{\left( \alpha ,\beta \right) ,\left( m,n\right) }\left( v,w\right)
\equiv \sup_{I\subset \mathbb{R}^{m},\ J\subset \mathbb{R}^{n}}\left\vert
J\right\vert ^{\frac{\beta }{n}-1}\left( \frac{1}{\left\vert I\right\vert }%
\diint\limits_{I\times J}w^{q}\right) ^{\frac{1}{q}}\left( \frac{1}{%
\left\vert I\right\vert }\diint\limits_{I\times J}v^{-p^{\prime }}\right) ^{%
\frac{1}{p^{\prime }}},
\end{equation*}%
and if we let $I$ shrink to a point $x\in \mathbb{R}^{m}$, we obtain the $1$%
-parameter conclusion that%
\begin{equation*}
\left( w_{x}\left( y\right) ,v_{x}\left( y\right) \right) \in A_{p,q}^{\beta
,n}\ ,\ \ \ \ \ \text{uniformly a.e. }x\in \mathbb{R}^{m}.
\end{equation*}%
Then Minkowski's inequality with $p\leq q$ gives%
\begin{eqnarray*}
&&\left\{ \int_{\mathbb{R}^{m}}\int_{\mathbb{R}^{n}}I_{\alpha ,\beta
}^{m,n}f\left( x,y\right) ^{q}\ w\left( x,y\right) ^{q}\ dxdy\right\} ^{%
\frac{1}{q}} \\
&=&\left\{ \int_{\mathbb{R}^{m}}\left[ \int_{\mathbb{R}^{n}}f\ast _{1}\Omega
_{\alpha }^{m}\ast _{2}\Omega _{\beta }^{n}\left( x,y\right) ^{q}\ w\left(
x,y\right) ^{q}\ dy\right] dx\right\} ^{\frac{1}{q}} \\
&\leq &A_{p,q}^{\beta ,n}\left\{ \int_{\mathbb{R}^{m}}\left[ \int_{\mathbb{R}%
^{n}}f\ast _{1}\Omega _{\alpha }^{m}\left( x,y\right) ^{p}\ v\left(
x,y\right) ^{p}\ dy\right] ^{\frac{q}{p}}dx\right\} ^{\frac{1}{q}} \\
&\leq &A_{p,q}^{\beta ,n}\left( v,w\right) \left\{ \int_{\mathbb{R}^{m}}%
\left[ \int_{\mathbb{R}^{n}}f\ast _{1}\Omega _{\alpha }^{m}\left( x,y\right)
^{q}\ v\left( x,y\right) ^{q}\ dy\right] ^{\frac{p}{q}}dx\right\} ^{\frac{1}{%
p}} \\
&\leq &A_{p,q}^{\beta ,n}\left( v,w\right) \ A_{1}\times A_{1}\left(
v\right) \ \left\{ \int_{\mathbb{R}^{m}}\int_{\mathbb{R}^{n}}f\left(
x,y\right) ^{p}\ v\left( x,y\right) ^{p}\ dxdy\right\} ^{\frac{1}{p}},
\end{eqnarray*}%
where in the final line we have applied the $1$-parameter one weight
inequality in Theorem \ref{Muck and Wheed} since $v\in A_{p,q}$ by Lemma \ref%
{easy}. This completes the proof of Theorem \ref{original A1}.
\end{proof}

\subsection{Example \protect\ref{half example}}

Now we verify the properties claimed in Example \ref{half example}. First we
note the following discretization of $\frac{1}{\left\vert I\right\vert }%
\int_{\mathbb{R}^{m}}\widehat{s}_{I}^{p^{\prime }}d\sigma $: 
\begin{eqnarray*}
&&\sum_{k=0}^{\infty }2^{k\left[ \left( \alpha -m\right) p^{\prime }+m\right]
}\frac{1}{\left\vert 2^{k}I\right\vert }\int_{2^{k}I}d\sigma \\
&=&\sum_{k=0}^{\infty }2^{k\left[ \left( \alpha -m\right) p^{\prime }+m%
\right] }\frac{1}{\left\vert 2^{k}I\right\vert }\sum_{\ell
=0}^{k}\int_{2^{\ell }I\setminus 2^{\ell -1}I}d\sigma =\sum_{\ell
=0}^{\infty }\left( \sum_{k=\ell }^{\infty }2^{k\left( \alpha -m\right)
p^{\prime }}\right) \int_{2^{\ell }I\setminus 2^{\ell -1}I}d\sigma \\
&\approx &\sum_{\ell =0}^{\infty }2^{\ell \left( \alpha -m\right) p^{\prime
}}\int_{2^{\ell }I\setminus 2^{\ell -1}I}d\sigma \left( x\right) \approx 
\frac{1}{\left\vert I\right\vert }\int_{\mathbb{R}^{m}}\left( 1+\frac{%
\left\vert x-c_{I}\right\vert }{\left\vert I\right\vert ^{\frac{1}{m}}}%
\right) ^{\left( \alpha -m\right) p^{\prime }}d\sigma \left( x\right) =\frac{%
1}{\left\vert I\right\vert }\int_{\mathbb{R}^{m}}\widehat{s}_{I}^{p^{\prime
}}d\sigma ,
\end{eqnarray*}%
since $\alpha <m$. Thus we have 
\begin{eqnarray*}
\overline{\mathbb{A}}_{p,q}^{\alpha ,m}\left( \sigma ,\omega ;I\right)
^{p^{\prime }} &\equiv &\left\{ \left\vert I\right\vert ^{\frac{\alpha }{m}%
-\Gamma }\left( \frac{1}{\left\vert I\right\vert }\int_{I}d\omega \right) ^{%
\frac{1}{q}}\left( \frac{1}{\left\vert I\right\vert }\int_{\mathbb{R}^{m}}%
\widehat{s}_{I}^{p^{\prime }}d\sigma \right) ^{\frac{1}{p^{\prime }}%
}\right\} ^{p^{\prime }} \\
&\approx &\left\vert I\right\vert ^{\left( \frac{\alpha }{m}-\Gamma \right)
p^{\prime }}\left( \frac{1}{\left\vert I\right\vert }\int_{I}d\omega \right)
^{\frac{p^{\prime }}{q}}\sum_{k=0}^{\infty }2^{k\left[ \left( \alpha
-m\right) p^{\prime }+m\right] }\frac{1}{\left\vert 2^{k}I\right\vert }%
\int_{2^{k}I}d\sigma \left( x\right) \\
&=&\sum_{k=0}^{\infty }2^{-k\left( \alpha -m\Gamma \right) p^{\prime
}}\left\vert 2^{k}I\right\vert ^{\left( \frac{\alpha }{m}-\Gamma \right)
p^{\prime }}2^{k\left[ \left( \alpha -m\right) p^{\prime }+m\right] }2^{km%
\frac{p^{\prime }}{q}}\left( \frac{1}{\left\vert 2^{k}I\right\vert }%
\int_{I}d\omega \right) ^{\frac{p^{\prime }}{q}}\left( \frac{1}{\left\vert
2^{k}I\right\vert }\left\vert 2^{k}I\right\vert _{\sigma }\right) \\
&=&\sum_{k=0}^{\infty }2^{-k\left( \alpha -m\Gamma \right) p^{\prime }}2^{k 
\left[ \left( \alpha -m\right) p^{\prime }+m\right] }2^{km\frac{p^{\prime }}{%
q}}\left\{ \left\vert 2^{k}I\right\vert ^{\left( \frac{\alpha }{m}-\Gamma
\right) p^{\prime }}\left( \frac{1}{\left\vert 2^{k}I\right\vert }%
\int_{I}d\omega \right) ^{\frac{p^{\prime }}{q}}\left( \frac{1}{\left\vert
2^{k}I\right\vert }\left\vert 2^{k}I\right\vert _{\sigma }\right) \right\} \\
&=&\sum_{k=0}^{\infty }\left\{ \left\vert 2^{k}I\right\vert ^{\left( \frac{%
\alpha }{m}-\Gamma \right) }\left( \frac{1}{\left\vert 2^{k}I\right\vert }%
\int_{I}d\omega \right) ^{\frac{1}{q}}\left( \frac{1}{\left\vert
2^{k}I\right\vert }\left\vert 2^{k}I\right\vert _{\sigma }\right) ^{\frac{1}{%
p^{\prime }}}\right\} ^{p^{\prime }}.
\end{eqnarray*}%
Now recall that 
\begin{equation*}
v\left( y\right) =\left\vert y\right\vert ^{-\frac{m}{q}}\text{ and }w\left(
x\right) =\left( 1+\left\vert x\right\vert \right) ^{-m},
\end{equation*}%
so that $d\omega \left( x\right) =w\left( x\right) ^{q}dx=\left(
1+\left\vert x\right\vert \right) ^{-mp^{\prime }}dx$ and $d\sigma \left(
y\right) =v\left( y\right) ^{-p^{\prime }}dy=\left\vert y\right\vert ^{\frac{%
mp^{\prime }}{q}}dy$. Then for $Q=\left[ -R,R\right] ^{m}$ we have%
\begin{eqnarray}
\mathbb{A}_{p,q}^{\alpha ,m}\left( \sigma ,\omega \right) \left[ Q\right]
&\approx &\left\vert Q\right\vert ^{\frac{\alpha }{m}-\Gamma }\left( \frac{1%
}{\left\vert Q\right\vert }\int_{Q}\left( 1+\left\vert x\right\vert \right)
^{-mp^{\prime }}dx\right) ^{\frac{1}{q}}\left( \frac{1}{\left\vert
Q\right\vert }\int_{Q}\left\vert y\right\vert ^{\frac{mp^{\prime }}{q}%
}dy\right) ^{\frac{1}{p^{\prime }}}  \label{centered} \\
&\approx &R^{\alpha -m\Gamma -m\frac{1}{q}-m\frac{1}{p^{\prime }}}\left(
\int_{0}^{R}\left( 1+r\right) ^{-\frac{mp^{\prime }}{q}}r^{m-1}dr\right) ^{%
\frac{1}{q}}\left( \int_{0}^{R}r^{\frac{mp^{\prime }}{q}+m-1}dr\right) ^{%
\frac{1}{p^{\prime }}}  \notag \\
&\approx &R^{\alpha -m}\min \left\{ R,1\right\} ^{m}R^{\frac{m}{q}+\frac{m}{%
p^{\prime }}}  \notag \\
&=&R^{\alpha +\frac{m}{q}-\frac{m}{p}}\min \left\{ R,1\right\} ^{m}\leq
C<\infty  \notag
\end{eqnarray}%
since $\alpha +\beta =\alpha +\frac{m}{q}-\frac{m}{p}=0$. Now if $Q=z+\left[
-R,R\right] ^{m}$ with $R\leq \frac{1}{10}\left\vert z\right\vert $, we have%
\begin{eqnarray*}
\mathbb{A}_{p,q}^{\alpha ,m}\left( \sigma ,\omega \right) \left[ Q\right]
&=&\left\vert Q\right\vert ^{\frac{\alpha }{m}-\Gamma }\left( \frac{1}{%
\left\vert Q\right\vert }\int_{Q}\left( 1+\left\vert x\right\vert \right)
^{-mp^{\prime }}dx\right) ^{\frac{1}{q}}\left( \frac{1}{\left\vert
Q\right\vert }\int_{Q}\left\vert y\right\vert ^{\frac{mp^{\prime }}{q}%
}dy\right) ^{\frac{1}{p^{\prime }}} \\
&\approx &R^{\alpha -m\Gamma }\left( 1+\left\vert z\right\vert \right)
^{-m}\left\vert z\right\vert ^{\frac{m}{q}}=\left( 1+\left\vert z\right\vert
\right) ^{-m}\left\vert z\right\vert ^{\frac{m}{q}}\leq C<\infty
\end{eqnarray*}%
since $0\leq \frac{m}{q}\leq m$. On the other hand, if $Q=z+\left[ -R,R%
\right] ^{m}$ with $R>\frac{1}{10}\left\vert z\right\vert $, then since $%
d\sigma \left( y\right) =\left\vert y\right\vert ^{\frac{mp^{\prime }}{q}}dy$
is doubling and $d\omega \left( x\right) =\left( 1+\left\vert x\right\vert
\right) ^{-mp^{\prime }}dx$ is radially decreasing, we have%
\begin{eqnarray*}
\mathbb{A}_{p,q}^{\alpha ,m}\left( \sigma ,\omega \right) \left[ Q\right]
&=&\left\vert Q\right\vert ^{\frac{\alpha }{m}-\Gamma }\left( \frac{1}{%
\left\vert Q\right\vert }\int_{Q}\left( 1+\left\vert x\right\vert \right)
^{-mp^{\prime }}dx\right) ^{\frac{1}{q}}\left( \frac{1}{\left\vert
Q\right\vert }\int_{Q}\left\vert y\right\vert ^{\frac{mp^{\prime }}{q}%
}dy\right) ^{\frac{1}{p^{\prime }}} \\
&\approx &\left\vert \left[ -R,R\right] ^{m}\right\vert ^{\frac{\alpha }{m}%
-\Gamma }\left( \frac{1}{\left\vert Q\right\vert }\int_{Q}\left(
1+\left\vert x\right\vert \right) ^{-mp^{\prime }}dx\right) ^{\frac{1}{q}%
}\left( \frac{1}{\left[ -R,R\right] ^{m}}\int_{\left[ -R,R\right]
^{m}}\left\vert y\right\vert ^{\frac{mp^{\prime }}{q}}dy\right) ^{\frac{1}{%
p^{\prime }}} \\
&\lesssim &\left\vert \left[ -R,R\right] ^{m}\right\vert ^{\frac{\alpha }{m}%
-\Gamma }\left( \frac{1}{\left[ -R,R\right] ^{m}}\int_{\left[ -R,R\right]
^{m}}\left( 1+\left\vert x\right\vert \right) ^{-mp^{\prime }}dx\right) ^{%
\frac{1}{q}}\left( \frac{1}{\left[ -R,R\right] ^{m}}\int_{\left[ -R,R\right]
^{m}}\left\vert y\right\vert ^{\frac{mp^{\prime }}{q}}dy\right) ^{\frac{1}{%
p^{\prime }}}
\end{eqnarray*}%
which by (\ref{centered}) is bounded by $C$.

On the other hand, with $I=\left[ -1,1\right] ^{m}$, we have%
\begin{eqnarray*}
&&\left\vert 2^{k}I\right\vert ^{\left( \frac{\alpha }{m}-\Gamma \right)
}\left( \frac{1}{\left\vert 2^{k}I\right\vert }\int_{I}d\omega \right) ^{%
\frac{1}{q}}\left( \frac{1}{\left\vert 2^{k}I\right\vert }\left\vert
2^{k}I\right\vert _{\sigma }\right) ^{\frac{1}{p^{\prime }}} \\
&\approx &\left\vert 2^{k}I\right\vert ^{\left( \frac{\alpha }{m}-\Gamma
\right) }\left( \frac{1}{\left\vert 2^{k}I\right\vert }\int_{2^{k}I}d\omega
\right) ^{\frac{1}{q}}\left( \frac{1}{\left\vert 2^{k}I\right\vert }%
\left\vert 2^{k}I\right\vert _{\sigma }\right) ^{\frac{1}{p^{\prime }}} \\
&\approx &2^{k\left[ \alpha -m\left( \frac{1}{p}-\frac{1}{q}\right) \right]
}\ 2^{-km\left[ \frac{1}{q}+\frac{1}{p^{\prime }}\right] }\left(
\int_{2^{k}I}d\sigma \right) ^{\frac{1}{p^{\prime }}}=2^{k\left( \alpha
-m\right) }\left( \int_{2^{k}I}d\sigma \right) ^{\frac{1}{p^{\prime }}}.
\end{eqnarray*}%
Thus we have 
\begin{eqnarray*}
\overline{\mathbb{A}}_{p,q}^{\alpha ,m}\left( \sigma ,\omega \right)
^{p^{\prime }} &\gtrsim &\sum_{k=0}^{\infty }\left\{ 2^{k\left( \alpha
-m\right) }\left( \int_{2^{k}I}d\sigma \right) ^{\frac{1}{p^{\prime }}%
}\right\} ^{p^{\prime }}=\sum_{k=0}^{\infty }2^{kp^{\prime }\left( \alpha
-m\right) }\left( \int_{2^{k}I}\left\vert y\right\vert ^{\frac{mp^{\prime }}{%
q}}dy\right) \\
&\approx &\sum_{k=0}^{\infty }2^{kp^{\prime }\left( \alpha -m\right) }\
2^{k\left( \frac{mp^{\prime }}{q}+m\right) }=\sum_{k=0}^{\infty
}2^{kp^{\prime }\left( \alpha +\frac{m}{q}-\frac{m}{p}\right) }=\infty
\end{eqnarray*}%
since $\alpha +\frac{m}{q}-\frac{m}{p}=0$. So we have that $A_{p,q}^{\alpha
,m}\left( v,w\right) <\infty $ and $\overline{A}_{p,q}^{\alpha ,m}\left(
v,w\right) =\infty $.

Finally note that $v^{-p^{\prime }}\in A_{p^{\prime }}$ (which is equivalent
to $v^{p}\in A_{p}$) since%
\begin{equation*}
-m<\frac{mp^{\prime }}{q}<m\left( p^{\prime }-1\right) \text{, i.e. }-\frac{q%
}{p^{\prime }}<1<\frac{q}{p}.
\end{equation*}

\section{Proof of Theorem \protect\ref{power weights}\label{Section power
weights}}

Here we give the proof of our main positive two weight result, Theorem \ref%
{power weights}, beginning with the necessity of (\ref{p at most q}) and (%
\ref{finiteness of A}) for the norm inequality (\ref{pwni}), i.e. the proof
of $\left( 1\right) \Longrightarrow \left( 2\right) $. The necessity of (\ref%
{p at most q}) follows from the inequality 
\begin{equation*}
\left\vert x-u\right\vert ^{\alpha -m}\left\vert y-t\right\vert ^{\beta
-n}\geq \left\vert \left( x-u,y-t\right) \right\vert ^{\alpha +\beta -\left(
m+n\right) },
\end{equation*}%
which shows that the $1$-parameter fractional integral $I_{\alpha +\beta
}^{m+n}f$ is dominated by the $2$-parameter fractional integral $I_{\alpha
,\beta }^{m,n}f$ when $f$ is nonnegative. Thus Theorem \ref{Stein-Weiss
sharp} shows that $p\leq q$. Local integrability of the kernel is necessary
for the norm inequality, and so $\alpha ,\beta >0$. The necessity of
finiteness of the two-tailed Muckenhoupt characteristic $\widehat{A}%
_{p,q}^{\left( \alpha ,\beta \right) ,\left( m,n\right) }\left( v_{\delta
},w_{\gamma }\right) $ now follows from Proposition \ref{tails} at the end
of the appendix, and then we use $A_{p,q}^{\left( \alpha ,\beta \right)
,\left( m,n\right) }\left( v_{\delta },w_{\gamma }\right) \leq \widehat{A}%
_{p,q}^{\left( \alpha ,\beta \right) ,\left( m,n\right) }\left( v_{\delta
},w_{\gamma }\right) $.\newline

Now we turn to proving $\left( 2\right) \Longrightarrow \left( 3\right) $,
i.e. that (\ref{p at most q}) and (\ref{finiteness of A}) imply the
conditions (\ref{Formula}), (\ref{gamma+delta}) and (\ref{star'}) on
indices, namely%
\begin{eqnarray}
&&{\frac{1}{p}}-{\frac{1}{q}}+{\frac{{\gamma +\delta }}{m+n}}={\frac{\alpha
+\beta }{m+n},}  \label{recall'} \\
&&{\gamma +\delta }\geq 0,  \notag \\
&&\beta -\frac{n}{p}<\delta \text{ and }\alpha -\frac{m}{p}<\delta \text{
when }\gamma \geq 0\geq \delta ,  \notag \\
&&\beta -\frac{n}{q^{\prime }}<\gamma \text{ and }\alpha -\frac{m}{q^{\prime
}}<\gamma \text{ when }\delta \geq 0\geq \gamma .  \notag
\end{eqnarray}%
So assume both (\ref{p at most q}) and (\ref{finiteness of A}). We begin
with the necessity of the equality in the top line of (\ref{recall'}). This
follows immediately from the calculation,%
\begin{eqnarray*}
&&\left\vert tI\right\vert ^{\frac{\alpha }{m}-1}\left\vert tJ\right\vert ^{%
\frac{\beta }{n}-1}\left( \diint\limits_{tI\times tJ}\left\vert \left(
x,y\right) \right\vert ^{-\gamma q}dxdy\right) ^{\frac{1}{q}}\left(
\diint\limits_{tI\times tJ}\left\vert \left( x,y\right) \right\vert
^{-\delta p^{\prime }}dxdy\right) ^{\frac{1}{p^{\prime }}} \\
&&t^{\alpha -m}t^{\beta -n}t^{-\gamma +\frac{m+n}{q}}t^{-\delta +\frac{m+n}{%
p^{\prime }}}\left\vert I\right\vert ^{\frac{\alpha }{m}-1}\left\vert
J\right\vert ^{\frac{\beta }{n}-1}\left( \diint\limits_{I\times J}\left\vert
\left( x,y\right) \right\vert ^{-\gamma q}dxdy\right) ^{\frac{1}{q}}\left(
\diint\limits_{I\times J}\left\vert \left( x,y\right) \right\vert ^{-\delta
p^{\prime }}dxdy\right) ^{\frac{1}{p^{\prime }}}
\end{eqnarray*}%
which implies $\alpha -m+\beta -n-\gamma +\frac{m+n}{q}-\delta +\frac{m+n}{%
p^{\prime }}=0$. Next we note that power weights have support equal to the
entire Euclidean space, and so Remark \ref{disjoint support} shows that%
\begin{equation*}
\frac{\alpha }{m},\frac{\beta }{n}\geq \frac{1}{p}-\frac{1}{q},
\end{equation*}%
and combining this with (\ref{Formula}) gives the second line in (\ref%
{recall'}). Finally, we turn to proving the necessity of the third and
fourth lines in (\ref{recall'}).

For this, we consider $P\times Q$ to be centered at the origin. Define the
truncated cubes $P^{\varepsilon }=P\setminus \{|x|<\varepsilon \}\subset 
\mathbb{R}^{m}$ and $Q^{\varepsilon }=Q\setminus \{|y|<\varepsilon \}\subset 
\mathbb{R}^{n}$ for some $\varepsilon >0$. Let $0<{\lambda }<1$. We set $%
|Q|^{\frac{1}{n}}=1$ and $|P|^{\frac{1}{m}}={\lambda }$. Suppose that ${%
\frac{\alpha }{m}}-\left( {\frac{1}{p}}-{\frac{1}{q}}\right) >0$. Then we
have%
\begin{eqnarray}
&&  \label{A_pq Decay Est} \\
&&\lim_{{\lambda }\rightarrow 0}|P|^{{\frac{\alpha }{m}}-\left( {\frac{1}{p}}%
-{\frac{1}{q}}\right) }|Q|^{{\frac{\beta }{n}}-\left( {\frac{1}{p}}-{\frac{1%
}{q}}\right) }\left\{ {\frac{1}{|P||Q|}}\iint_{P\times Q^{\varepsilon
}}\left( {\frac{1}{|x|+|y|}}\right) ^{{\gamma }q}dxdy\right\} ^{\frac{1}{q}%
}\left\{ {\frac{1}{|P||Q|}}\int_{P\times Q^{\varepsilon }}\left( {\frac{1}{%
|x|+|y|}}\right) ^{\delta \left( {\frac{p}{p-1}}\right) }dxdy\right\} ^{%
\frac{p-1}{p}}  \notag \\
&=&\lim_{{\lambda }\rightarrow 0}{\lambda }^{\alpha -m\left( {\frac{1}{p}}-{%
\frac{1}{q}}\right) }\left\{ {\frac{1}{|P||Q|}}\iint_{P\times Q^{\varepsilon
}}\left( {\frac{1}{|x|+|y|}}\right) ^{{\gamma }q}dxdy\right\} ^{\frac{1}{q}%
}\left\{ {\frac{1}{|P||Q|}}\int_{P\times Q^{\varepsilon }}\left( {\frac{1}{%
|x|+|y|}}\right) ^{\delta \left( {\frac{p}{p-1}}\right) }dxdy\right\} ^{%
\frac{p-1}{p}}  \notag \\
&=&\left( \lim_{{\lambda }\rightarrow 0}{\lambda }^{\alpha -m\left( {\frac{1%
}{p}}-{\frac{1}{q}}\right) }\right) \left\{ \int_{Q^{\varepsilon }}\left( {%
\frac{1}{|y|}}\right) ^{{\gamma }q}dy\right\} ^{\frac{1}{q}}\left\{
\int_{Q^{\varepsilon }}\left( {\frac{1}{|y|}}\right) ^{\delta \left( {\frac{p%
}{p-1}}\right) }dy\right\} ^{\frac{p-1}{p}}=0\text{ for every }\varepsilon
>0.  \notag
\end{eqnarray}

\textbf{Case 1:} Suppose ${\gamma }>0,\delta \leq 0$. Let $|Q|^{\frac{1}{n}%
}=1$ and $|P|^{\frac{1}{m}}={\lambda }$. Suppose $\alpha -m\left( {\frac{1}{p%
}}-{\frac{1}{q}}\right) =0$. Then we have%
\begin{eqnarray}
&&  \label{Characteristic Est1 Case1} \\
&&|P|^{{\frac{\alpha }{m}}-\left( {\frac{1}{p}}-{\frac{1}{q}}\right) }|Q|^{{%
\frac{\beta }{n}}-\left( {\frac{1}{p}}-{\frac{1}{q}}\right) }\left\{ {\frac{1%
}{|P||Q|}}\iint_{P\times Q^{\varepsilon }}\left( {\frac{1}{|x|+|y|}}\right)
^{{\gamma }q}dxdy\right\} ^{\frac{1}{q}}\left\{ {\frac{1}{|P||Q|}}%
\int_{P\times Q^{\varepsilon }}\left( {\frac{1}{|x|+|y|}}\right) ^{\delta
\left( {\frac{p}{p-1}}\right) }dxdy\right\} ^{\frac{p-1}{p}}  \notag \\
&\gtrsim &\left\{ \int_{Q}\left( {\frac{1}{{\lambda }+|y|}}\right) ^{{\gamma 
}q}dy\right\} ^{\frac{1}{q}}\left\{ \int_{Q}\left( {\frac{1}{|y|}}\right)
^{\delta \left( {\frac{p}{p-1}}\right) }dy\right\} ^{\frac{p-1}{p}}\gtrsim
\left\{ \int_{{\lambda }<|y|\leq 1}\left( {\frac{1}{{\lambda }+|y|}}\right)
^{{\gamma }q}dy\right\} ^{\frac{1}{q}}.  \notag
\end{eqnarray}
A direct computation shows 
\begin{equation}
\int_{{\lambda }<|y|\leq 1}\left( {\frac{1}{{\lambda }+|y|}}\right) ^{{%
\gamma }q}dy\approx \ln \left( {\frac{1+{\lambda }}{2{\lambda }}}\right) 
\text{\ if }{\gamma }={\frac{n}{q}}  \label{Necessary Case1 Est1}
\end{equation}%
and 
\begin{equation}
\int_{{\lambda }<|x_{i}|\leq 1}\left( {\frac{1}{{\lambda }+|y|}}\right) ^{{%
\gamma }q}dy\approx \left( {\frac{1}{2{\lambda }}}\right) ^{{\gamma }%
q-n}-\left( {\frac{1}{{\lambda }+1}}\right) ^{{\gamma }q-n}\text{\ if }{%
\gamma }>{\frac{n}{q}}.  \label{Necessary Case1 Est2}
\end{equation}%
Using (\ref{Characteristic Est1 Case1}), (\ref{Necessary Case1 Est1}) and (%
\ref{Necessary Case1 Est2}), and letting ${\lambda }\rightarrow 0$, we
obtain that 
\begin{equation}
{\gamma }<{\frac{n}{q}}.  \label{Necessity est1}
\end{equation}%
On the other hand, suppose that $\alpha -m\left( {\frac{1}{p}}-{\frac{1}{q}}%
\right) >0$. Then we have%
\begin{eqnarray}
&&  \label{Characteristic Est2 Case1} \\
&&|P|^{{\frac{\alpha }{m}}-\left( {\frac{1}{p}}-{\frac{1}{q}}\right) }|Q|^{{%
\frac{\beta }{n}}-\left( {\frac{1}{p}}-{\frac{1}{q}}\right) }\left\{ {\frac{1%
}{|P||Q|}}\iint_{P\times Q^{\varepsilon }}\left( {\frac{1}{|x|+|y|}}\right)
^{{\gamma }q}dxdy\right\} ^{\frac{1}{q}}\left\{ {\frac{1}{|P||Q|}}%
\int_{P\times Q^{\varepsilon }}\left( {\frac{1}{|x|+|y|}}\right) ^{\delta
\left( {\frac{p}{p-1}}\right) }dxdy\right\} ^{\frac{p-1}{p}}  \notag \\
&\gtrsim &\left( {\lambda }\right) ^{\alpha -m\left( {\frac{1}{p}}-{\frac{1}{%
q}}\right) }\left\{ \int_{Q}\left( {\frac{1}{{\lambda }+|y|}}\right) ^{{%
\gamma }q}dy\right\} ^{\frac{1}{q}}\left\{ \int_{Q}\left( {\frac{1}{|y|}}%
\right) ^{\delta \left( {\frac{p}{p-1}}\right) }dy\right\} ^{\frac{p-1}{p}} 
\notag \\
&\gtrsim &\left( {\lambda }\right) ^{\alpha -m\left( {\frac{1}{p}}-{\frac{1}{%
q}}\right) }\left\{ \int_{0<|y|\leq {\lambda }}\left( {\frac{1}{{\lambda }}}%
\right) ^{{\gamma }q}dy\right\} ^{\frac{1}{q}}\gtrsim \left( {\lambda }%
\right) ^{{\frac{n}{q}}-{\gamma }+\alpha -m\left( {\frac{1}{p}}-{\frac{1}{q}}%
\right) }.  \notag
\end{eqnarray}
Recall the estimate in (\ref{A_pq Decay Est}). Now we note that (\ref%
{Characteristic Est2 Case1}) converges to zero as ${\lambda }\longrightarrow
0$. By putting (\ref{Characteristic Est2 Case1}) together with (\ref%
{Necessity est1}), we obtain 
\begin{equation}
{\gamma }<{\frac{n}{q}}+\alpha -m\left( {\frac{1}{p}}-{\frac{1}{q}}\right) .
\label{Necessary Est1}
\end{equation}%
The formula in (\ref{Formula}) implies that (\ref{Necessary Est1}) is
equivalent to 
\begin{equation}
\beta -{\frac{n}{p}}<\delta .  \label{Necessary Est1*}
\end{equation}%
Switching the roles of $P$ and $Q$ in the argument above shows that 
\begin{equation}
\alpha -{\frac{m}{p}}<\delta .  \label{Necessary Est1**}
\end{equation}

\textbf{Case Two:} Consider ${\gamma }\leq 0,\delta >0$. Let $|Q|^{\frac{1}{n%
}}=1$ and $|P|^{\frac{1}{m}}={\lambda }$. Suppose $\alpha -m\left( {\frac{1}{%
p}}-{\frac{1}{q}}\right) =0$. Then we have%
\begin{eqnarray}
&&  \label{Characteristic Est1 Case2} \\
&&|P|^{{\frac{\alpha }{m}}-\left( {\frac{1}{p}}-{\frac{1}{q}}\right) }|Q|^{{%
\frac{\beta }{n}}-\left( {\frac{1}{p}}-{\frac{1}{q}}\right) }\left\{ {\frac{1%
}{|P||Q|}}\iint_{P\times Q^{\varepsilon }}\left( {\frac{1}{|x|+|y|}}\right)
^{{\gamma }q}dxdy\right\} ^{\frac{1}{q}}\left\{ {\frac{1}{|P||Q|}}%
\int_{P\times Q^{\varepsilon }}\left( {\frac{1}{|x|+|y|}}\right) ^{\delta
\left( {\frac{p}{p-1}}\right) }dxdy\right\} ^{\frac{p-1}{p}}  \notag \\
&\gtrsim &\left\{ \int_{Q}\left( {\frac{1}{|y|}}\right) ^{{\gamma }%
q}dy\right\} ^{\frac{1}{q}}\left\{ \int_{Q}\left( {\frac{1}{{\lambda }+|y|}}%
\right) ^{\delta \left( {\frac{p}{p-1}}\right) }dy\right\} ^{\frac{p-1}{p}%
}\gtrsim \left\{ \int_{{\lambda }<|y|\leq 1}\left( {\frac{1}{{\lambda }+|y|}}%
\right) ^{\delta \left( {\frac{p}{p-1}}\right) }dy\right\} ^{\frac{p-1}{p}}.
\notag
\end{eqnarray}%
A direct computation shows 
\begin{equation}
\int_{{\lambda }<|y|\leq 1}\left( {\frac{1}{{\lambda }+|y|}}\right) ^{\delta
\left( {\frac{p}{p-1}}\right) }dy\approx \ln \left( {\frac{1+{\lambda }}{2{%
\lambda }}}\right) \text{ if }\delta =n\left( {\frac{p-1}{p}}\right)
\label{Necessary Case2 Est1}
\end{equation}%
and 
\begin{equation}
\int_{{\lambda }<|y|\leq 1}\left( {\frac{1}{{\lambda }+|y|}}\right) ^{\delta
\left( {\frac{p}{p-1}}\right) }dy\approx \left( {\frac{1}{2{\lambda }}}%
\right) ^{\delta \left( {\frac{p}{p-1}}\right) -n}-\left( {\frac{1}{{\lambda 
}+1}}\right) ^{\delta \left( {\frac{p}{p-1}}\right) -n}\text{ if }\delta
>n\left( {\frac{p-1}{p}}\right) .  \label{Necessary Case2 Est2}
\end{equation}%
Using (\ref{Characteristic Est1 Case2}), (\ref{Necessary Case2 Est1}) and (%
\ref{Necessary Case2 Est2}), and letting ${\lambda }\longrightarrow 0$, we
obtain 
\begin{equation}
\delta <n\left( {\frac{p-1}{p}}\right) .  \label{Necessity est2}
\end{equation}%
On the other hand, suppose that $\alpha -m\left( {\frac{1}{p}}-{\frac{1}{q}}%
\right) >0$. Then we have%
\begin{eqnarray}
&&  \label{Characteristic Est2 Case2} \\
&&|P|^{{\frac{\alpha }{m}}-\left( {\frac{1}{p}}-{\frac{1}{q}}\right) }|Q|^{{%
\frac{\beta }{n}}-\left( {\frac{1}{p}}-{\frac{1}{q}}\right) }\left\{ {\frac{1%
}{|P||Q|}}\iint_{P\times Q^{\varepsilon }}\left( {\frac{1}{|x|+|y|}}\right)
^{{\gamma }q}dxdy\right\} ^{\frac{1}{q}}\left\{ {\frac{1}{|P||Q|}}%
\int_{P\times Q^{\varepsilon }}\left( {\frac{1}{|x|+|y|}}\right) ^{\delta
\left( {\frac{p}{p-1}}\right) }dxdy\right\} ^{\frac{p-1}{p}}  \notag \\
&\gtrsim &\left( {\lambda }\right) ^{\alpha -m\left( {\frac{1}{p}}-{\frac{1}{%
q}}\right) }\left\{ \int_{Q}\left( {\frac{1}{|y|}}\right) ^{{\gamma }%
q}dy\right\} ^{\frac{1}{q}}\left\{ \int_{Q}\left( {\frac{1}{{\lambda }+|y|}}%
\right) ^{\delta \left( {\frac{p}{p-1}}\right) }dy\right\} ^{\frac{p-1}{p}} 
\notag \\
&\gtrsim &\left( {\lambda }\right) ^{\alpha -m\left( {\frac{1}{p}}-{\frac{1}{%
q}}\right) }\left\{ \int_{0<|y|\leq {\lambda }}\left( {\frac{1}{{\lambda }}}%
\right) ^{\delta \left( {\frac{p}{p-1}}\right) }dy\right\} ^{\frac{p-1}{p}%
}\gtrsim \left( {\lambda }\right) ^{\left( {\frac{p-1}{p}}\right) n-\delta
+\alpha -m\left( {\frac{1}{p}}-{\frac{1}{q}}\right) }.  \notag
\end{eqnarray}%
Recall the estimate in (\ref{A_pq Decay Est}). Now we note that (\ref%
{Characteristic Est2 Case2}) converges to zero as ${\lambda }\longrightarrow
0$. By putting (\ref{Characteristic Est2 Case2}) together with (\ref%
{Necessity est2}), we have 
\begin{equation}
\delta <n\left( {\frac{p-1}{p}}\right) +\alpha -m\left( {\frac{1}{p}}-{\frac{%
1}{q}}\right) .  \label{Necessary Est2}
\end{equation}%
The formula in (\ref{Formula}) implies that (\ref{Necessary Est2}) is
equivalent to 
\begin{equation}
\beta -n\left( {\frac{q-1}{q}}\right) <{\gamma }.  \label{Necessary Est2*}
\end{equation}%
Switching the roles of $P$ and $Q$ in the argument above shows that 
\begin{equation}
\alpha -m\left( {\frac{q-1}{q}}\right) <{\gamma }.  \label{Necessary Est2**}
\end{equation}%
This completes the proof that (\ref{recall'}) is necessary for (\ref{p at
most q}) and (\ref{finiteness of A}), and hence we have proved $\left(
2\right) \Longrightarrow \left( 3\right) $.\newline

Now we turn to proving $\left( 3\right) \Longrightarrow \left( 1\right) $,
i.e. that (\ref{p at most q}) and (\ref{recall'}) are sufficient for the
norm inequality (\ref{pwni}). To this end we use Young's inequality 
\begin{equation}
a^{1-\theta }b^{\theta }\leq \sqrt{a^{2}+b^{2}}\leq a+b,  \label{Young}
\end{equation}%
valid for $a,b\geq 0$ and $0\leq \theta \leq 1$, in order to define weight
pairs to which the sandwiching principle can be applied. The special cases $%
\alpha =m$ and $\beta =n$, along with some additional exceptional cases,
will be treated at the end of the proof.

\subsection{The nonexceptional cases}

We assume that $\alpha <m$ and $\beta <n$ and show here that $%
N_{p,q}^{\left( \alpha ,\beta \right) ,\left( m,n\right) }\left( v_{\delta
},w_{\gamma }\right) <\infty $ follows from (\ref{p at most q}), (\ref%
{Formula}), (\ref{gamma+delta}) and (\ref{star'}) when both $\gamma \geq 0$
and $\delta \geq 0$, and also when either $\gamma <0$ or $\delta <0$.

\textbf{Case 1}: First we suppose that $\gamma \geq 0$ and $\delta \geq 0$.
Define the weight pair 
\begin{equation*}
\left( V\left( u,t\right) ,W\left( x,y\right) \right) \equiv \left(
\left\vert u\right\vert ^{\frac{\delta _{1}m}{m+n}}\left\vert t\right\vert ^{%
\frac{\delta _{2}n}{m+n}},\left( \frac{1}{\left\vert x\right\vert }\right) ^{%
\frac{\gamma _{1}m}{m+n}}\left( \frac{1}{\left\vert y\right\vert }\right) ^{%
\frac{\gamma _{2}n}{m+n}}\right)
\end{equation*}%
where the indices $\delta _{1}$, $\delta _{2}$, $\gamma _{1}$, $\gamma _{2}$
satisfy%
\begin{equation}
\frac{\delta _{1}m}{m+n}+\frac{\delta _{2}n}{m+n}=\delta \text{ and }\frac{%
\gamma _{1}m}{m+n}+\frac{\gamma _{2}n}{m+n}=\gamma ,  \label{prov 1}
\end{equation}%
and%
\begin{eqnarray}
\frac{\alpha -\left( \frac{\gamma _{1}m}{m+n}+\frac{\delta _{1}m}{m+n}%
\right) }{m} &=&\Gamma =\frac{\beta -\left( \frac{\gamma _{2}n}{m+n}+\frac{%
\delta _{2}n}{m+n}\right) }{n},  \label{prov 2} \\
\text{i.e. }\frac{\gamma _{1}m}{m+n}+\frac{\delta _{1}m}{m+n} &=&\alpha
-m\Gamma \text{ and }\frac{\gamma _{2}n}{m+n}+\frac{\delta _{2}n}{m+n}=\beta
-n\Gamma  \notag
\end{eqnarray}%
Solving for 
\begin{equation*}
\bigtriangleup _{1}\equiv \frac{\delta _{1}m}{m+n},\ \ \ \bigtriangleup
_{2}\equiv \frac{\delta _{2}n}{m+n},\ \ \ \Gamma _{1}\equiv \frac{\gamma
_{1}m}{m+n}\gamma _{1},\text{\ \ \ }\Gamma _{2}\equiv \frac{\gamma _{2}n}{m+n%
},
\end{equation*}%
we obtain the system%
\begin{equation}
\left[ 
\begin{array}{cccc}
1 & 1 & 0 & 0 \\ 
0 & 1 & 0 & 1 \\ 
0 & 0 & 1 & 1 \\ 
1 & 0 & 1 & 0%
\end{array}%
\right] \left( 
\begin{array}{c}
\bigtriangleup _{1} \\ 
\bigtriangleup _{2} \\ 
\Gamma _{1} \\ 
\Gamma _{2}%
\end{array}%
\right) =\left( 
\begin{array}{c}
\delta \\ 
\beta -n\Gamma \\ 
\gamma \\ 
\alpha -m\Gamma%
\end{array}%
\right) ,  \label{system}
\end{equation}%
which in reduced row echelon form is%
\begin{equation*}
\left[ 
\begin{array}{cccc}
1 & 0 & 0 & -1 \\ 
0 & 1 & 0 & 1 \\ 
0 & 0 & 1 & 1 \\ 
0 & 0 & 0 & 0%
\end{array}%
\right] \left( 
\begin{array}{c}
\bigtriangleup _{1} \\ 
\bigtriangleup _{2} \\ 
\Gamma _{1} \\ 
\Gamma _{2}%
\end{array}%
\right) =\left( 
\begin{array}{c}
\delta -\beta +n\Gamma \\ 
\beta -n\Gamma \\ 
\gamma \\ 
\alpha -m\Gamma -\delta +\beta -n\Gamma%
\end{array}%
\right) .
\end{equation*}%
The system is solvable since $\alpha -m\Gamma -\delta +\beta -n\Gamma =0$ by
the power weight equality in the first line of (\ref{3 lines}), and the
general solution to the system (\ref{system}) is thus given by%
\begin{equation*}
\left( 
\begin{array}{c}
\bigtriangleup _{1} \\ 
\bigtriangleup _{2} \\ 
\Gamma _{1} \\ 
\Gamma _{2}%
\end{array}%
\right) =\left( 
\begin{array}{c}
\delta -\beta +n\Gamma \\ 
\beta -n\Gamma \\ 
\gamma \\ 
0%
\end{array}%
\right) +\lambda \left( 
\begin{array}{c}
-1 \\ 
1 \\ 
1 \\ 
-1%
\end{array}%
\right) \equiv \mathbf{z}_{\lambda },\ \ \ \ \ \lambda \in \mathbb{R}.
\end{equation*}%
Among these solution vectors $\mathbf{z}_{\lambda }$, we will find a vector
satisfying all of the constraint inequalities needed below.

Now by Young's inequality (\ref{Young}) and (\ref{prov 1}), we have%
\begin{equation*}
\frac{w\left( x,y\right) }{v\left( u,t\right) }\leq \frac{W\left( x,y\right) 
}{V\left( u,t\right) },
\end{equation*}%
and so by the sandwiching principle, Lemma \ref{sandwich}, we have 
\begin{equation}
N_{p,q}^{\left( \alpha ,\beta \right) ,\left( m,n\right) }\left( v,w\right)
\leq N_{p,q}^{\left( \alpha ,\beta \right) ,\left( m,n\right) }\left(
V,W\right) .  \label{N sand}
\end{equation}%
Moreover, the weights $V,W$ are product weights,%
\begin{eqnarray*}
V\left( u,t\right) &=&V_{1}\left( u\right) V_{2}\left( t\right) \text{ and }%
W\left( x,y\right) =W_{1}\left( x\right) W_{2}\left( y\right) ; \\
V_{1}\left( u\right) &=&\left\vert u\right\vert ^{\frac{\delta _{1}m}{m+n}%
},\ \ \ V_{2}\left( t\right) =\left\vert t\right\vert ^{\frac{\delta _{2}n}{%
m+n}}, \\
W_{1}\left( x\right) &=&\left( \frac{1}{\left\vert x\right\vert }\right) ^{%
\frac{\gamma _{1}m}{m+n}},\ \ \ W_{2}\left( y\right) =\left( \frac{1}{%
\left\vert y\right\vert }\right) ^{\frac{\gamma _{2}n}{m+n}},
\end{eqnarray*}%
where the $1$-parameter weight pairs $\left( V_{1}\left( u\right)
,W_{1}\left( x\right) \right) $ and $\left( V_{2}\left( t\right)
,W_{2}\left( y\right) \right) $ each satisfy the hypotheses of Theorem \ref%
{Stein-Weiss} on $\mathbb{R}^{m}$ and $\mathbb{R}^{n}$ respectively for an
appropriate choice of solution vector above.

Indeed, to see this, note that the first weight pair $\left( V_{1}\left(
u\right) ,W_{1}\left( x\right) \right) =\left( \left\vert u\right\vert
^{\bigtriangleup _{1}},\left\vert x\right\vert ^{-\Gamma _{1}}\right) $ on $%
\mathbb{R}^{m}$ satisfies the equality 
\begin{equation}
\frac{1}{p}-\frac{1}{q}=\frac{\alpha -\left( \Gamma _{1}+\bigtriangleup
_{1}\right) }{m}  \label{V1 W1 equ}
\end{equation}%
by (\ref{prov 2}). We also claim the inequalities%
\begin{equation}
q\Gamma _{1}<m\text{ and }p^{\prime }\bigtriangleup _{1}<m\text{ and }\Gamma
_{1}+\bigtriangleup _{1}\geq 0,  \label{V1 W1 inequ}
\end{equation}%
for an appropriate family of solution vectors $\mathbf{z}_{\lambda }$. The
third inequality actually holds for all solution vectors $\mathbf{z}%
_{\lambda }$ since 
\begin{equation*}
\Gamma _{1}+\bigtriangleup _{1}=\gamma +\delta -\beta +n\Gamma \geq 0,
\end{equation*}%
by (\ref{3 lines}). Thus the equality (\ref{V1 W1 equ}), and the
inequalities (\ref{V1 W1 inequ}), all hold for those solution vectors $%
\mathbf{z}_{\lambda }$ satisfying%
\begin{eqnarray}
&&\Gamma _{1}<\frac{m}{q}\text{ and }\bigtriangleup _{1}<\frac{m}{p^{\prime }%
};  \label{Sol 1} \\
&\text{i.e }&\gamma +\lambda <\frac{m}{q}\text{ and }\delta -\beta +n\Gamma
-\lambda <\frac{m}{p^{\prime }};  \notag \\
&\text{i.e.}&\delta -\beta +n\Gamma -\frac{m}{p^{\prime }}<\lambda <\frac{m}{%
q}-\gamma .  \notag
\end{eqnarray}

The second weight pair $\left( V_{2}\left( t\right) ,W_{2}\left( y\right)
\right) =\left( \left\vert t\right\vert ^{\bigtriangleup _{2}},\left\vert
y\right\vert ^{-\Gamma _{2}}\right) $ on $\mathbb{R}^{n}$ satisfies the
equality 
\begin{equation}
\frac{1}{p}-\frac{1}{q}=\frac{\alpha -\left( \Gamma _{2}+\bigtriangleup
_{2}\right) }{n}  \label{V2 W2 equ}
\end{equation}%
by (\ref{prov 2}). We also claim the inequalities%
\begin{equation}
q\Gamma _{2}<n\text{ and }p^{\prime }\bigtriangleup _{2}<n\text{ and }\Gamma
_{2}+\bigtriangleup _{2}\geq 0,  \label{V2 W2 inequ}
\end{equation}%
for an appropriate family of solution vectors $\mathbf{z}_{\lambda }$. The
third inequality actually holds for all solution vectors $\mathbf{z}%
_{\lambda }$ since 
\begin{equation*}
\Gamma _{2}+\bigtriangleup _{2}=\alpha +\beta -\left( m+n\right) \Gamma
=\gamma +\delta \geq 0,
\end{equation*}%
by (\ref{3 lines}). Thus the equality (\ref{V2 W2 equ}), and the
inequalities (\ref{V2 W2 inequ}), all hold for those solution vectors $%
\mathbf{z}_{\lambda }$ satisfying%
\begin{eqnarray}
&&\Gamma _{2}<\frac{n}{q}\text{ and }\bigtriangleup _{2}<\frac{n}{p^{\prime }%
};  \label{Sol 2} \\
&\text{i.e.}&-\lambda <\frac{n}{q}\text{ and }\beta -n\Gamma +\lambda <\frac{%
n}{p^{\prime }};  \notag \\
&\text{i.e.}&-\frac{n}{q}<\lambda <\frac{n}{p^{\prime }}-\beta +n\Gamma . 
\notag
\end{eqnarray}

In order to find $\lambda $ satisfying (\ref{Sol 1}) and (\ref{Sol 2})
simultaneously, we must establish the four strict inequalities in%
\begin{equation*}
\max \left\{ \delta -\beta +n\Gamma -\frac{m}{p^{\prime }},-\frac{n}{q}%
\right\} <\min \left\{ \frac{m}{q}-\gamma ,\frac{n}{p^{\prime }}-\beta
+n\Gamma \right\} .
\end{equation*}%
Now two of these four strict inequalities follow from the local
integrability of the weights $\left\vert \left( x,y\right) \right\vert
^{-\gamma q}$ and $\left\vert \left( u,t\right) \right\vert ^{-\gamma
p^{\prime }}$ on $\mathbb{R}^{m}\times \mathbb{R}^{n}$, namely%
\begin{eqnarray*}
\delta -\beta +n\Gamma -\frac{m}{p^{\prime }} &<&\frac{n}{p^{\prime }}-\beta
+n\Gamma \text{ and }-\frac{n}{q}<\frac{m}{q}-\gamma ; \\
\text{i.e. }\delta &<&\frac{m+n}{p^{\prime }}\text{ and }\gamma <\frac{m+n}{q%
}.
\end{eqnarray*}%
The other two strict inequalities follow from the assumptions that $\alpha
<m $ and $\beta <n$, namely%
\begin{eqnarray*}
&&\delta -\beta +n\Gamma -\frac{m}{p^{\prime }}<\frac{m}{q}-\gamma \text{
and }-\frac{n}{q}<\frac{n}{p^{\prime }}-\beta +n\Gamma ; \\
\text{i.e. } &&\delta +\gamma <\beta -n\Gamma +m\left( \frac{1}{p^{\prime }}+%
\frac{1}{q}\right) \text{ and }\beta <n\left( \frac{1}{p^{\prime }}+\frac{1}{%
q}\right) +n\Gamma ; \\
\text{i.e. } &&\delta +\gamma <\beta -n\Gamma +m\left( 1-\Gamma \right) 
\text{ and }\beta <n\left( 1-\Gamma \right) +n\Gamma ; \\
\text{i.e. } &&\left( m+n\right) \Gamma +\left( \delta +\gamma \right)
<m+\beta \text{ and }\beta <n; \\
\text{i.e. } &&\alpha +\beta <m+\beta \text{ and }\beta <n,
\end{eqnarray*}%
where in the final line above we have used the power weight equality from
the first line of (\ref{3 lines}).

Thus there does indeed exist a choice of $\gamma \in \mathbb{R}$ so that the
equalities (\ref{V1 W1 equ}), (\ref{V2 W2 equ}) and inequalities (\ref{V1 W1
inequ}), (\ref{V2 W2 inequ}) all hold. It now follows from Theorem \ref%
{iteration} that%
\begin{eqnarray*}
N_{p,q}^{\left( \alpha ,\beta \right) ,\left( m,n\right) }\left( V,W\right)
&\leq &N_{p,q}^{\alpha ,m}\left( V_{1},W_{1}\right) \ N_{p,q}^{\beta
,n}\left( V_{2},W_{2}\right) \\
&\leq &CA_{p,q}^{\alpha ,m}\left( V_{1},W_{1}\right) \ A_{p,q}^{\beta
,n}\left( V_{2},W_{2}\right) <\infty ,
\end{eqnarray*}%
and combined with (\ref{N sand}) this yields%
\begin{equation*}
N_{p,q}^{\left( \alpha ,\beta \right) ,\left( m,n\right) }\left( v,w\right)
<\infty ,
\end{equation*}%
in the case $\gamma \geq 0$ and $\delta \geq 0$.

\textbf{Case 2}: Next we suppose that $\gamma <0<\delta $ and use the fourth
line in (\ref{recall'}). Let $\rho \equiv \gamma +$ $\delta \geq 0$ and $%
\eta \equiv -\gamma >0$. Then by Young's inequality (\ref{Young}), the
weight pairs%
\begin{eqnarray*}
\left( V\left( u,t\right) ,W\left( x,y\right) \right) &\equiv &\left(
\left\vert u\right\vert ^{\frac{\rho _{1}m}{m+n}+\eta _{1}}\left\vert
t\right\vert ^{\frac{\rho _{1}n}{m+n}},\left\vert x\right\vert ^{\eta
}\right) , \\
\left( V^{\prime }\left( u,t\right) ,W^{\prime }\left( x,y\right) \right)
&\equiv &\left( \left\vert u\right\vert ^{\frac{\rho _{2}m}{m+n}}\left\vert
t\right\vert ^{\frac{\rho _{2}n}{m+n}+\eta _{2}},\left\vert y\right\vert
^{\eta }\right) ,
\end{eqnarray*}%
where $\rho _{j}+\eta _{j}=\rho +\eta =\delta >0$ for $j=1,2$, satisfy%
\begin{eqnarray*}
\frac{w\left( x,y\right) }{v\left( u,t\right) } &=&\frac{\left( \left\vert
x\right\vert ^{2}+\left\vert y\right\vert ^{2}\right) ^{-\frac{\gamma }{2}}}{%
\left( \left\vert u\right\vert ^{2}+\left\vert t\right\vert ^{2}\right) ^{%
\frac{\delta }{2}}}=\frac{\left( \left\vert x\right\vert ^{2}+\left\vert
y\right\vert ^{2}\right) ^{\frac{\eta }{2}}}{\left( \left\vert u\right\vert
^{2}+\left\vert t\right\vert ^{2}\right) ^{\frac{\rho +\eta }{2}}} \\
&\lesssim &\frac{\left\vert x\right\vert ^{\eta }+\left\vert y\right\vert
^{\eta }}{\left( \left\vert u\right\vert ^{2}+\left\vert t\right\vert
^{2}\right) ^{\frac{\rho _{1}}{2}}\left( \left\vert u\right\vert
^{2}+\left\vert t\right\vert ^{2}\right) ^{\frac{\eta _{1}}{2}}}+\frac{%
\left\vert x\right\vert ^{\eta }+\left\vert y\right\vert ^{\eta }}{\left(
\left\vert u\right\vert ^{2}+\left\vert t\right\vert ^{2}\right) ^{\frac{%
\rho _{2}}{2}}\left( \left\vert u\right\vert ^{2}+\left\vert t\right\vert
^{2}\right) ^{\frac{\eta _{2}}{2}}} \\
&\lesssim &\frac{\left\vert x\right\vert ^{\eta }}{\left( \left\vert
u\right\vert ^{\frac{m}{m+n}}\left\vert t\right\vert ^{\frac{n}{m+n}}\right)
^{\rho _{1}}\left\vert u\right\vert ^{\eta _{1}}}+\frac{\left\vert
y\right\vert ^{\eta }}{\left( \left\vert u\right\vert ^{\frac{m}{m+n}%
}\left\vert t\right\vert ^{\frac{n}{m+n}}\right) ^{\rho _{2}}\left\vert
t\right\vert ^{\eta _{2}}}=\frac{W\left( x,y\right) }{V\left( u,t\right) }+%
\frac{W^{\prime }\left( x,y\right) }{V^{\prime }\left( u,t\right) },
\end{eqnarray*}%
and so by the sandwiching principle, Lemma \ref{sandwich}, we have 
\begin{equation*}
N_{p,q}^{\left( \alpha ,\beta \right) ,\left( m,n\right) }\left( v,w\right)
\leq N_{p,q}^{\left( \alpha ,\beta \right) ,\left( m,n\right) }\left(
V,W\right) +N_{p,q}^{\left( \alpha ,\beta \right) ,\left( m,n\right) }\left(
V^{\prime },W^{\prime }\right) .
\end{equation*}%
Moreover, the weights $V,V^{\prime },W,W^{\prime }$ are product weights,%
\begin{eqnarray*}
V\left( u,t\right) &=&V_{1}\left( u\right) V_{2}\left( t\right) \text{ and }%
V^{\prime }\left( u,t\right) =V_{1}^{\prime }\left( u\right) V_{2}^{\prime
}\left( t\right) , \\
W\left( x,y\right) &=&W_{1}\left( x\right) W_{2}\left( y\right) \text{ and }%
W^{\prime }\left( x,y\right) =W_{1}^{\prime }\left( x\right) W_{2}^{\prime
}\left( y\right) ,
\end{eqnarray*}%
where%
\begin{equation*}
\left\{ 
\begin{array}{cccc}
V_{1}\left( u\right) =\left\vert u\right\vert ^{\frac{\rho _{1}m}{m+n}+\eta
_{1}} & V_{2}\left( t\right) =\left\vert t\right\vert ^{\frac{\rho _{1}n}{m+n%
}} & V_{1}^{\prime }\left( u\right) =\left\vert u\right\vert ^{\frac{\rho
_{2}m}{m+n}} & V_{2}^{\prime }\left( t\right) =\left\vert t\right\vert ^{%
\frac{\rho _{2}n}{m+n}+\eta _{2}} \\ 
W_{1}\left( x\right) =\left\vert x\right\vert ^{\eta } & W_{2}\left(
y\right) =1 & W_{1}^{\prime }\left( x\right) =1 & W_{2}^{\prime }\left(
y\right) =\left\vert y\right\vert ^{\eta }%
\end{array}%
\right. ,
\end{equation*}%
and where the $1$-parameter weight pairs $\left( V_{1}\left( u\right)
,W_{1}\left( x\right) \right) ,\left( V_{1}^{\prime }\left( u\right)
,W_{1}^{\prime }\left( x\right) \right) $ on $\mathbb{R}^{m}$ and $\left(
V_{2}\left( t\right) ,W_{2}\left( y\right) \right) ,\left( V_{2}^{\prime
}\left( t\right) ,W_{2}^{\prime }\left( y\right) \right) $ on $\mathbb{R}%
^{n} $ each satisfy the hypotheses of Theorem \ref{Stein-Weiss} provided we
choose $\rho _{1}$ and $\eta _{1}$ to satisfy%
\begin{eqnarray}
\frac{\alpha -\left( -\eta +\frac{\rho _{1}m}{m+n}+\eta _{1}\right) }{m}
&=&\Gamma =\frac{\alpha +\beta -\rho }{m+n},  \label{satisfy} \\
\text{and }\frac{\beta -\left( -0+\frac{\rho _{1}n}{m+n}\right) }{n}
&=&\Gamma =\frac{\alpha +\beta -\rho }{m+n},  \notag
\end{eqnarray}%
i.e.%
\begin{eqnarray*}
\frac{\alpha }{m}-\frac{\rho _{1}}{m+n}+\frac{\eta -\eta _{1}}{m} &=&\frac{%
\alpha +\beta -\rho }{m+n}=\frac{\beta }{n}-\frac{\rho _{1}}{m+n}; \\
\frac{\alpha }{m}+\frac{\eta -\eta _{1}}{m} &=&\frac{\alpha +\beta +\rho
_{1}-\rho }{m+n}=\frac{\beta }{n}; \\
\frac{\beta }{n}-\frac{\alpha }{m} &=&\frac{\eta -\eta _{1}}{m}\text{ and }%
\frac{\alpha +\beta }{m+n}=\frac{\beta }{n}+\frac{\rho -\rho _{1}}{m+n}; \\
\frac{\eta -\eta _{1}}{m} &=&\frac{\beta }{n}-\frac{\alpha }{m}\text{ and }%
\frac{\rho -\rho _{1}}{m+n}=\frac{\alpha +\beta }{m+n}-\frac{\beta }{n}; \\
\frac{\eta _{1}}{m} &=&\frac{\alpha +\eta }{m}-\frac{\beta }{n}\text{ and }%
\frac{\rho _{1}}{m+n}=\frac{\rho -\left( \alpha +\beta \right) }{m+n}+\frac{%
\beta }{n},
\end{eqnarray*}%
so that%
\begin{equation}
\eta _{1}=\alpha +\eta -\frac{m}{n}\beta \text{ and }\rho _{1}=\rho -\left(
\alpha +\beta \right) +\frac{m+n}{n}\beta ,  \label{def 1}
\end{equation}%
and where $\rho _{2}$ and $\eta _{2}$ will be chosen below. Once we have
established the appropriate hypotheses of Theorem \ref{Stein-Weiss}, Theorem %
\ref{iteration} will show that 
\begin{equation*}
N_{p,q}^{\left( \alpha ,\beta \right) ,\left( m,n\right) }\left( v,w\right)
<\infty ,
\end{equation*}%
in the case $\gamma <0<\delta $.

To show that these four weight pairs satisfy the hypotheses of Theorem \ref%
{Stein-Weiss}, we first note that%
\begin{equation}
\rho _{1}=\left( m+n\right) \left[ \frac{\left( \gamma +\delta \right)
-\left( \alpha +\beta \right) }{m+n}+\frac{\beta }{n}\right] =\left(
m+n\right) \left( \frac{\beta }{n}-\Gamma \right) \geq 0  \label{rho 1 0}
\end{equation}%
by the power weight equality in (\ref{3 lines}), and the third line in (\ref%
{3 lines}), and also note that 
\begin{eqnarray*}
\rho _{1}+\eta _{1} &=&\left( m+n\right) \left\{ \frac{\rho -\left( \alpha
+\beta \right) }{m+n}+\frac{\beta }{n}\right\} +m\left\{ \frac{\alpha +\eta 
}{m}-\frac{\beta }{n}\right\} \\
&=&\rho -\left( \alpha +\beta \right) +\frac{m+n}{n}\beta +\left( \alpha
+\eta \right) -\frac{m}{n}\beta =\rho +\eta .
\end{eqnarray*}%
Next, we verify that the first weight pair $\left( V_{1}\left( u\right)
,W_{1}\left( x\right) \right) =\left( \left\vert u\right\vert ^{\frac{\rho
_{1}m}{m+n}+\eta _{1}},\left\vert x\right\vert ^{\eta }\right) $ on $\mathbb{%
R}^{m}$ satisfies the $1$-parameter power weight equality%
\begin{equation}
\frac{1}{p}-\frac{1}{q}=\frac{\alpha -\left( -\eta +\frac{\rho _{1}m}{m+n}%
+\eta _{1}\right) }{m},  \label{equ diag}
\end{equation}%
as well as the $1$-parameter constraint inequalities%
\begin{equation}
q\left( -\eta \right) <m\text{ and }p^{\prime }\left( \frac{\rho _{1}m}{m+n}%
+\eta _{1}\right) <m\text{ and }-\eta +\frac{\rho _{1}m}{m+n}+\eta _{1}\geq
0.  \label{inequ diag}
\end{equation}%
The equality (\ref{equ diag}) follows immediately from (\ref{def 1}), the
first line in (\ref{3 lines}), and (\ref{satisfy}). The first inequality in (%
\ref{inequ diag}) is trivial, and the third inequality in (\ref{inequ diag})
is%
\begin{equation*}
-\eta +\frac{\rho _{1}m}{m+n}+\eta _{1}\geq 0,
\end{equation*}%
which follows from (\ref{def 1}):%
\begin{eqnarray}
-\eta +\frac{m}{m+n}\rho _{1}+\eta _{1} &=&-\eta +\frac{m}{m+n}\left[ \rho
-\left( \alpha +\beta \right) +\frac{m+n}{n}\beta \right] +\alpha +\eta -%
\frac{m}{n}\beta  \label{just proved in} \\
&=&-\eta +\frac{m}{m+n}\left( \rho -\left( \alpha +\beta \right) \right) +%
\frac{m}{n}\beta +\alpha +\eta -\frac{m}{n}\beta  \notag \\
&=&\frac{m}{m+n}\left( \rho -\left( \alpha +\beta \right) \right) +\alpha 
\notag \\
&=&m\left( \frac{\rho -\left( \alpha +\beta \right) }{m+n}+\frac{\alpha }{m}%
\right) =m\left( \frac{\alpha }{m}-\Gamma \right) \geq 0.  \notag
\end{eqnarray}%
The second inequality in (\ref{inequ diag}) is%
\begin{equation}
\frac{1}{m}\left( \frac{\rho _{1}m}{m+n}+\eta _{1}\right) <\frac{1}{%
p^{\prime }},  \label{second inequ}
\end{equation}%
which, using $\gamma =-\eta $ with the equality $-\eta +\frac{m}{m+n}\rho
_{1}+\eta _{1}=m\left( \frac{\alpha }{m}-\Gamma \right) =\alpha -m\Gamma $
just proved above in (\ref{just proved in}), is equivalent to 
\begin{eqnarray*}
&&\frac{1}{m}\left( \alpha -m\Gamma +\eta \right) <\frac{1}{p^{\prime }}; \\
&\Longleftrightarrow &\frac{\alpha +\eta }{m}-\Gamma <\frac{1}{p^{\prime }};
\\
&\Longleftrightarrow &\frac{\alpha }{m}<\Gamma +\frac{\gamma }{m}+\frac{1}{%
p^{\prime }}; \\
&\Longleftrightarrow &\frac{\alpha }{m}-\frac{1}{q^{\prime }}<\frac{\gamma }{%
m}.
\end{eqnarray*}

Next we note that the weight pair $\left( V_{2}\left( t\right) ,W_{2}\left(
y\right) \right) =\left( \left\vert t\right\vert ^{\frac{\rho _{1}n}{m+n}%
},1\right) $ on $\mathbb{R}^{n}$ satisfies the $1$-parameter power weight
equality%
\begin{equation*}
\frac{1}{p}-\frac{1}{q}=\frac{\beta -\left( -0+\frac{\rho _{1}m}{m+n}\right) 
}{n},
\end{equation*}%
and the $1$-parameter constraint inequalities,%
\begin{equation*}
q\left( -0\right) <n\text{ and }p^{\prime }\left( \frac{\rho _{1}n}{m+n}%
\right) <n\text{ and }-0+\frac{\rho _{1}m}{m+n}\geq 0.
\end{equation*}%
Indeed, for the equality we use (\ref{def 1}), the first line in (\ref{3
lines}), and (\ref{satisfy}). The first of the constraint inequalities is
trivial and the third constraint inequality follows from (\ref{rho 1 0}).
The second of the constraint inequalities, namely $p^{\prime }\left( \frac{%
\rho _{1}n}{m+n}\right) <n$, is equivalent to%
\begin{eqnarray*}
&&\left( \frac{n}{m+n}\right) \left( m+n\right) \left( \frac{\beta }{n}%
-\Gamma \right) <\frac{n}{p^{\prime }}; \\
\text{i.e. } &&\frac{\beta }{n}<\Gamma +\frac{1}{p^{\prime }}.
\end{eqnarray*}%
However from the third line in (\ref{3 lines}) and the assumption that $%
\gamma <0$ we have%
\begin{eqnarray*}
\frac{\beta }{n} &\leq &\Gamma +\frac{\gamma +\delta }{n}-\frac{\left(
\delta -\frac{n}{p^{\prime }}\right) _{+}}{n} \\
&=&\Gamma +\frac{\gamma }{n}+\frac{1}{p^{\prime }}+\frac{\delta -\frac{n}{%
p^{\prime }}}{n}-\left( \frac{\delta }{n}-\frac{1}{p^{\prime }}\right) _{+}
\\
&=&\Gamma +\frac{\gamma }{n}+\frac{1}{p^{\prime }}-\left( \frac{\delta }{n}-%
\frac{1}{p^{\prime }}\right) _{-} \\
&\leq &\Gamma +\frac{\gamma }{n}+\frac{1}{p^{\prime }}<\Gamma +\frac{1}{%
p^{\prime }},
\end{eqnarray*}%
and this time there is no exceptional endpoint case. As indicated above,
this completes the proof in the case $\gamma <0<\delta $.

The same arguments apply to the weight pair $\left( V^{\prime },W^{\prime
}\right) $, which we now sketch briefly. The weight pairs $\left(
V_{2}^{\prime }\left( t\right) ,W_{2}^{\prime }\left( y\right) \right)
=\left( \left\vert t\right\vert ^{\frac{\rho _{2}n}{m+n}+\eta
_{2}},\left\vert y\right\vert ^{\eta }\right) $ and $\left( V_{1}^{\prime
}\left( u\right) ,W_{1}^{\prime }\left( x\right) \right) =\left( \left\vert
u\right\vert ^{\frac{\rho _{2}m}{m+n}},1\right) $ each satisfy the
hypotheses of Theorem \ref{Stein-Weiss} provided we choose $\rho _{2}$ and $%
\eta _{2}$ to satisfy%
\begin{eqnarray}
\frac{\beta -\left( -\eta +\frac{\rho _{2}n}{m+n}+\eta _{2}\right) }{n}
&=&\Gamma =\frac{\alpha +\beta -\rho }{m+n},  \label{satisfy'} \\
\text{and }\frac{\alpha -\left( -0+\frac{\rho _{2}m}{m+n}\right) }{m}
&=&\Gamma =\frac{\alpha +\beta -\rho }{m+n},  \notag
\end{eqnarray}%
i.e.%
\begin{eqnarray*}
\frac{\alpha }{m}-\frac{\rho _{2}}{m+n} &=&\frac{\alpha +\beta -\rho }{m+n}=%
\frac{\beta }{n}-\frac{\rho _{2}}{m+n}+\frac{\eta -\eta _{2}}{n}; \\
\frac{\alpha }{m} &=&\frac{\alpha +\beta +\rho _{2}-\rho }{m+n}=\frac{\beta 
}{n}+\frac{\eta -\eta _{2}}{n}; \\
\frac{\alpha }{m}-\frac{\beta }{n} &=&\frac{\eta -\eta _{2}}{n}\text{ and }%
\frac{\alpha +\beta }{m+n}=\frac{\alpha }{m}+\frac{\rho -\rho _{2}}{m+n}; \\
\frac{\eta _{2}}{n} &=&\frac{\beta +\eta }{n}-\frac{\alpha }{m}\text{ and }%
\frac{\rho _{2}}{m+n}=\frac{\rho -\left( \alpha +\beta \right) }{m+n}+\frac{%
\alpha }{m};
\end{eqnarray*}%
so that%
\begin{equation}
\eta _{2}=\beta +\eta -\frac{n}{m}\alpha \text{ and }\rho _{2}=\rho -\left(
\alpha +\beta \right) +\frac{m+n}{m}\alpha .  \label{def 1'}
\end{equation}

All of the constraint inequalities hold by arguments similar to those above,
except of course in the analogous exceptional endpoint case, and by way of
example we treat just the second of the constraint inequalities for the
weight pair that gives rise to the exceptional endpoint case. The second
constraint inequality for the weight pair $\left( V_{2}^{\prime }\left(
t\right) ,W_{2}^{\prime }\left( y\right) \right) =\left( \left\vert
t\right\vert ^{\frac{\rho _{2}m}{m+n}+\eta _{2}},\left\vert y\right\vert
^{\eta }\right) $ on $\mathbb{R}^{n}$ is%
\begin{equation}
\frac{1}{n}\left( \frac{\rho _{2}m}{m+n}+\eta _{2}\right) <\frac{1}{%
p^{\prime }}.  \label{which is}
\end{equation}%
Using $\gamma =-\eta $ with the equality 
\begin{eqnarray*}
&&-\eta +\frac{n}{m+n}\rho _{2}+\eta _{2} \\
&=&-\eta +\frac{n}{m+n}\left[ \rho -\left( \alpha +\beta \right) +\frac{m+n}{%
m}\alpha \right] +\beta +\eta -\frac{n}{m}\alpha \\
&=&-\eta +\frac{n}{m+n}\left( \rho -\left( \alpha +\beta \right) \right) +%
\frac{n}{m}\alpha +\beta +\eta -\frac{n}{m}\alpha \\
&=&\frac{n}{m+n}\left( \rho -\left( \alpha +\beta \right) \right) +\beta \\
&=&n\left( \frac{\rho -\left( \alpha +\beta \right) }{m+n}+\frac{\beta }{n}%
\right) =n\left( \frac{\beta }{n}-\Gamma \right) ,
\end{eqnarray*}%
we see that (\ref{which is}) is equivalent to%
\begin{eqnarray*}
&&\frac{1}{n}\left( n\left( \frac{\beta }{n}-\Gamma \right) +\eta \right) <%
\frac{1}{p^{\prime }} \\
&\Longleftrightarrow &\frac{\beta }{n}<\Gamma +\frac{1}{p^{\prime }}+\frac{%
\gamma }{n} \\
&\Longleftrightarrow &\frac{\beta }{n}-\frac{1}{q^{\prime }}<\frac{\gamma }{n%
}.
\end{eqnarray*}

\textbf{Case 3}: The case $\delta <0<\gamma $ is handled similarly using the
third line in (\ref{recall'}).

\subsection{The exceptional cases $\protect\alpha =m$ or $\protect\beta =n$}

Here we consider the two cases where $\alpha =m$ or $\beta =n$. We first
show that the power weight norm inequality (\ref{pwni}) holds when $\alpha
=m $ and $p\leq q$ and $A_{p,q}^{\left( \alpha ,\beta \right) ,\left(
m,n\right) }\left( v_{\delta },w_{\gamma }\right) <\infty $. To see this, we
note that from the first line in (\ref{corres}) that 
\begin{eqnarray*}
0 &<&\beta <n, \\
\frac{m}{q} &<&\gamma <\frac{m+n}{q}, \\
\frac{m}{p^{\prime }} &<&\delta <\frac{m+n}{p^{\prime }}.
\end{eqnarray*}%
Then\ we compute that%
\begin{equation*}
I_{m,\beta }^{m,n}f\left( x,y\right) =\int \int_{\mathbb{R}^{m}\times 
\mathbb{R}^{n}}\left\vert y-t\right\vert ^{\beta -n}f\left( u,t\right)
dudt=\int_{\mathbb{R}^{n}}\left\vert y-t\right\vert ^{\beta -n}F\left(
t\right) dt=I_{\beta }^{n}F\left( y\right) ,
\end{equation*}%
where $F\left( t\right) \equiv \int_{\mathbb{R}^{m}}f\left( u,t\right) du$.
Thus we have%
\begin{eqnarray*}
&&\int \int_{\mathbb{R}^{m}\times \mathbb{R}^{n}}\left\vert I_{m,\beta
}^{m,n}f\left( x,y\right) \right\vert ^{q}\ \left\vert \left( x,y\right)
\right\vert ^{-\gamma q}dxdy \\
&=&\int_{\mathbb{R}^{n}}\left\vert I_{\beta }^{n}F\left( y\right)
\right\vert ^{q}\left( \int_{\mathbb{R}^{m}}\left\vert \left( x,y\right)
\right\vert ^{-\gamma q}dx\right) dy \\
&\approx &\int_{\mathbb{R}^{n}}\left\vert I_{\beta }^{n}F\left( y\right)
\right\vert ^{q}\left\vert y\right\vert ^{m-\gamma q}dy,
\end{eqnarray*}%
since 
\begin{equation*}
\int_{\mathbb{R}^{m}}\left\vert \left( x,y\right) \right\vert ^{-\gamma
q}dx\approx \int_{\mathbb{R}^{m}}\left\vert \left( \left\vert x\right\vert
+\left\vert y\right\vert \right) \right\vert ^{-\gamma q}dx=\left\vert
y\right\vert ^{m-\gamma q}\int_{\mathbb{R}^{m}}\left\vert \left( \frac{%
\left\vert x\right\vert }{\left\vert y\right\vert }+1\right) \right\vert
^{-\gamma q}d\left( \frac{x}{\left\vert y\right\vert }\right) \approx
\left\vert y\right\vert ^{m-\gamma q}
\end{equation*}%
for $m-\gamma q<0$. We also have%
\begin{eqnarray*}
\int_{\mathbb{R}^{n}}\left\vert F\left( t\right) \right\vert ^{p}\left\vert
t\right\vert ^{\left( \delta -\frac{m}{p^{\prime }}\right) p}dt &=&\int_{%
\mathbb{R}^{n}}\left\vert \int_{\mathbb{R}^{m}}f\left( u,t\right)
du\right\vert ^{p}\left\vert t\right\vert ^{\left( \delta -\frac{m}{%
p^{\prime }}\right) p}dt \\
&\leq &\int_{\mathbb{R}^{n}}\left\{ \int_{\mathbb{R}^{m}}\left\vert f\left(
u,t\right) \right\vert ^{p}\left( \left\vert u\right\vert +\left\vert
t\right\vert \right) ^{\delta p}du\right\} \left\{ \int_{\mathbb{R}%
^{m}}\left( \left\vert u\right\vert +\left\vert t\right\vert \right)
^{-\delta p^{\prime }}du\right\} ^{p-1}\left\vert t\right\vert ^{\left(
\delta -\frac{m}{p^{\prime }}\right) p}dt \\
&\approx &\int_{\mathbb{R}^{n}}\left\{ \int_{\mathbb{R}^{m}}\left\vert
f\left( u,t\right) \right\vert ^{p}\left( \left\vert u\right\vert
+\left\vert t\right\vert \right) ^{\delta p}du\right\} \left\{ \left\vert
t\right\vert ^{m-\delta p^{\prime }}\right\} ^{p-1}\left\vert t\right\vert
^{\left( \delta -\frac{m}{p^{\prime }}\right) p}dt \\
&=&\int \int_{\mathbb{R}^{m}\times \mathbb{R}^{n}}\left\vert f\left(
u,t\right) \right\vert ^{p}\left( \left\vert u\right\vert +\left\vert
t\right\vert \right) ^{\delta p}dudt,
\end{eqnarray*}%
since%
\begin{equation*}
\int_{\mathbb{R}^{m}}\left( \left\vert u\right\vert +\left\vert t\right\vert
\right) ^{-\delta p^{\prime }}du\approx \left\vert t\right\vert ^{m-\delta
p^{\prime }}
\end{equation*}%
for $m-\delta p^{\prime }<0$. Thus we conclude that (\ref{pwni}) holds
provided we have the $1$-parameter power weight norm inequality%
\begin{equation*}
\left( \int_{\mathbb{R}^{n}}\left\vert I_{\beta }^{n}F\left( y\right)
\right\vert ^{q}\left\vert y\right\vert ^{-\left( \gamma -\frac{m}{q}\right)
q}dy\right) ^{\frac{1}{q}}\lesssim \left( \int_{\mathbb{R}^{n}}\left\vert
F\left( t\right) \right\vert ^{p}\left\vert t\right\vert ^{\left( \delta -%
\frac{m}{p^{\prime }}\right) p}dt\right) ^{\frac{1}{p}}.
\end{equation*}%
But this inequality holds by Theorem \ref{Stein-Weiss} since the $1$%
-parameter power weight equality holds,%
\begin{eqnarray*}
\beta -\left( \gamma -\frac{m}{q}\right) -\left( \delta -\frac{m}{p^{\prime }%
}\right) &=&\beta -\gamma -\delta +m\left( \frac{1}{q}+\frac{1}{p^{\prime }}%
\right) \\
&=&\alpha +\beta -\gamma -\delta +m\left( \frac{1}{q}+\frac{1}{p^{\prime }}%
-1\right) \\
&=&\left( m+n\right) \Gamma -m\Gamma =n\Gamma ,
\end{eqnarray*}%
and\ each of the following four constraint inequalities holds,%
\begin{eqnarray*}
&&0<\beta <n, \\
&&\left( \gamma -\frac{m}{q}\right) <\frac{n}{q}, \\
&&\left( \delta -\frac{m}{p^{\prime }}\right) <\frac{n}{p^{\prime }}, \\
&&\left( \gamma -\frac{m}{q}\right) +\left( \delta -\frac{m}{p^{\prime }}%
\right) \geq 0.
\end{eqnarray*}

A similar argument shows that (\ref{pwni}) holds when $\beta =n$ and $p\leq
q $ and $A_{p,q}^{\left( \alpha ,\beta \right) ,\left( m,n\right) }\left(
v_{\delta },w_{\gamma }\right) <\infty $.

\section{Proof of Theorem \protect\ref{testing}}

Define the eccentricity $\varkappa \left( R\right) $ of a rectangle $%
R=I\times J$ to be $\varkappa \left( R\right) =\frac{\ell \left( I\right) }{%
\ell \left( J\right) }$. For $j\in \mathbb{Z}$ define the conical operator $%
\bigtriangleup _{j}I_{\alpha ,\beta }$ acting on a measure $\mu $ by%
\begin{equation*}
\bigtriangleup _{j}I_{\alpha ,\beta }\mu \left( x,y\right) \equiv
\diint\limits_{\mathbb{S}_{j}+\left( x,y\right) }\left( \frac{1}{\left\vert
x-u\right\vert }\right) ^{m-\alpha }\left( \frac{1}{\left\vert
y-t\right\vert }\right) ^{n-\beta }d\mu \left( u,t\right) ,
\end{equation*}%
where $\mathbb{S}_{j}\equiv \left\{ \left( x,y\right) \in \mathbb{R}%
^{m}\times \mathbb{R}^{n}:2^{-j-1}\leq \frac{\left\vert y\right\vert }{%
\left\vert x\right\vert }<2^{-j+1}\right\} $ is a cone with aperature
roughly $2^{-j}$ and slope roughly $2^{-j}$.

\begin{lemma}
\label{gain}Suppose $1<p<q<\infty $, $0<\alpha ,\beta <1$ and that both $%
\sigma $ and $\omega $ are rectangle doubling, and that the reverse doubling
exponent $\varepsilon $ for $\sigma $ satisfies%
\begin{equation*}
1-\varepsilon <\frac{\alpha }{m}=\frac{\beta }{n}.
\end{equation*}%
For $\frac{1}{p}-\frac{1}{q}<\frac{\alpha }{m}=\frac{\beta }{n}$ we have%
\begin{equation}
\left( \diint\limits_{\mathbb{R}^{m}\times \mathbb{R}^{n}}\left\vert
\bigtriangleup _{j}I_{\alpha ,\beta }\left( \mathbf{1}_{R}\sigma \right)
\left( x,y\right) \right\vert ^{q}d\omega \left( x,y\right) \right) ^{\frac{1%
}{q}}\lesssim 2^{-\varepsilon ^{\prime }\left\vert j-k\right\vert }\left(
\diint\limits_{R}d\sigma \right) ^{\frac{1}{p}},  \label{cone block decay}
\end{equation}%
for all rectangles $R$ with eccentricity $\varkappa \left( R\right) =2^{-k}$%
, and where $\varepsilon ^{\prime }>0$ depends on $p,q,\alpha ,\beta
,\varepsilon $.
\end{lemma}

\begin{proof}
We will prove the special case when $\ell \left( I\right) =1$ and $\varkappa
\left( R\right) =1$, i.e. $R$ is a square of side length $1$. The general
case is similar. We now place the origin so that $R$ is the block $B_{0}=%
\left[ 1,2\right] ^{m}\times \left[ 1,2\right] ^{n}$.

We first consider the region $\mathcal{R}_{1}\equiv \left\{ \left(
x,y\right) :\left\vert x\right\vert \leq \frac{1}{4}\left\vert y\right\vert
\right\} $. In this region the sum over $j<0$ is easy. So we consider $j>0$.
To see that (\ref{cone block decay}) holds with integration on the left
restricted to $\mathcal{R}_{1}$, we begin by noting that%
\begin{eqnarray*}
&&\bigtriangleup _{j}I_{\alpha ,\beta }\left( \mathbf{1}_{R}\sigma \right)
\left( x,y\right) =\diint\limits_{R\cap \left\{ \mathbb{S}_{j}+\left(
x,y\right) \right\} }\left( \frac{1}{\left\vert x-u\right\vert }\right)
^{m-\alpha }\left( \frac{1}{\left\vert y-t\right\vert }\right) ^{n-\beta
}d\sigma \left( u,t\right) \\
&=&\sum_{r=1}^{2^{j}}\left[ \diint\limits_{R\left( r\right) }\left( \frac{1}{%
\left\vert x-u\right\vert }\right) ^{m-\alpha }\left( \frac{1}{\left\vert
y-t\right\vert }\right) ^{n-\beta }d\sigma \left( u,t\right) \right] \mathbf{%
1}_{R^{\ast }\left( r\right) }\left( x,y\right) ,
\end{eqnarray*}%
where the tiles $R\left( r\right) $ are rectangles of size $1$ by $2^{-j}$
and the tiles $R^{\ast }\left( r\right) $ are slightly enlarged reflections
of the $R\left( r\right) $ centered on the $y$-axis. Now we continue by
computing%
\begin{eqnarray*}
&&\diint\limits_{\mathcal{R}_{1}}\left\vert \bigtriangleup _{j}I_{\alpha
,\beta }\left( \mathbf{1}_{R}\sigma \right) \left( x,y\right) \right\vert
^{q}d\omega \left( x,y\right) \\
&=&\diint\limits_{\mathbb{R}^{m}\times \mathbb{R}^{n}}\left\{
\sum_{r=1}^{2^{j}}\left[ \diint\limits_{R\left( r\right) }\left( \frac{1}{%
\left\vert x-u\right\vert }\right) ^{m-\alpha }\left( \frac{1}{\left\vert
y-t\right\vert }\right) ^{n-\beta }d\sigma \left( u,t\right) \right] \mathbf{%
1}_{R^{\ast }\left( r\right) }\left( x,y\right) \right\} ^{q}d\omega \left(
x,y\right) \\
&\approx &\sum_{r=1}^{2^{j}}\diint\limits_{R^{\ast }\left( r\right) }\left\{
\diint\limits_{R\left( r\right) }\left( \frac{1}{\left\vert x-u\right\vert }%
\right) ^{m-\alpha }\left( \frac{1}{\left\vert y-t\right\vert }\right)
^{n-\beta }d\sigma \left( u,t\right) \right\} ^{q}d\omega \left( x,y\right) ,
\end{eqnarray*}%
where matters have been reduced to the diagonal terms since $R^{\ast }\left(
r\right) \cap R^{\ast }\left( s\right) =\emptyset $ unless $\left\vert
r-s\right\vert \leq c_{0}$ (this follows easily from the fact that the tiles 
$R\left( r\right) $ are pairwise disjoint in $r$). Now we apply H\"{o}lder's
inequality to obtain%
\begin{eqnarray}
&&\diint\limits_{\mathbb{R}^{m}\times \mathbb{R}^{n}}\left\vert I_{\alpha
,\beta }\left( \mathbf{1}_{B_{0}}\sigma \right) \right\vert ^{q}w^{q}
\label{apply H} \\
&\lesssim &\sum_{r=1}^{2^{j}}\diint\limits_{R^{\ast }\left( r\right)
}\left\{ \diint\limits_{R\left( r\right) }d\sigma \right\} ^{\frac{q}{p}%
}\left\{ \diint\limits_{R\left( r\right) }\left( \frac{1}{\left\vert
x-u\right\vert }\right) ^{\left( m-\alpha \right) p^{\prime }}\left( \frac{1%
}{\left\vert y-t\right\vert }\right) ^{\left( n-\beta \right) p^{\prime
}}d\sigma \left( u,t\right) \right\} ^{\frac{q}{p^{\prime }}}w\left(
x,y\right) ^{q}dxdy  \notag \\
&\lesssim &\left( A_{p,q}^{\alpha ,\beta }\right)
^{q}\sum_{r=1}^{2^{j}}\left\{ \diint\limits_{R\left( r\right) }d\sigma
\right\} ^{\frac{q}{p}}=\left( A_{p,q}^{\alpha ,\beta }\right)
^{q}\sum_{r=1}^{2^{j}}\left\{ \diint\limits_{R\left( r\right) }d\sigma
\right\} ^{\frac{q}{p}-1}\left( \diint\limits_{R\left( r\right) }d\sigma
\right)  \notag \\
&\lesssim &\left( A_{p,q}^{\alpha ,\beta }\right)
^{q}\sum_{r=1}^{2^{j}}\left\{ 2^{-\varepsilon j}\diint\limits_{R}d\sigma
\right\} ^{\frac{q}{p}-1}\left( \diint\limits_{R\left( r\right) }d\sigma
\right) =2^{-\varepsilon ^{\prime }j}\left( \diint\limits_{R}d\sigma \right)
^{\frac{q}{p}}.  \notag
\end{eqnarray}

Now we consider the region $\mathcal{R}_{2}\equiv \left\{ \left( x,y\right) :%
\frac{1}{4}\left\vert y\right\vert \leq \left\vert x\right\vert \leq
4\left\vert y\right\vert \right\} $. Here we must perform an additional
calculation involving the intersection of the tile $R\left( r\right) $ and
the cone $\mathbb{S}_{j}+\left( x,y\right) $:%
\begin{eqnarray*}
&&\diint\limits_{\mathcal{R}_{2}}\left\vert \bigtriangleup _{j}I_{\alpha
,\beta }\left( \mathbf{1}_{R\left( r\right) \cap \left[ \mathbb{S}%
_{j}+\left( x,y\right) \right] }\sigma \right) \left( x,y\right) \right\vert
^{q}d\omega \left( x,y\right) \\
&=&\diint\limits_{\mathcal{R}_{2}}\left\{ \sum_{r=1}^{2^{j}}\left[
\diint\limits_{R\left( r\right) \cap \left[ \mathbb{S}_{j}+\left( x,y\right) %
\right] }\left( \frac{1}{\left\vert x-u\right\vert }\right) ^{m-\alpha
}\left( \frac{1}{\left\vert y-t\right\vert }\right) ^{n-\beta }d\sigma
\left( u,t\right) \right] \mathbf{1}_{R^{\ast }\left( r\right) }\left(
x,y\right) \right\} ^{q}d\omega \left( x,y\right) ,
\end{eqnarray*}%
where now $R^{\ast }\left( r\right) $ can overlap $R\left( r\right) $
considerably. However, for each fixed $\left( x,y\right) \in R\left(
r\right) $ we further decompose 
\begin{equation*}
R\left( r\right) \cap \left[ \mathbb{S}_{j}+\left( x,y\right) \right]
=\dbigcup\limits_{\ell }R_{\ell }\left( r\right)
\end{equation*}%
where the widths of the $R_{\ell }\left( r\right) $ form a geometric
sequence that approaches $0$ as $R_{\ell }\left( r\right) $ approaches $%
\left( x,y\right) $, and the dependence of $R_{\ell }\left( r\right) $ on $%
\left( x,y\right) $ is suppressed. Moreover, the tiles $R_{\ell }\left(
r\right) $ are roughly rectangles of dimension $2^{-\ell }\times 2^{-\ell
-j} $, and so%
\begin{eqnarray*}
&&\diint\limits_{R\left( r\right) \cap \left[ \mathbb{S}_{j}+\left(
x,y\right) \right] }\left( \frac{1}{\left\vert x-u\right\vert }\right)
^{m-\alpha }\left( \frac{1}{\left\vert y-t\right\vert }\right) ^{n-\beta
}d\sigma \left( u,t\right) \lesssim \sum_{\ell =1}^{\infty
}\diint\limits_{R_{\ell }\left( r\right) }\left( \frac{1}{2^{-\ell }}\right)
^{m-\alpha }\left( \frac{1}{2^{-\ell -j}}\right) ^{n-\beta }d\sigma \left(
u,t\right) \\
&\approx &2^{j\left( n-\beta \right) }\sum_{\ell =1}^{\infty }2^{\ell \left(
m-\alpha +n-\beta \right) }\left\vert R_{\ell }\left( r\right) \right\vert
_{\sigma }\lesssim 2^{j\left( n-\beta \right) }\sum_{\ell =1}^{\infty
}2^{\ell \left( m-\alpha +n-\beta \right) }\left( 2^{-\varepsilon \left(
m\ell +n\ell +jn\right) }\left\vert R\right\vert _{\sigma }\right) \\
&\lesssim &2^{j\left( n-\beta -\varepsilon n\right) }\sum_{\ell =1}^{\infty
}2^{\ell \left( m-\alpha +n-\beta -\left( m+n\right) \varepsilon \right)
}\left\vert R\right\vert _{\sigma }\approx 2^{jn\left( 1-\frac{\beta }{n}%
-\varepsilon \right) }\left\vert R\right\vert _{\sigma }
\end{eqnarray*}%
provided $1-\varepsilon <\frac{\alpha }{m}=\frac{\beta }{n}$. Thus we have%
\begin{eqnarray*}
&&\diint\limits_{\mathcal{R}_{2}}\left\vert \bigtriangleup _{j}I_{\alpha
,\beta }\left( \mathbf{1}_{R}\sigma \right) \right\vert ^{q}d\omega \\
&=&\diint\limits_{\mathcal{R}_{2}}\left\{ \sum_{r=1}^{2^{j}}\sum_{\ell
=1}^{N}\left[ \diint\limits_{R_{\ell }\left( r\right) \cap \left[ \mathbb{S}%
_{j}+\left( x,y\right) \right] }\left( \frac{1}{\left\vert x-u\right\vert }%
\right) ^{m-\alpha }\left( \frac{1}{\left\vert y-t\right\vert }\right)
^{n-\beta }d\sigma \left( u,t\right) \right] \mathbf{1}_{R_{\ell }^{\ast
}\left( r\right) }\left( x,y\right) \right\} ^{q}d\omega \left( x,y\right) \\
&\lesssim &\diint\limits_{\mathcal{R}_{2}}\left\{
\sum_{r=1}^{2^{j}}\sum_{\ell =1}^{N}\left[ 2^{jn\left( 1-\frac{\beta }{n}%
-\varepsilon \right) }\left\vert R\right\vert _{\sigma }\right] \mathbf{1}%
_{R_{\ell }^{\ast }\left( r\right) }\left( x,y\right) \right\} ^{q}d\omega
\left( x,y\right) \\
&\lesssim &\sum_{r=1}^{2^{j}}\diint\limits_{\mathcal{R}_{2}}\left\{
\sum_{\ell =1}^{N}\left[ 2^{jn\left( 1-\frac{\beta }{n}-\varepsilon \right)
}\left\vert R\right\vert _{\sigma }\right] \mathbf{1}_{R_{\ell }^{\ast
}\left( r\right) }\left( x,y\right) \right\} ^{q}d\omega \left( x,y\right) \\
&\approx &\sum_{r=1}^{2^{j}}\sum_{\ell =1}^{N}\diint\limits_{\mathcal{R}%
_{2}}\left\{ \left[ 2^{jn\left( 1-\frac{\beta }{n}-\varepsilon \right)
}\left\vert R\right\vert _{\sigma }\right] \mathbf{1}_{R_{\ell }^{\ast
}\left( r\right) }\left( x,y\right) \right\} ^{q}d\omega \left( x,y\right) ,
\end{eqnarray*}%
since the $R_{\ell }^{\ast }\left( r\right) $ are essentially pairwise
disjoint in both $r$ and $\ell $ (there is also decay in the kernel in the
parameter $\ell $). Now we apply H\"{o}lder's inequality and continue as in (%
\ref{apply H}) above.

Region $\mathcal{R}_{3}\equiv \left\{ \left( x,y\right) :\left\vert
x\right\vert \leq 4\left\vert y\right\vert \right\} $ is handled
symmetrically to Region $\mathcal{R}_{1}$, and this completes the proof of
Lemma \ref{gain}.
\end{proof}

Now we can easily obtain the testing condition.

\begin{corollary}
Suppose $1<p<q<\infty $, $0<\alpha ,\beta <1$ and that both $\sigma $ and $%
\omega $ are rectangle doubling, and that the reverse doubling exponent $%
\varepsilon $ for $\sigma $ satisfies%
\begin{equation*}
1-\varepsilon <\frac{\alpha }{m}=\frac{\beta }{n}.
\end{equation*}%
For $\frac{1}{p}-\frac{1}{q}<\frac{\alpha }{m}=\frac{\beta }{n}$ we have the
testing condition%
\begin{equation*}
\left( \diint\limits_{\mathbb{R}^{m}\times \mathbb{R}^{n}}\left\vert
I_{\alpha ,\beta }\left( \mathbf{1}_{R}\sigma \right) \right\vert
^{q}d\omega \right) ^{\frac{1}{q}}\lesssim \mathbb{A}_{p,q}^{\left( \alpha
,\beta \right) ,\left( m,n\right) }\left( \sigma ,\omega \right) \left(
\diint\limits_{R}d\sigma \right) ^{\frac{1}{p}},\ \ \ \ \ \text{for all }%
R=I\times J,
\end{equation*}%
as well as the dual testing condition in which the roles of the measures $%
\sigma $ and $\omega $ are reversed, and the exponents $p,q$ are replaced
with $q^{\prime },p^{\prime }$ respectively.
\end{corollary}

\begin{proof}
We prove the testing condition, and leave the dual testing condition to the
reader. Suppose that $R$ has eccentricity $\varkappa \left( R\right) =2^{-k}$%
. Then from Minkowski's inequality and Lemma \ref{gain} we have%
\begin{eqnarray*}
\left( \diint\limits_{\mathbb{R}^{m}\times \mathbb{R}^{n}}\left\vert
I_{\alpha ,\beta }\left( \mathbf{1}_{R}\sigma \right) \right\vert
^{q}d\omega \right) ^{\frac{1}{q}} &=&\left( \diint\limits_{\mathbb{R}%
^{m}\times \mathbb{R}^{n}}\left\vert \sum_{i\in \mathbb{Z}}I_{\alpha ,\beta
}\left( \mathbf{1}_{R}\sigma \right) \right\vert ^{q}d\omega \right) ^{\frac{%
1}{q}} \\
&\leq &\sum_{i\in \mathbb{Z}}\left( \diint\limits_{\mathbb{R}^{m}\times 
\mathbb{R}^{n}}\left\vert I_{\alpha ,\beta }\left( \mathbf{1}_{R}\sigma
\right) \right\vert ^{q}d\omega \right) ^{\frac{1}{q}} \\
&\leq &\sum_{i\in \mathbb{Z}}2^{-\varepsilon ^{\prime }\left\vert
j-k\right\vert }\left( \diint\limits_{R}d\sigma \right) ^{\frac{1}{p}%
}=C_{\varepsilon ^{\prime }}\left( \diint\limits_{R}d\sigma \right) ^{\frac{1%
}{p}}.
\end{eqnarray*}
\end{proof}

\section{Appendix}

\subsection{Sharpness in the Muckenhoupt-Wheeden theorem}

Here we give a sharp form of the Muckenhoupt-Wheeden Theorem \ref{Muck and
Wheed}.

\begin{theorem}
\label{Muck and Wheed sharp}Let $0<\alpha <m$, $1<p,q<\infty $, and let $%
w\left( x\right) $ be a nonnegative weight on $\mathbb{R}^{m}$. Then%
\begin{equation}
\left\{ \int_{\mathbb{R}^{m}}I_{\alpha }^{m}f\left( x\right) ^{q}\ w\left(
x\right) ^{q}\ dx\right\} ^{\frac{1}{q}}\leq N_{p,q}^{\alpha ,m}\left(
w\right) \left\{ \int_{\mathbb{R}^{m}}f\left( x\right) ^{p}\ w\left(
x\right) ^{p}\ dx\right\} ^{\frac{1}{p}}  \label{wni MW}
\end{equation}%
for all $f\geq 0$ and $N_{p,q}^{\alpha ,m}\left( w\right) <\infty $ \emph{if
and only if} both%
\begin{equation}
\frac{1}{p}-\frac{1}{q}=\frac{\alpha }{m},  \label{bal MW'}
\end{equation}%
and%
\begin{equation*}
A_{p,q}\left( w\right) \equiv \sup_{\text{cubes }I\subset \mathbb{R}%
^{m}}\left( \frac{1}{\left\vert I\right\vert }\int_{I}w\left( x\right) ^{q}\
dx\right) ^{\frac{1}{q}}\left( \frac{1}{\left\vert I\right\vert }%
\int_{I}w\left( x\right) ^{-p^{\prime }}\ dx\right) ^{\frac{1}{p^{\prime }}%
}<\infty .
\end{equation*}
\end{theorem}

\begin{proof}
Given Theorem \ref{Muck and Wheed}, it remains only to prove that (\ref{bal
MW'}) is necessary for (\ref{wni MW}). To see this, we apply H\"{o}lder's
inequality with dual exponents $\frac{p^{\prime }+1}{p^{\prime }}$ and $%
p^{\prime }+1$ to obtain%
\begin{eqnarray*}
1 &=&\left\{ \frac{1}{\left\vert I\right\vert }\int_{I}w^{\frac{p^{\prime }}{%
p^{\prime }+1}}w^{-\frac{p^{\prime }}{p^{\prime }+1}}\right\} ^{\frac{%
p^{\prime }+1}{p^{\prime }}} \\
&\leq &\left\{ \left( \frac{1}{\left\vert I\right\vert }\int_{I}w\right) ^{%
\frac{p^{\prime }}{p^{\prime }+1}}\left( \frac{1}{\left\vert I\right\vert }%
\int_{I}w^{-p^{\prime }}\right) ^{\frac{1}{p^{\prime }+1}}\right\} ^{\frac{%
p^{\prime }+1}{p^{\prime }}} \\
&=&\left( \frac{1}{\left\vert I\right\vert }\int_{I}w\right) \left( \frac{1}{%
\left\vert I\right\vert }\int_{I}w^{-p^{\prime }}\right) ^{\frac{1}{%
p^{\prime }}}\leq \left( \frac{1}{\left\vert I\right\vert }%
\int_{I}w^{q}\right) ^{\frac{1}{q}}\left( \frac{1}{\left\vert I\right\vert }%
\int_{I}w^{-p^{\prime }}\right) ^{\frac{1}{p^{\prime }}}.
\end{eqnarray*}%
Then we conclude that%
\begin{equation*}
\left\vert I\right\vert ^{\frac{\alpha }{m}-1+\frac{1}{q}+\frac{1}{p^{\prime
}}}\leq \left\vert I\right\vert ^{\frac{\alpha }{m}-1+\frac{1}{q}+\frac{1}{%
p^{\prime }}}\left( \frac{1}{\left\vert I\right\vert }\int_{I}w^{q}\right) ^{%
\frac{1}{q}}\left( \frac{1}{\left\vert I\right\vert }\int_{I}w^{-p^{\prime
}}\right) ^{\frac{1}{p^{\prime }}}\leq A_{p,q}^{\alpha ,m}\left( w\right)
\leq N_{p,q}^{\alpha ,m}\left( w\right)
\end{equation*}%
for all cubes $I$, which implies $\frac{\alpha }{m}-1+\frac{1}{q}+\frac{1}{%
p^{\prime }}=0$ as required.
\end{proof}

\subsection{Sharpness in the Stein-Weiss theorem}

Here we give a sharp form of the Stein-Weiss Theorem \ref{Stein-Weiss}.

\begin{theorem}
\label{Stein-Weiss sharp}Let 
\begin{eqnarray*}
-\infty &<&\alpha ,\beta ,\gamma ,\delta <\infty , \\
1 &<&p,q<\infty , \\
m &\in &\mathbb{N}.
\end{eqnarray*}%
Then the power weighted norm inequality%
\begin{equation}
\left\{ \int_{\mathbb{R}^{m}}I_{\alpha }^{m}f\left( x\right) ^{q}\
\left\vert x\right\vert ^{-\gamma q}\ dx\right\} ^{\frac{1}{q}}\leq
N_{p,q}\left( w,v\right) \left\{ \int_{\mathbb{R}^{m}}f\left( x\right) ^{p}\
\left\vert x\right\vert ^{\delta p}\ dx\right\} ^{\frac{1}{p}},\ \ \ \ \ 
\text{for all }f\geq 0,  \label{pwni'}
\end{equation}%
holds \emph{if and only if} the power weight equality%
\begin{equation}
\frac{1}{p}-\frac{1}{q}=\frac{\alpha -\left( \gamma +\delta \right) }{m}
\label{pwe'}
\end{equation}%
holds along with the constraint inequalities,%
\begin{equation}
0<\alpha <m,  \label{con 1}
\end{equation}%
\begin{equation}
p\leq q,  \label{con 2}
\end{equation}%
\begin{equation}
q\gamma <m,  \label{con 3}
\end{equation}%
\begin{equation}
p^{\prime }\delta <m,  \label{con 4}
\end{equation}%
and%
\begin{equation}
\gamma +\delta \geq 0.  \label{con 5}
\end{equation}
\end{theorem}

\begin{proof}
The sufficiency of these conditions for (\ref{pwni'}) is the classical
theorem of Stein and Weiss, so we turn to proving their necessity. The
required local integrability of the power weights $\left\vert x\right\vert
^{-\gamma q}$ and $\left\vert x\right\vert ^{-\delta p^{\prime }}$ shows
that (\ref{con 3}) and (\ref{con 4}) hold. The kernel $\left\vert
x-y\right\vert ^{\alpha -m}$ of the convolution operator must be locally
integrable on $\mathbb{R}^{m}$ and this implies that $0<\alpha $. Using
this, we can now use the argument in the proof of Proposition \ref{tails'}
to prove that finiteness of the Muckenhoupt characteristic $A_{p,q}\left(
w,v\right) $ is necessary, and this in turn implies both (\ref{con 5}) and
the power weight equality (\ref{pwe'}) just as in the proof of Theorem \ref%
{A char mn} above for the $2$-parameter case. From (\ref{pwe'}), (\ref{con 3}%
) and (\ref{con 4}) we now obtain%
\begin{equation*}
\alpha =m\left( \frac{1}{p}-\frac{1}{q}\right) +\left( \gamma +\delta
\right) <m\left( \frac{1}{p}-\frac{1}{q}\right) +\left( \frac{m}{q}+\frac{m}{%
p^{\prime }}\right) =m,
\end{equation*}%
which completes the proof that (\ref{con 1}) holds.

Finally we turn to proving (\ref{con 2}). Let $f\left( y\right) =f\left(
s\right) $, $s=\left\vert y\right\vert $, be a radial function on $\mathbb{R}%
^{m}$. Then $I_{\alpha }f\left( x\right) =I_{\alpha }f\left( r\right) $, $%
r=\left\vert x\right\vert $, is also radial and%
\begin{equation*}
I_{\alpha }f\left( x\right) =\int_{\mathbb{R}^{m}}\left\vert x-y\right\vert
^{\alpha -m}f\left( y\right) dy\gtrsim \int_{\left\vert y\right\vert \leq
\left\vert x\right\vert }\left\vert x\right\vert ^{\alpha -m}f\left(
y\right) dy=r^{\alpha -m}\int_{0}^{r}f\left( s\right) s^{m-1}ds.
\end{equation*}%
Now suppose, in order to derive a contradiction, that (\ref{pwni'}) holds.
Then we have%
\begin{eqnarray*}
&&\left\{ \int_{0}^{\infty }\left( \int_{0}^{r}f\left( s\right)
s^{m-1}ds\right) ^{q}\ r^{\left( \alpha -m\right) q-\gamma q}\
r^{n-1}dr\right\} ^{\frac{1}{q}}\lesssim \left\{ \int_{\mathbb{R}%
^{m}}I_{\alpha }^{m}f\left( x\right) ^{q}\ \left\vert x\right\vert ^{-\gamma
q}\ dx\right\} ^{\frac{1}{q}} \\
&&\ \ \ \ \ \leq N_{p,q}\left( w,v\right) \left\{ \int_{\mathbb{R}%
^{m}}f\left( x\right) ^{p}\ \left\vert x\right\vert ^{\delta p}\ dx\right\}
^{\frac{1}{p}}=N_{p,q}\left( w,v\right) \left\{ \int_{0}^{\infty }f\left(
s\right) \ s^{\delta p+m-1}ds\right\} ^{\frac{1}{p}}
\end{eqnarray*}%
for all $f\geq 0$. With $g\left( s\right) \equiv f\left( s\right) s^{m-1}$,
this last inequality can be rewritten as%
\begin{equation}
\left\{ \int_{0}^{\infty }\left( \int_{0}^{r}g\left( s\right) ds\right)
^{q}\ v\left( r\right) dr\right\} ^{\frac{1}{q}}\lesssim N_{p,q}\left(
w,v\right) \left\{ \int_{0}^{\infty }g\left( s\right) \ u\left( s\right)
ds\right\} ^{\frac{1}{p}}  \label{Hardy}
\end{equation}%
for all $g\geq 0$, and where the weights are given by $v\left( r\right)
=r^{\left( \alpha -m\right) q-\gamma q+n-1}$ and $u\left( s\right)
=s^{\delta p}$. By a result of Maz'ja \cite{Maz}, the two weight Hardy
inequality (\ref{Hardy}) with $q<p$ holds if and only if%
\begin{equation*}
\int_{0}^{\infty }\left[ \left( \int_{0}^{r}u^{1-p^{\prime }}\right) ^{\frac{%
1}{p^{\prime }}}\left( \int_{r}^{\infty }v\right) ^{\frac{1}{p}}\right]
^{\rho }v\left( r\right) dr<\infty ,
\end{equation*}%
where $\frac{1}{\rho }=\frac{1}{q}-\frac{1}{p}>0$. But since the weights $u$
and $v$ are power functions, the integrand above is also a power function,
and hence cannot belong to any Lebesgue space $L^{\rho }\left( 0,\infty
\right) $, thus providing the required contradiction.
\end{proof}

\subsection{Optimal powers of Muckenhoupt characteristics}

Recall the one parameter two-tailed characteristic $\widehat{A}_{p,q}$ given
above by%
\begin{equation*}
\widehat{A}_{p,q}\left( v,w\right) \equiv \sup_{Q\subset \mathbb{R}%
^{n}}\left( \frac{1}{\left\vert Q\right\vert }\int_{\mathbb{R}^{n}}\left[ 
\widehat{s}_{Q}\left( x\right) w\left( x\right) \right] ^{q}\ dx\right) ^{%
\frac{1}{q}}\left( \frac{1}{\left\vert Q\right\vert }\int_{\mathbb{R}^{n}}%
\left[ \widehat{s}_{Q}\left( x\right) v\left( x\right) ^{-1}\right]
^{p^{\prime }}\ dx\right) ^{\frac{1}{p^{\prime }}},
\end{equation*}%
where 
\begin{equation*}
\widehat{s}_{Q}\left( x\right) \equiv \left( 1+\frac{\left\vert
x-c_{Q}\right\vert }{\left\vert Q\right\vert ^{\frac{1}{n}}}\right) ^{\alpha
-n},\ \ \ \ \ c_{Q}\text{ is the center of }Q,
\end{equation*}%
and 
\begin{eqnarray*}
\frac{1}{q} &=&\frac{1}{p}-\frac{\alpha }{n}, \\
\text{i.e. }\frac{1}{q}+\frac{1}{p^{\prime }} &=&1-\frac{\alpha }{n},
\end{eqnarray*}%
From Theorem \ref{Saw and Wheed} we know that the characteristic $\widehat{A}%
_{p,q}\left( w,v\right) $ is finite \emph{if and only if} the following norm
inequality for the fractional integral $I_{\alpha }^{n}$ holds:%
\begin{equation}
\left\{ \int_{\mathbb{R}^{n}}I_{\alpha }^{n}f\left( x\right) ^{q}\ w\left(
x\right) ^{q}\ dx\right\} ^{\frac{1}{q}}\leq C_{p,q}\left( v,w\right)
\left\{ \int_{\mathbb{R}^{n}}f\left( x\right) ^{p}\ v\left( x\right) ^{p}\
dx\right\} ^{\frac{1}{p}}.  \label{strong type}
\end{equation}%
Moreover, it is claimed there that 
\begin{equation}
C_{p,q}\left( v,w\right) \approx \widehat{A}_{p,q}\left( v,w\right) ,
\label{C equiv A hat}
\end{equation}%
and this equivalence can be verified by carefully tracking the constants in
the proof of Theorem 1 in \cite{SaWh}. We also have the same consequence for
the one-tailed characteristic.

\begin{description}
\item[Porism] The same arguments as used in the proof of Theorem 1 in \cite%
{SaWh} also prove that 
\begin{equation*}
C_{\alpha ,\beta }\left( v,w\right) \approx \overline{A}_{p,q}\left(
v,w\right) ,
\end{equation*}%
where $\overline{A}_{p,q}$ is the one-tailed characteristic, 
\begin{eqnarray*}
\overline{A}_{p,q}\left( v,w\right) &=&\sup_{Q\subset \mathbb{R}^{n}}\left( 
\frac{1}{\left\vert Q\right\vert }\int_{Q}w^{q}\right) ^{\frac{1}{q}}\left( 
\frac{1}{\left\vert Q\right\vert }\int_{\mathbb{R}^{n}}\widehat{s}%
_{Q}^{p^{\prime }}v^{-p^{\prime }}\right) ^{\frac{1}{p^{\prime }}} \\
&&+\sup_{Q\subset \mathbb{R}^{n}}\left( \frac{1}{\left\vert Q\right\vert }%
\int_{\mathbb{R}^{n}}\widehat{s}_{Q}^{q}w^{q}\right) ^{\frac{1}{q}}\left( 
\frac{1}{\left\vert Q\right\vert }\int_{Q}v^{-p^{\prime }}\right) ^{\frac{1}{%
p^{\prime }}}.
\end{eqnarray*}
\end{description}

\subsubsection{Comparison with the inequality of Lacey, Moen, P\'{e}rez and
Torres}

Here we give a simple and instructive proof of the `$A_{1}$ conjecture' in
the setting of fractional integrals. Recall that from \cite{LaMoPeTo} we
have the estimate%
\begin{equation}
C_{p,q}\left( w\right) \lesssim A_{p,q}\left( w\right) ^{1+\max \left\{ 
\frac{p^{\prime }}{q},\frac{q}{p^{\prime }}\right\} },  \label{LMPT}
\end{equation}%
and if we restrict the two weight result above to the case $w=v$, we have
the equivalence%
\begin{equation*}
C_{\alpha ,\beta }\left( w,w\right) \approx \widehat{A}_{p,q}\left(
w,w\right) .
\end{equation*}%
Thus the estimate (\ref{LMPT}) is equivalent to 
\begin{equation}
\widehat{A}_{p,q}\left( w\right) \lesssim A_{p,q}\left( w\right) ^{\rho },
\label{hat inequ}
\end{equation}%
and also equivalent to%
\begin{equation}
\overline{A}_{p,q}\left( w\right) \lesssim A_{p,q}\left( w\right) ^{\rho },
\label{bar inequ}
\end{equation}%
where $\rho \equiv 1+\max \left\{ \frac{p^{\prime }}{q},\frac{q}{p^{\prime }}%
\right\} $. Written out in full, one half of inequality (\ref{bar inequ}) is%
\begin{eqnarray}
&&\sup_{Q\subset \mathbb{R}^{n}}\left( \frac{1}{\left\vert Q\right\vert }%
\int_{\mathbb{R}^{n}}\widehat{s}_{Q}^{q}w^{q}\right) ^{\frac{1}{q}}\left( 
\frac{1}{\left\vert Q\right\vert }\int_{Q}w^{-p^{\prime }}\right) ^{\frac{1}{%
p^{\prime }}}  \label{full} \\
&\lesssim &\left\{ \sup_{Q\subset \mathbb{R}^{n}}\left( \frac{1}{\left\vert
Q\right\vert }\int_{Q}w^{q}\right) ^{\frac{1}{q}}\left( \frac{1}{\left\vert
Q\right\vert }\int_{Q}w^{-p^{\prime }}\right) ^{\frac{1}{p^{\prime }}%
}\right\} ^{\rho }.  \notag
\end{eqnarray}

\begin{claim}
The inequality (\ref{bar inequ}) holds directly, without any reference to
norm inequalities at all.
\end{claim}

\begin{proof}
Take the $q^{th}$ power of the left hand side of (\ref{full}) and fix a cube 
$Q$ which comes close to achieving the supremum over all cubes. Then we write%
\begin{eqnarray}
&&\left( \frac{1}{\left\vert Q\right\vert }\int_{\mathbb{R}^{n}}\widehat{s}%
_{Q}^{q}w^{q}\right) \left( \frac{1}{\left\vert Q\right\vert }%
\int_{Q}w^{-p^{\prime }}\right) ^{\frac{q}{p^{\prime }}}  \label{control} \\
&\lesssim &\left( \frac{1}{\left\vert Q\right\vert }\int_{Q}w^{-p^{\prime
}}\right) ^{\frac{q}{p^{\prime }}}\sum_{k=0}^{\infty }2^{k\left[ \left(
\alpha -n\right) q+n\right] }\frac{1}{\left\vert 2^{k}Q\right\vert }%
\int_{2^{k}Q}w^{q}  \notag \\
&\lesssim &\left( \frac{1}{\left\vert Q\right\vert }\int_{Q}w^{-p^{\prime
}}\right) ^{\frac{q}{p^{\prime }}}\sum_{k=0}^{\infty }2^{k\left[ \left(
\alpha -n\right) q+n\right] }A_{p,q}\left( w\right) ^{q}\left( \frac{1}{%
\left\vert 2^{k}Q\right\vert }\int_{2^{k}Q}w^{-p^{\prime }}\right) ^{-\frac{q%
}{p^{\prime }}}  \notag \\
&=&A_{p,q}\left( w\right) ^{q}\sum_{k=0}^{\infty }\left( \frac{\left\vert
Q\right\vert _{w^{-p^{\prime }}}}{\left\vert 2^{k}Q\right\vert
_{w^{-p^{\prime }}}}\right) ^{\frac{q}{p^{\prime }}}\lesssim A_{p,q}\left(
w\right) ^{q}\sum_{k=0}^{\infty }\left( 2^{-k\delta }\right) ^{\frac{q}{%
p^{\prime }}}  \notag \\
&\lesssim &A_{p,q}\left( w\right) ^{q}\frac{1}{1-2^{-\delta \frac{q}{%
p^{\prime }}}}\approx A_{p,q}\left( w\right) ^{q}\frac{1}{\delta },  \notag
\end{eqnarray}%
where $\delta $ is the reverse doubling exponent for the $A_{p}$ weight $%
w^{-p^{\prime }}$,%
\begin{equation*}
\left\vert Q\right\vert _{w^{-p^{\prime }}}\leq C2^{-k\delta }\left\vert
2^{k}Q\right\vert _{w^{-p^{\prime }}}\text{ for all cubes }Q\text{.}
\end{equation*}

Now we claim that the reverse doubling exponents $\delta =\delta \left(
w^{-p^{\prime }}\right) $ and $\delta \left( w^{q}\right) $ for the weights $%
w^{-p^{\prime }}$ and $w^{q}$ satisfy%
\begin{equation}
\frac{1}{\delta \left( w^{-p^{\prime }}\right) }\leq C_{p,q,n}\
A_{p,q}\left( w\right) ^{p^{\prime }}\text{and }\frac{1}{\delta \left(
w^{q}\right) }\leq C_{p,q,n}\ A_{p,q}\left( w\right) ^{q}.  \label{delta}
\end{equation}%
Indeed, we have for all $f\geq 0$ and any $0<\varepsilon <1$,%
\begin{eqnarray*}
&&\left( \frac{1}{\left\vert Q\right\vert }\int_{Q}f\right) \left(
\int_{Q}w^{q}\right) ^{\frac{1}{q}} \\
&=&\left( \frac{1}{\left\vert Q\right\vert }\int_{Q}\left( f^{\varepsilon
}w\right) \left( f^{1-\varepsilon }\right) \left( w^{-1}\right) \right)
\left( \int_{Q}w^{q}\right) ^{\frac{1}{q}} \\
&\leq &\left( \frac{1}{\left\vert Q\right\vert }\int_{Q}\left(
f^{\varepsilon }w\right) ^{q}\right) ^{\frac{1}{q}}\left( \frac{1}{%
\left\vert Q\right\vert }\int_{Q}\left( f^{1-\varepsilon }\right) ^{\frac{pq%
}{q-p}}\right) ^{\frac{1}{p}-\frac{1}{q}}\left( \frac{1}{\left\vert
Q\right\vert }\int_{Q}w^{-p^{\prime }}\right) ^{\frac{1}{p^{\prime }}}\left(
\int_{Q}w^{q}\right) ^{\frac{1}{q}} \\
&=&\left( \frac{1}{\left\vert Q\right\vert }\int_{Q}w^{q}\right) ^{\frac{1}{q%
}}\left( \frac{1}{\left\vert Q\right\vert }\int_{Q}w^{-p^{\prime }}\right) ^{%
\frac{1}{p^{\prime }}}\left( \int_{Q}f^{\varepsilon q}w^{q}\right) ^{\frac{1%
}{q}}\left( \frac{1}{\left\vert Q\right\vert }\int_{Q}f^{\left(
1-\varepsilon \right) \frac{pq}{q-p}}\right) ^{\frac{1}{p}-\frac{1}{q}} \\
&\leq &A_{p,q}\left( w\right) \left( \int_{Q}f^{\varepsilon q}w^{q}\right) ^{%
\frac{1}{q}}\left( \frac{1}{\left\vert Q\right\vert }\int_{Q}f^{\left(
1-\varepsilon \right) \frac{pq}{q-p}}\right) ^{\frac{1}{p}-\frac{1}{q}},
\end{eqnarray*}%
and now plugging in $f=\mathbf{1}_{\frac{1}{3}Q}$ we get%
\begin{eqnarray*}
\left( \frac{\left\vert \frac{1}{3}Q\right\vert }{\left\vert Q\right\vert }%
\right) \left( \int_{Q}w^{q}\right) ^{\frac{1}{q}} &\leq &A_{p,q}\left(
w\right) \left( \int_{\frac{1}{3}Q}w^{q}\right) ^{\frac{1}{q}}\left( \frac{%
\left\vert \frac{1}{3}Q\right\vert }{\left\vert Q\right\vert }\right) ^{%
\frac{1}{p}-\frac{1}{q}}; \\
\frac{\left\vert Q\right\vert _{w^{q}}}{\left\vert \frac{1}{3}Q\right\vert
_{w^{q}}} &\leq &3^{n\left( 1+\frac{q}{p^{\prime }}\right) }A_{p,q}\left(
w\right) ^{q}.
\end{eqnarray*}%
With $\ell \left( Q\right) $ denoting the side length of $Q$, we have%
\begin{eqnarray*}
\left\vert 3Q\right\vert _{w^{q}} &\leq &\sum_{\alpha \in \left\{
-1,0,1\right\} ^{n}\setminus \left( 0,...,0\right) }\left\vert 3\left(
Q+\ell \left( Q\right) \mathbf{e}_{1}\right) \right\vert _{w^{q}} \\
&\leq &\sum_{\alpha \in \left\{ -1,0,1\right\} ^{n}\setminus \left(
0,...,0\right) }3^{n\left( 1+\frac{q}{p^{\prime }}\right) }A_{p,q}\left(
w\right) ^{q}\left\vert Q+\ell \left( Q\right) \mathbf{e}_{1}\right\vert
_{w^{q}} \\
&=&3^{n\left( 1+\frac{q}{p^{\prime }}\right) }A_{p,q}\left( w\right)
^{q}\left\vert 3Q\setminus Q\right\vert _{w^{q}}\ ,
\end{eqnarray*}%
and so%
\begin{equation*}
\frac{\left\vert Q\right\vert _{w^{q}}}{\left\vert 3Q\right\vert _{w^{q}}}=%
\frac{\left\vert 3Q\right\vert _{w^{q}}-\left\vert 3Q\setminus Q\right\vert
_{w^{q}}}{\left\vert 3Q\right\vert _{w^{q}}}\leq 1-\frac{1}{3^{n\left( 1+%
\frac{q}{p^{\prime }}\right) }A_{p,q}\left( w\right) ^{q}}\equiv \gamma \in
\left( 1-3^{-n\left( 1+\frac{q}{p^{\prime }}\right) },1\right) .
\end{equation*}%
Iterating, we get%
\begin{equation*}
\frac{\left\vert \frac{1}{3^{k}}Q\right\vert _{w^{q}}}{\left\vert
Q\right\vert _{w^{q}}}=\frac{\left\vert \frac{1}{3^{k}}Q\right\vert _{w^{q}}%
}{\left\vert \frac{1}{3^{k-1}}Q\right\vert _{w^{q}}}\frac{\left\vert \frac{1%
}{3^{k-1}}Q\right\vert _{w^{q}}}{\left\vert \frac{1}{3^{k-2}}Q\right\vert
_{w^{q}}}...\frac{\left\vert \frac{1}{3}Q\right\vert _{w^{q}}}{\left\vert
Q\right\vert _{w^{q}}}\leq \gamma ^{k},
\end{equation*}%
from which we obtain%
\begin{equation*}
\frac{\left\vert \frac{1}{2^{k}}Q\right\vert _{w^{q}}}{\left\vert
Q\right\vert _{w^{q}}}=\frac{\left\vert \frac{1}{3^{\frac{\ln 2}{\ln 3}k}}%
Q\right\vert _{w^{q}}}{\left\vert Q\right\vert _{w^{q}}}\leq C\gamma ^{\frac{%
\ln 2}{\ln 3}k}=C2^{\frac{\ln \gamma }{\ln 3}k}=C2^{-k\delta },
\end{equation*}%
where%
\begin{equation*}
\delta =\delta \left( w^{q}\right) =\frac{\ln \frac{1}{\gamma }}{\ln 3}=%
\frac{1}{\ln 3}\ln \frac{1}{1-\frac{1}{3^{n\left( 1+\frac{q}{p^{\prime }}%
\right) }A_{p,q}\left( w\right) ^{q}}}\approx \frac{1}{A_{p,q}\left(
w\right) ^{q}}.
\end{equation*}%
This proves the second assertion in (\ref{delta}), and the proof of the
first assertion is similar.

Thus from (\ref{control}) we obtain%
\begin{eqnarray*}
\overline{A}_{p,q}\left( w\right) &=&\sup_{Q}\left( \frac{1}{\left\vert
Q\right\vert }\int_{\mathbb{R}^{n}}\widehat{s}_{Q}^{q}w^{q}\right) ^{\frac{1%
}{q}}\left( \frac{1}{\left\vert Q\right\vert }\int_{Q}w^{-p^{\prime
}}\right) ^{\frac{1}{p^{\prime }}} \\
&\leq &C_{p,q,n}\ A_{p,q}\left( w\right) \left( \frac{1}{\delta \left(
w^{-p^{\prime }}\right) }\right) ^{\frac{1}{q}} \\
&\leq &C_{p,q,n}\ A_{p,q}\left( w\right) ^{1+\frac{q}{p^{\prime }}}.
\end{eqnarray*}%
Similarly we obtain that the expression%
\begin{eqnarray*}
&&\sup_{Q\subset \mathbb{R}^{n}}\left( \frac{1}{\left\vert Q\right\vert }%
\int_{\mathbb{R}^{n}}w^{q}\right) ^{\frac{1}{q}}\left( \frac{1}{\left\vert
Q\right\vert }\int_{Q}\left[ \widehat{s}_{Q}w^{-1}\right] ^{p^{\prime
}}\right) ^{\frac{1}{p^{\prime }}} \\
&=&\sup_{Q\subset \mathbb{R}^{n}}\left( \frac{1}{\left\vert Q\right\vert }%
\int_{Q}\left[ \widehat{s}_{Q}w^{-1}\right] ^{p^{\prime }}\right) ^{\frac{1}{%
p^{\prime }}}\left( \frac{1}{\left\vert Q\right\vert }\int_{\mathbb{R}^{n}}%
\left[ w^{-1}\right] ^{-\left( q^{\prime }\right) ^{\prime }}\right) ^{\frac{%
1}{\left( q^{\prime }\right) ^{\prime }}}
\end{eqnarray*}%
is dominated by%
\begin{equation*}
C_{q^{\prime },p^{\prime },n}\ A_{q^{\prime },p^{\prime }}\left(
w^{-1}\right) ^{1+\frac{p^{\prime }}{q}}=C_{p,q,n}\ A_{p,q}\left( w\right)
^{1+\frac{p^{\prime }}{q}}.
\end{equation*}%
This completes the proof of the claim.
\end{proof}

Thus we have given in the Claim above a simple direct proof of the following
theorem using the proof of Theorem 1 in \cite{SaWh}.

\begin{theorem}
(Lacey, Moen, P\'{e}rez and Torres \cite{LaMoPeTo}) With $p,q,n$ as above we
have%
\begin{equation*}
C_{p,q}\left( w\right) \leq C_{p,q,n}\ A_{p,q}\left( w\right) ^{1+\max
\left\{ \frac{p^{\prime }}{q},\frac{q}{p^{\prime }}\right\} }.
\end{equation*}
\end{theorem}

\subsection{The product fractional integral, the Dirac mass and a modified
example}

Here we\ consider both the two weight norm inequality (\ref{2 weight arb
meas}) and the two-tailed characteristic (\ref{A hat fraktur}) in the
special case when $\sigma =\delta _{\left( 0,0\right) }$, and show that they
are equivalent in this case. When $\sigma =\delta _{\left( 0,0\right) }$ the
two weight norm inequality (\ref{2 weight arb meas}) yields%
\begin{eqnarray*}
\mathbb{N}_{p,q}^{\left( \alpha ,\beta \right) ,\left( m,n\right) }\left(
\delta _{\left( 0,0\right) },\omega \right) &=&\left\{ \int_{\mathbb{R}%
^{m}}\int_{\mathbb{R}^{n}}I_{\alpha ,\beta }^{m,n}\delta _{\left( 0,0\right)
}\left( x,y\right) ^{q}\ d\omega \left( x,y\right) \right\} ^{\frac{1}{q}} \\
&=&\left\{ \int_{\mathbb{R}^{m}}\int_{\mathbb{R}^{n}}\left( \left\vert
x\right\vert ^{\alpha -m}\left\vert y\right\vert ^{\beta -n}\right) ^{q}\
d\omega \left( x,y\right) \right\} ^{\frac{1}{q}},
\end{eqnarray*}%
and since $\widehat{s}_{I\times J}\left( x,y\right) \equiv \left( 1+\frac{%
\left\vert x-c_{I}\right\vert }{\left\vert I\right\vert ^{\frac{1}{m}}}%
\right) ^{\alpha -m}\left( 1+\frac{\left\vert y-c_{J}\right\vert }{%
\left\vert J\right\vert ^{\frac{1}{n}}}\right) ^{\beta -n}$, (\ref{A hat
fraktur}) yields%
\begin{eqnarray*}
&&\widehat{\mathbb{A}}_{p,q}^{\left( \alpha ,\beta \right) ,\left(
m,n\right) }\left( \delta _{\left( 0,0\right) },\omega \right) \\
&=&\sup_{I\times J\subset \mathbb{R}^{m}\times \mathbb{R}^{n}}\left\vert
I\right\vert ^{\frac{\alpha }{m}-1}\left\vert J\right\vert ^{\frac{\beta }{n}%
-1}\left( \diint\limits_{\mathbb{R}^{m}\times \mathbb{R}^{n}}\widehat{s}%
_{I\times J}\left( x,y\right) ^{q}d\omega \left( x,y\right) \right) ^{\frac{1%
}{q}}\left( \diint\limits_{\mathbb{R}^{m}\times \mathbb{R}^{n}}\widehat{s}%
_{I\times J}\left( u,t\right) ^{p^{\prime }}d\delta _{\left( 0,0\right)
}\left( u,t\right) \right) ^{\frac{1}{p^{\prime }}} \\
&=&\sup_{I\times J\subset \mathbb{R}^{m}\times \mathbb{R}^{n}}\left(
\diint\limits_{\mathbb{R}^{m}\times \mathbb{R}^{n}}\left\vert I\right\vert
^{\left( \frac{\alpha }{m}-1\right) q}\left( 1+\frac{\left\vert
x-c_{I}\right\vert }{\left\vert I\right\vert ^{\frac{1}{m}}}\right) ^{\left(
\alpha -m\right) q}\left\vert J\right\vert ^{\left( \frac{\beta }{n}%
-1\right) q}\left( 1+\frac{\left\vert y-c_{J}\right\vert }{\left\vert
J\right\vert ^{\frac{1}{n}}}\right) ^{\left( \beta -n\right) q}d\omega
\left( x,y\right) \right) ^{\frac{1}{q}} \\
&&\ \ \ \ \ \ \ \ \ \ \ \ \ \ \ \ \ \ \ \ \ \ \ \ \ \ \ \ \ \ \ \ \ \ \ \ \
\ \ \ \times \left( 1+\frac{\left\vert c_{I}\right\vert }{\left\vert
I\right\vert ^{\frac{1}{m}}}\right) ^{\alpha -m}\left( 1+\frac{\left\vert
c_{J}\right\vert }{\left\vert J\right\vert ^{\frac{1}{n}}}\right) ^{\beta -n}
\\
&=&\sup_{I\times J\subset \mathbb{R}^{m}\times \mathbb{R}^{n}}\left(
\diint\limits_{\mathbb{R}^{m}\times \mathbb{R}^{n}}\left( \left\vert
I\right\vert ^{\frac{1}{m}}+\left\vert x-c_{I}\right\vert \right) ^{\left(
\alpha -m\right) q}\left( 1+\frac{\left\vert c_{I}\right\vert }{\left\vert
I\right\vert ^{\frac{1}{m}}}\right) ^{\left( \alpha -m\right) q}\left(
\left\vert J\right\vert ^{\frac{1}{n}}+\left\vert y-c_{J}\right\vert \right)
^{\left( \beta -n\right) q}\left( 1+\frac{\left\vert c_{J}\right\vert }{%
\left\vert J\right\vert ^{\frac{1}{n}}}\right) ^{\left( \beta -n\right)
q}d\omega \left( x,y\right) \right) ^{\frac{1}{q}}.
\end{eqnarray*}%
If $I\times J$ has center $\left( c_{I},c_{J}\right) =\left( 0,0\right) $,
then this supremum is equal to or greater than%
\begin{equation*}
\left( \diint\limits_{\mathbb{R}^{m}\times \mathbb{R}^{n}}\left\vert
x\right\vert ^{\left( \alpha -m\right) q}\left\vert y\right\vert ^{\left(
\beta -n\right) q}d\omega \left( x,y\right) \right) ^{\frac{1}{q}}=\mathbb{N}%
_{p,q}^{\left( \alpha ,\beta \right) ,\left( m,n\right) }\left( \delta
_{\left( 0,0\right) },\omega \right) .
\end{equation*}%
Thus from Lemma \ref{tails} we conclude that%
\begin{equation*}
\mathbb{N}_{p,q}^{\left( \alpha ,\beta \right) ,\left( m,n\right) }\left(
\delta _{\left( 0,0\right) },\omega \right) \approx \widehat{\mathbb{A}}%
_{p,q}^{\left( \alpha ,\beta \right) ,\left( m,n\right) }\left( \delta
_{\left( 0,0\right) },\omega \right) .
\end{equation*}

\subsection{Reverse doubling and Muckenhoupt type conditions}

Recall the reverse doubling condition with exponent $\varepsilon >0$, and
rectangle reverse doubling condition with exponent pair $\left( \varepsilon
^{1},\varepsilon ^{2}\right) $.

\begin{definition}
We say that a measure $\mu $ on $\mathbb{R}^{m}$ satisfies the \emph{reverse
doubling condition} $\left( R_{ev}D_{oub}\right) $ with exponent $%
\varepsilon >0$ if 
\begin{equation*}
\left\vert sI\right\vert _{\mu }\leq Cs^{m\varepsilon }\ \left\vert
I\right\vert _{\mu },\ \ \ \ \ 0<s<1,
\end{equation*}%
for some constant $C$; and we say that a measure $\mu $ on $\mathbb{R}%
^{m}\times \mathbb{R}^{n}$ satisfies the \emph{rectangle reverse doubling
condition} $\left( R_{ect}R_{ev}D_{oub}\right) $ with exponent pair $\left(
\varepsilon ^{1},\varepsilon ^{2}\right) $ if 
\begin{equation*}
\left\vert sI\times tJ\right\vert _{\mu }\leq Cs^{m\varepsilon
^{1}}t^{n\varepsilon ^{2}}\ \left\vert I\times J\right\vert _{\mu },\ \ \ \
\ 0<s,t<1,
\end{equation*}%
for some constant $C$. If $\varepsilon =\varepsilon ^{1}=\varepsilon ^{2}$,
then we say simply that $\varepsilon $ is the reverse doubling exponent for $%
\mu $.
\end{definition}

Note that unlike doubling measures, nontrivial reverse doubling measures can
vanish on open subsets. In particular, Cantor measures are typically reverse
doubling, while never doubling. If both weights are reverse doubling, then
the two-tailed, one-tailed, and no-tailed Muckenhoupt conditions are all
equivalent.

\begin{lemma}
\label{reverse}Suppose that $\sigma $ and $\omega $ are reverse rectangle
doubling measures on $\mathbb{R}^{n}$, and that 
\begin{equation*}
0<\alpha <m,0<\beta <n,1<p,q<\infty .
\end{equation*}%
Then we have the following equivalence: 
\begin{equation*}
\widehat{\mathbb{A}}_{p,q}^{\left( \alpha ,\beta \right) ,\left( m,n\right)
}\left( \sigma ,\omega \right) \approx \overline{\mathbb{A}}_{p,q}^{\left(
\alpha ,\beta \right) ,\left( m,n\right) }\left( \sigma ,\omega \right)
\approx \mathbb{A}_{p,q}^{\left( \alpha ,\beta \right) ,\left( m,n\right)
}\left( \sigma ,\omega \right) .
\end{equation*}
\end{lemma}

\begin{proof}
We first consider the one-parameter equivalence 
\begin{equation}
\widehat{\mathbb{A}}_{p,q}^{\alpha ,m}\left( \sigma ,\omega \right) \approx 
\overline{\mathbb{A}}_{p,q}^{\alpha ,m}\left( \sigma ,\omega \right) \approx 
\mathbb{A}_{p,q}^{\alpha ,m}\left( \sigma ,\omega \right) ,
\label{one par equ}
\end{equation}%
since the extension to two parameters will follow the same proof with
obvious modifications. Since $\mathbb{A}_{p,q}^{\alpha ,m}\left( \sigma
,\omega \right) \leq \overline{\mathbb{A}}_{p,q}^{\alpha ,m}\left( \sigma
,\omega \right) \leq \widehat{\mathbb{A}}_{p,q}^{\alpha ,m}\left( \sigma
,\omega \right) $, it suffices to prove 
\begin{equation}
\widehat{\mathbb{A}}_{p,q}^{\alpha ,m}\left( \sigma ,\omega \right) \lesssim 
\overline{\mathbb{A}}_{p,q}^{\alpha ,m}\left( \sigma ,\omega \right)
\lesssim \mathbb{A}_{p,q}^{\alpha ,m}\left( \sigma ,\omega \right) ,
\label{suff 2 prove}
\end{equation}%
and we begin with the second inequality $\overline{\mathbb{A}}_{p,q}^{\alpha
,m}\left( \sigma ,\omega \right) \lesssim \mathbb{A}_{p,q}^{\alpha ,m}\left(
\sigma ,\omega \right) $ in (\ref{suff 2 prove}).

So fix a cube $I$ which comes close to achieving the supremum in $\overline{%
\mathbb{A}}_{p,q}^{\alpha ,m}\left( \sigma ,\omega \right) $ over all cubes,
and suppose the tail $\widehat{s}_{I}$ occurs in the factor for $\omega $.
We have%
\begin{eqnarray*}
&&\sum_{k=0}^{\infty }2^{k\left[ \left( \alpha -m\right) q+m\right] }\frac{1%
}{\left\vert 2^{k}I\right\vert }\int_{2^{k}I}d\omega \left( x\right) \\
&=&\sum_{k=0}^{\infty }2^{k\left[ \left( \alpha -m\right) q+m\right] }\frac{1%
}{\left\vert 2^{k}I\right\vert }\sum_{\ell =0}^{k}\int_{2^{\ell }I\setminus
2^{\ell -1}I}d\omega \left( x\right) =\sum_{\ell =0}^{\infty }\left(
\sum_{k=\ell }^{\infty }2^{k\left( \alpha -m\right) q}\right) \int_{2^{\ell
}I\setminus 2^{\ell -1}I}d\omega \left( x\right) \\
&\approx &\sum_{\ell =0}^{\infty }2^{\ell \left( \alpha -m\right)
q}\int_{2^{\ell }I\setminus 2^{\ell -1}I}d\omega \left( x\right) \approx 
\frac{1}{\left\vert I\right\vert }\int_{\mathbb{R}^{m}}\left( 1+\frac{%
\left\vert x-c_{I}\right\vert }{\left\vert I\right\vert ^{\frac{1}{m}}}%
\right) ^{\left( \alpha -m\right) q}d\omega \left( x\right) =\frac{1}{%
\left\vert I\right\vert }\int_{\mathbb{R}^{m}}\widehat{s}_{I}^{q}d\omega .
\end{eqnarray*}%
since $\alpha <m$. Then we write%
\begin{eqnarray}
&&  \label{control two weights} \\
&&\overline{\mathbb{A}}_{p,q}^{\alpha ,m}\left( \sigma ,\omega \right)
^{q}\approx \left\{ \left\vert I\right\vert ^{\frac{\alpha }{m}-\Gamma
}\left( \frac{1}{\left\vert I\right\vert }\int_{\mathbb{R}^{m}}\widehat{s}%
_{I}^{q}d\omega \right) ^{\frac{1}{q}}\left( \frac{1}{\left\vert
I\right\vert }\int_{I}d\sigma \right) ^{\frac{1}{p^{\prime }}}\right\} ^{q} 
\notag \\
&\approx &\left\vert I\right\vert ^{\left( \frac{\alpha }{m}-\Gamma \right)
q}\sum_{k=0}^{\infty }2^{k\left[ \left( \alpha -m\right) q+m\right] }\left( 
\frac{1}{\left\vert 2^{k}I\right\vert }\int_{2^{k}I}d\omega \right) \left( 
\frac{1}{\left\vert I\right\vert }\int_{I}d\sigma \right) ^{\frac{q}{%
p^{\prime }}}  \notag \\
&\lesssim &\left\vert I\right\vert ^{\left( \frac{\alpha }{m}-\Gamma \right)
q}\sum_{k=0}^{\infty }2^{k\left[ \left( \alpha -m\right) q+m\right] }\left( 
\frac{1}{\left\vert 2^{k}I\right\vert }\int_{2^{k}I}d\omega \right) \left( 
\frac{1}{\left\vert I\right\vert }C2^{-k\delta }\left\vert 2^{k}I\right\vert
_{\sigma }\right) ^{\frac{q}{p^{\prime }}}  \notag \\
&=&\sum_{k=0}^{\infty }2^{-k\left( \alpha -m\Gamma \right) q}\left\vert
2^{k}I\right\vert ^{\left( \frac{\alpha }{m}-\Gamma \right) q}2^{k\left[
\left( \alpha -m\right) q+m\right] }\left( \frac{1}{\left\vert
2^{k}I\right\vert }\int_{2^{k}I}d\omega \right) 2^{km\frac{q}{p^{\prime }}%
}\left( \frac{1}{\left\vert 2^{k}I\right\vert }C2^{-k\delta }\left\vert
2^{k}I\right\vert _{\sigma }\right) ^{\frac{q}{p^{\prime }}}  \notag \\
&=&\sum_{k=0}^{\infty }C2^{-km\varepsilon }2^{-k\left( \alpha -m\Gamma
\right) q}2^{k\left[ \left( \alpha -m\right) q+m\right] }2^{km\frac{q}{%
p^{\prime }}}\left\{ \left\vert 2^{k}I\right\vert ^{\left( \frac{\alpha }{m}%
-\Gamma \right) q}\left( \frac{1}{\left\vert 2^{k}I\right\vert }%
\int_{2^{k}I}d\omega \right) \left( \frac{1}{\left\vert 2^{k}I\right\vert }%
\left\vert 2^{k}I\right\vert _{\sigma }\right) ^{\frac{q}{p^{\prime }}%
}\right\} ,  \notag
\end{eqnarray}%
where $\varepsilon >0$ is the reverse doubling exponent for the weight $%
\sigma $,%
\begin{equation*}
\left\vert Q\right\vert _{\sigma }\leq C2^{-km\varepsilon }\left\vert
2^{k}Q\right\vert _{\sigma }\text{ for all cubes }Q.
\end{equation*}%
We then dominate the term in braces by $\mathbb{A}_{p,q}^{\alpha ,m}\left(
\sigma ,\omega \right) $ to obtain that%
\begin{equation*}
\overline{\mathbb{A}}_{p,q}^{\alpha ,m}\left( \sigma ,\omega \right)
^{q}\lesssim \mathbb{A}_{p,q}^{\alpha ,m}\left( \sigma ,\omega \right)
^{q}\sum_{k=0}^{\infty }2^{-km\varepsilon }\lesssim \mathbb{A}_{p,q}^{\alpha
,m}\left( \sigma ,\omega \right) ^{q},
\end{equation*}%
since $-k\left( \alpha -m\Gamma \right) q+k\left[ \left( \alpha -m\right) q+m%
\right] +km\frac{q}{p^{\prime }}=0$, and this completes the proof of the
second inequality in (\ref{suff 2 prove}). Note that if the cube $I$ which
comes close to achieving the supremum in $\overline{\mathbb{A}}%
_{p,q}^{\alpha ,m}\left( \sigma ,\omega \right) $ over all cubes has the
tail $\widehat{s}_{I}$ occurrings in the factor for $\sigma $, then the
above argument applies using reverse doubling for the weight $\omega $.

Now we consider the first inequality in (\ref{suff 2 prove}). Here we can
use reverse doubling for either $\sigma $ or $\omega $, and we will use a
decomposition that uses reverse doubling for $\omega $ with exponent $%
\varepsilon >0$, i.e. $\left\vert Q\right\vert _{\omega }\leq
C2^{-km\varepsilon }\left\vert 2^{k}Q\right\vert _{\omega }$ for all cubes $%
Q $. This then has the analogous consequence for the $\omega $-integrals
with tails: 
\begin{eqnarray*}
\int_{\mathbb{R}^{m}}\widehat{s}_{I}^{q}d\omega &\approx &\sum_{k=0}^{\infty
}2^{k\left[ \left( \alpha -m\right) q\right] }\left( \int_{2^{k}I}d\omega
\right) \leq \sum_{k=0}^{\infty }2^{k\left[ \left( \alpha -m\right) q\right]
}\left( C2^{-\ell m\varepsilon }\int_{2^{k+\ell }I}d\omega \right) \\
&=&C2^{-\ell m\varepsilon }\sum_{k=0}^{\infty }2^{k\left[ \left( \alpha
-m\right) q\right] }\left( \int_{2^{k+\ell }I}d\omega \right) \approx
C2^{-\ell m\varepsilon }\int_{\mathbb{R}^{m}}\widehat{s}_{2^{\ell
}I}^{q}d\omega .
\end{eqnarray*}%
Now let $I$ be a cube which comes close to achieving the supremum in $%
\widehat{\mathbb{A}}_{p,q}^{\alpha ,m}\left( \sigma ,\omega \right) $ over
all cubes. Then we have%
\begin{eqnarray*}
&&\widehat{\mathbb{A}}_{p,q}^{\alpha ,m}\left( \sigma ,\omega \right)
^{p^{\prime }}\approx \left\{ \left\vert I\right\vert ^{\frac{\alpha }{m}%
-\Gamma }\left( \frac{1}{\left\vert I\right\vert }\int_{\mathbb{R}^{m}}%
\widehat{s}_{I}^{q}d\omega \right) ^{\frac{1}{q}}\left( \frac{1}{\left\vert
I\right\vert }\int_{\mathbb{R}^{m}}\widehat{s}_{I}^{p^{\prime }}d\sigma
\right) ^{\frac{1}{p^{\prime }}}\right\} ^{p^{\prime }} \\
&\approx &\left\vert I\right\vert ^{\left( \frac{\alpha }{m}-\Gamma \right)
p^{\prime }}\left( \frac{1}{\left\vert I\right\vert }\int_{\mathbb{R}^{m}}%
\widehat{s}_{I}^{q}d\omega \right) ^{\frac{p^{\prime }}{q}}\sum_{\ell
=0}^{\infty }2^{\ell \left[ \left( \alpha -m\right) p^{\prime }+m\right]
}\left( \frac{1}{\left\vert 2^{\ell }I\right\vert }\int_{2^{\ell }I}d\sigma
\right) \\
&\lesssim &\left\vert I\right\vert ^{\left( \frac{\alpha }{m}-\Gamma \right)
p^{\prime }}\sum_{\ell =0}^{\infty }C2^{-\ell m\varepsilon }2^{\ell \left[
\left( \alpha -m\right) p^{\prime }+m\right] }\left( \frac{1}{\left\vert
I\right\vert }\int_{\mathbb{R}^{m}}\widehat{s}_{2^{\ell }I}^{q}d\omega
\right) ^{\frac{p^{\prime }}{q}}\left( \frac{1}{\left\vert 2^{\ell
}I\right\vert }\int_{2^{\ell }I}d\sigma \right) \\
&=&\sum_{\ell =0}^{\infty }C2^{-\ell m\varepsilon }2^{-\ell \left( \alpha
-m\Gamma \right) p^{\prime }}2^{\ell \left[ \left( \alpha -m\right)
p^{\prime }+m\right] }2^{\ell m\frac{p^{\prime }}{q}}\left\{ \left\vert
2^{\ell }I\right\vert ^{\left( \frac{\alpha }{m}-\Gamma \right) p^{\prime
}}\left( \frac{1}{\left\vert 2^{\ell }I\right\vert }\int_{\mathbb{R}^{m}}%
\widehat{s}_{2^{\ell }I}^{q}d\omega \right) ^{\frac{p^{\prime }}{q}}\left( 
\frac{1}{\left\vert 2^{\ell }I\right\vert }\int_{2^{\ell }I}d\sigma \right)
\right\} .
\end{eqnarray*}%
We then dominate the term in braces by $\overline{\mathbb{A}}_{p,q}^{\alpha
,m}\left( \sigma ,\omega \right) ^{p^{\prime }}$ to obtain that%
\begin{equation*}
\widehat{\mathbb{A}}_{p,q}^{\alpha ,m}\left( \sigma ,\omega \right)
^{p^{\prime }}\lesssim \overline{\mathbb{A}}_{p,q}^{\alpha ,m}\left( \sigma
,\omega \right) ^{p^{\prime }}\sum_{\ell =0}^{\infty }2^{-\ell m\varepsilon
}\lesssim \overline{\mathbb{A}}_{p,q}^{\alpha ,m}\left( \sigma ,\omega
\right) ^{p^{\prime }},
\end{equation*}%
since $-\ell \left( \alpha -m\Gamma \right) p^{\prime }+\ell \left[ \left(
\alpha -m\right) p^{\prime }+m\right] +\ell m\frac{p^{\prime }}{q}=0$, and
this completes the proof of the first inequality in (\ref{suff 2 prove}),
which in turn completes the proof of the one parameter equivalence (\ref{one
par equ}).

In order to prove the two parameter equivalence in Lemma \ref{reverse}, we
begin with%
\begin{equation*}
\sum_{k,j=0}^{\infty }2^{k\left[ \left( \alpha -m\right) q+m\right] +j\left[
\left( \beta -n\right) q+n\right] }\frac{1}{\left\vert 2^{k}I\times
2^{j}J\right\vert }\int_{2^{k}I\times 2^{j}J}d\omega \left( x\right) \approx 
\frac{1}{\left\vert I\times J\right\vert }\int_{\mathbb{R}^{m}\times \mathbb{%
R}^{n}}\widehat{s}_{I,J}^{q}d\omega ,
\end{equation*}%
and proceed with the two parameter analogue of (\ref{control two weights}),
followed by the straightforward modifications of the remaining arguments.
This completes our proof of Lemma \ref{reverse}.
\end{proof}

\subsection{Proof of Theorem \protect\ref{A char mn}}

Let $I=\dprod\limits_{i=1}^{m}\left[ a_{i},a_{i}+s\right] $ and $%
J=\dprod\limits_{j=1}^{n}\left[ b_{j},b_{j}+t\right] $ be cubes in $\mathbb{R%
}_{+}^{m}$ and $\mathbb{R}_{+}^{n}$ with side lengths $s>0$ and $t>0$
respectively. Then the local characteristic $A_{p,q}^{\left( \alpha ,\beta
\right) ,\left( m,n\right) }\left( v,w\right) \left[ I,J\right] $ is given by%
\begin{eqnarray*}
&&A_{p,q}^{\left( \alpha ,\beta \right) ,\left( m,n\right) }\left(
v,w\right) \left[ I,J\right] \\
&=&s^{\alpha -m}\ t^{\beta -n}\left( \int_{\left[ a,a+s\right] ^{m}}\int_{%
\left[ b,b+t\right] ^{n}}\left\vert \left( x,y\right) \right\vert ^{-\gamma
q}dxdy\right) ^{\frac{1}{q}}\left( \int_{\left[ a,a+s\right] ^{m}}\int_{%
\left[ b,b+t\right] ^{n}}\left\vert \left( x,y\right) \right\vert ^{-\delta
p^{\prime }}dxdy\right) ^{\frac{1}{p^{\prime }}} \\
&=&s^{\alpha -m}\ t^{\beta -n}\ \left[ \mathcal{I}_{\left\{ a,b\right\}
;\left( s,t\right) }^{m,n}\left( -\gamma q\right) \right] ^{\frac{1}{q}}\ %
\left[ \mathcal{I}_{\left\{ a,b\right\} ;\left( s,t\right) }^{m,n}\left(
-\delta p^{\prime }\right) \right] ^{\frac{1}{p^{\prime }}},
\end{eqnarray*}%
where%
\begin{eqnarray*}
\mathcal{I}_{\left( a,b\right) ;\left( s,t\right) }^{m,n}\left( \eta \right)
&\equiv &\int_{\dprod\limits_{i=1}^{m}\left[ a_{i},a_{i}+s\right]
}\int_{\dprod\limits_{j=1}^{n}\left[ b_{j},b_{j}+t\right] }\left\vert \left(
x,y\right) \right\vert ^{\eta }dxdy=\int_{\dprod\limits_{i=1}^{m}\left[
a_{i},a_{i}+s\right] }\int_{\dprod\limits_{j=1}^{n}\left[ b_{j},b_{j}+t%
\right] }\left( \left\vert x\right\vert ^{2}+\left\vert y\right\vert
^{2}\right) ^{\frac{\eta }{2}}dxdy \\
&\approx &\int_{\dprod\limits_{i=1}^{m}\left[ a_{i},a_{i}+s\right]
}\int_{\dprod\limits_{j=1}^{n}\left[ b_{j},b_{j}+t\right] }\left(
x_{1}+...+x_{m}+y_{1}+...+y_{n}\right) ^{\eta }dxdy\equiv \mathsf{I}_{\left(
a,b\right) ;\left( s,t\right) }^{m,n}\left( \eta \right) ,
\end{eqnarray*}%
for $\left( a,b\right) \in \mathbb{R}_{+}^{m}\times \mathbb{R}_{+}^{n}$. Two
immediate necessary conditions for the finiteness of $A_{p,q}^{\left( \alpha
,\beta \right) ,\left( m,n\right) }\left( v,w\right) $ are the local
integrability (\ref{local integ}) of the weights $w^{q}$ and $v^{-p^{\prime
}}$, namely $\gamma q<m+n$ and $\delta p^{\prime }<m+n$. Define the sans
serif local characteristic $\mathsf{A}_{p,q}^{\left( \alpha ,\beta \right)
,\left( m,n\right) }\left( v,w\right) \left[ I,J\right] $ by%
\begin{eqnarray}
\mathsf{A}_{p,q}^{\left( \alpha ,\beta \right) ,\left( m,n\right) }\left(
v,w\right) \left[ I,J\right] &\equiv &s^{\alpha -m}\ t^{\beta -n}\ \left[ 
\mathsf{I}_{\left\{ a,b\right\} ;\left( s,t\right) }^{m,n}\left( -\gamma
q\right) \right] ^{\frac{1}{q}}\ \left[ \mathsf{I}_{\left\{ a,b\right\}
;\left( s,t\right) }^{m,n}\left( -\delta p^{\prime }\right) \right] ^{\frac{1%
}{p^{\prime }}},  \label{def A tilda} \\
\text{for}\ \ \ \ \ \text{ }\left( a,b\right) &\in &\mathbb{R}_{+}^{m}\times 
\mathbb{R}_{+}^{n}\ .  \notag
\end{eqnarray}%
Symmetry considerations show that the characteristic $A_{p,q}^{\left( \alpha
,\beta \right) ,\left( m,n\right) }\left( v,w\right) $ satisfies%
\begin{equation*}
A_{p,q}^{\left( \alpha ,\beta \right) ,\left( m,n\right) }\left( v,w\right)
\equiv \sup_{I\times J\subset \mathbb{R}^{m}\times \mathbb{R}%
^{n}}A_{p,q}^{\left( \alpha ,\beta \right) ,\left( m,n\right) }\left(
v,w\right) \left[ I,J\right] \approx \sup_{I\times J\subset \mathbb{R}%
_{+}^{m}\times \mathbb{R}_{+}^{n}}\mathsf{A}_{p,q}^{\left( \alpha ,\beta
\right) ,\left( m,n\right) }\left( v,w\right) \left[ I,J\right] .
\end{equation*}

Suppose that $I=\dprod\limits_{i=1}^{m}\left[ a_{i},a_{i}+s\right] \subset 
\mathbb{R}_{+}^{m}$ and $J=\dprod\limits_{j=1}^{n}\left[ b_{j},b_{j}+t\right]
\subset \mathbb{R}_{+}^{n}$ are as above. If we slide (translate) the
rectangle $I\times J$ to a new position 
\begin{equation*}
I^{\prime }\times J^{\prime }=\dprod\limits_{i=1}^{m}\left[ a_{i}^{\prime
},a_{i}^{\prime }+s\right] \times \dprod\limits_{j=1}^{n}\left[
b_{j}^{\prime },b_{j}^{\prime }+t\right] \subset \mathbb{R}_{+}^{m}\times 
\mathbb{R}_{+}^{n}
\end{equation*}%
in which their centers 
\begin{equation*}
c_{I\times J}=\left( a,b\right) ,c_{I^{\prime }\times J^{\prime }}=\left(
a^{\prime },b^{\prime }\right) \in \mathbb{R}_{+}^{m}\times \mathbb{R}%
_{+}^{n}
\end{equation*}%
lie on the same simplex 
\begin{equation*}
S\left( \nu \right) \equiv \left\{ \left( x,y\right) \in \mathbb{R}%
_{+}^{m}\times \mathbb{R}_{+}^{n}:x_{1}+...+x_{m}+y_{1}+...+y_{n}=\nu
\right\} ,\ \ \ \ \ \nu >0,
\end{equation*}%
then the sans serif local characteristics are unchanged, i.e. 
\begin{equation*}
\mathsf{A}_{p,q}^{\left( \alpha ,\beta \right) ,\left( m,n\right) }\left(
v,w\right) \left[ I,J\right] =\mathsf{A}_{p,q}^{\left( \alpha ,\beta \right)
,\left( m,n\right) }\left( v,w\right) \left[ I^{\prime },J^{\prime }\right] .
\end{equation*}%
Thus it suffices to consider only rectangles $I\times J\subset \mathbb{R}%
_{+}^{m}\times \mathbb{R}_{+}^{n}$ having the special forms%
\begin{equation}
I=\left[ 0,s\right] ^{k-1}\times \left[ a_{k},a_{k}+s\right] \times \left[
0,s\right] ^{m-k}\text{ and }J=\left[ 0,t\right] ^{n},  \label{form a}
\end{equation}%
and%
\begin{equation}
I=\left[ 0,s\right] ^{m}\text{ and }J=\left[ 0,t\right] ^{\ell -1}\times %
\left[ b_{\ell },b_{\ell }+t\right] \times \left[ 0,t\right] ^{n-\ell },
\label{form b}
\end{equation}%
for some $1\leq k\leq m$ and $1\leq \ell \leq n$. We have thus shown that
the characteristic $A_{p,q}^{\left( \alpha ,\beta \right) ,\left( m,n\right)
}\left( v,w\right) $ is controlled by the supremum of the sans serif local
characteristics $\mathsf{A}_{p,q}^{\left( \alpha ,\beta \right) ,\left(
m,n\right) }\left( v,w\right) \left[ I,J\right] $ taken over rectangles $%
I\times J$ of the special forms given in (\ref{form a}) and (\ref{form b}):%
\begin{equation*}
A_{p,q}^{\left( \alpha ,\beta \right) ,\left( m,n\right) }\left( v,w\right)
\approx \sup_{I\times J\text{ as in (\ref{form a}) or (\ref{form b})}}%
\mathsf{A}_{p,q}^{\left( \alpha ,\beta \right) ,\left( m,n\right) }\left(
v,w\right) \left[ I,J\right] .
\end{equation*}

Let us further consider a rectangle $I\times J$ of the form (\ref{form a}),
and without loss of generality, we may take $k=1$ so that%
\begin{equation*}
I=\left[ a,a+s\right] \times \left[ 0,s\right] ^{m-1}\text{ and }J=\left[ 0,t%
\right] ^{n}.
\end{equation*}%
Recall that the corresponding sans serif local characteristic is given by%
\begin{eqnarray*}
&&\mathsf{A}_{p,q}^{\left( \alpha ,\beta \right) ,\left( m,n\right) }\left(
v,w\right) \left[ I,J\right] \\
&=&s^{\alpha -m}\ t^{\beta -n}\ \left[ \mathsf{I}_{\left\{ \left(
a,0,...,0\right) ,\mathbf{0}\right\} ;\left( s,t\right) }^{m,n}\left(
-\gamma q\right) \right] ^{\frac{1}{q}}\ \left[ \mathsf{I}_{\left\{ \left(
a,0,...,0\right) ,\mathbf{0}\right\} ;\left( s,t\right) }^{m,n}\left(
-\delta p^{\prime }\right) \right] ^{\frac{1}{p^{\prime }}},
\end{eqnarray*}%
where $\mathbf{0}=\left( 0,...,0\right) \in \mathbb{R}^{n}$ and%
\begin{equation*}
\mathsf{I}_{\left\{ \left( a,0,...,0,\right) ,\mathbf{0}\right\} ;\left(
s,t\right) }^{m,n}\left( \eta \right) =\int_{x\in \left[ a,a+s\right] \times %
\left[ 0,s\right] ^{m-1}}\left\{ \int_{y\in \left[ 0,t\right] ^{n}}\left(
x_{1}+...+x_{m}+y_{1}+...+y_{n}\right) ^{\eta }dy\right\} dx.
\end{equation*}

\begin{enumerate}
\item First we note that if $s\geq a>0$, then 
\begin{equation*}
\mathsf{I}_{\left\{ \left( a,0,...,0,\right) ,\mathbf{0}\right\} ;\left(
s,t\right) }^{m,n}\left( \eta \right) \approx \mathsf{I}_{\left\{ \mathbf{0},%
\mathbf{0}\right\} ;s,t}^{m,n}\left( \eta \right) ,\ \ \ \ \ \eta \geq 0,
\end{equation*}%
and 
\begin{equation*}
\mathsf{I}_{\left\{ \left( a,0,...,0,\right) ,\mathbf{0}\right\} ;\left(
s,t\right) }^{m,n}\left( \eta \right) \lesssim \mathsf{I}_{\left\{ \mathbf{0}%
,\mathbf{0}\right\} ;s,t}^{m,n}\left( \eta \right) ,\ \ \ \ \ -\left(
m+n\right) <\eta <0,
\end{equation*}%
as is easily seen using the doubling and monotonicity properties of the
locally integrable weight $\left( x_{1}+...+x_{m}+y_{1}+...+y_{n}\right)
^{\eta }$. We conclude using (\ref{local integ}) that%
\begin{eqnarray}
&&s^{\alpha -m}\ t^{\beta -n}\ \left[ \mathsf{I}_{\left\{ \left(
a,0,...,0\right) ,\mathbf{0}\right\} ;\left( s,t\right) }^{m,n}\left(
-\gamma q\right) \right] ^{\frac{1}{q}}\ \left[ \mathsf{I}_{\left\{ \left(
a,0,...,0\right) ,\mathbf{0}\right\} ;\left( s,t\right) }^{m,n}\left(
-\delta p^{\prime }\right) \right] ^{\frac{1}{p^{\prime }}}  \label{con that}
\\
&&\ \ \ \ \ \lesssim s^{\alpha -m}\ t^{\beta -n}\ \left[ \mathsf{I}_{\left\{ 
\mathbf{0},\mathbf{0}\right\} ;\left( s,t\right) }^{m,n}\left( -\gamma
q\right) \right] ^{\frac{1}{q}}\ \left[ \mathsf{I}_{\left\{ \mathbf{0},%
\mathbf{0}\right\} ;\left( s,t\right) }^{m,n}\left( -\delta p^{\prime
}\right) \right] ^{\frac{1}{p^{\prime }}},  \notag
\end{eqnarray}%
holds when $s\geq a>0$.

\item Now consider the case $0<s<a\leq t$. Then by doubling and monotonicity
we have 
\begin{equation*}
\mathsf{I}_{\left\{ \left( a,0,...,0,\right) ,\mathbf{0}\right\} ;\left(
s,t\right) }^{m,n}\left( \eta \right) \lesssim \mathsf{I}_{\left\{ \mathbf{0}%
,\mathbf{0}\right\} ;s,t}^{m,n}\left( \eta \right) ,\ \ \ \ \ -\left(
m+n\right) <\eta <\infty ,
\end{equation*}%
and so we conclude again using (\ref{local integ}) that (\ref{con that})
also holds when $0<s<a<t$.

\item Now consider the remaining case $0<s,t\leq a$. Since the weight $%
\left( x_{1}+...+x_{m}+y_{1}+...+y_{n}\right) ^{\eta }$ is roughly equal to
the constant $a^{\eta }$ on the cube $\left[ a,2a\right] \times \left[ 0,a%
\right] ^{m-1}\times \left[ 0,a\right] ^{n}$, we have%
\begin{eqnarray*}
&&\mathsf{A}_{p,q}^{\left( \alpha ,\beta \right) ,\left( m,n\right) }\left(
v,w\right) \left[ \left[ a,a+s\right] \times \left[ 0,s\right] ^{m-1},\left[
0,t\right] ^{n}\right] \\
&=&s^{\alpha -m1}\ t^{\beta -n}\ \left[ \mathsf{I}_{\left\{ \left(
a,0,...,0,\right) ,\mathbf{0}\right\} ;\left( s,t\right) }^{m,n}\left(
-\gamma q\right) \right] ^{\frac{1}{q}}\ \left[ \mathsf{I}_{\left\{ \left(
a,0,...,0,\right) ,\mathbf{0}\right\} ;\left( s,t\right) }^{m,n}\left(
-\delta p^{\prime }\right) \right] ^{\frac{1}{p^{\prime }}} \\
&=&s^{\alpha -m\Gamma }\ t^{\beta -n\Gamma }\ \left[ \frac{1}{s^{m}t^{n}}%
\mathsf{I}_{\left\{ \left( a,0,...,0,\right) ,\mathbf{0}\right\} ;\left(
s,t\right) }^{m,n}\left( -\gamma q\right) \right] ^{\frac{1}{q}}\ \left[ 
\frac{1}{s^{m}t^{n}}\mathsf{I}_{\left\{ \left( a,0,...,0,\right) ,\mathbf{0}%
\right\} ;\left( s,t\right) }^{m,n}\left( -\delta p^{\prime }\right) \right]
^{\frac{1}{p^{\prime }}} \\
&\approx &s^{\alpha -m\Gamma }\ t^{\beta -n\Gamma }\ a^{-\left( \gamma
+\delta \right) }\ ,
\end{eqnarray*}%
for $0<s,t<a$. We note for future reference that%
\begin{eqnarray*}
&&\mathsf{A}_{p,q}^{\left( \alpha ,\beta \right) ,\left( m,n\right) }\left(
v,w\right) \left[ \left[ a,a+s\right] \times \left[ 0,s\right] ^{m-1},\left[
0,t\right] ^{n}\right] \\
&\approx &\mathsf{A}_{p,q}^{\left( \alpha ,\beta \right) ,\left( m,n\right)
}\left( v,w\right) \left[ \left[ a,a+s\right] ^{m},\left[ a,a+t\right] ^{n}%
\right] .
\end{eqnarray*}
\end{enumerate}

Choosing $s=t=a$ in (3) we see that $a^{\left( \alpha +\beta \right) -\left(
\gamma +\delta \right) -\left( m+n\right) \Gamma }$ must be bounded for all $%
a>0$, which is equivalent to the balanced condition%
\begin{equation}
\Gamma =\frac{\left( \alpha +\beta \right) -\left( \gamma +\delta \right) }{%
m+n}.  \label{bal mn}
\end{equation}%
Now we also have that both $s^{\alpha -m\Gamma }\ $and $t^{\beta -n\Gamma }$
are bounded for $0<s,t\leq a$, which is equivalent to 
\begin{equation*}
\frac{\alpha }{m},\frac{\beta }{n}\geq \Gamma .
\end{equation*}%
If we fix $t=1$ and let $s=a\rightarrow \infty $, we obtain that the
boundedness of $a^{\alpha -m\Gamma }\ a^{-\left( \gamma +\delta \right) }$
for $a$ large implies%
\begin{equation*}
\frac{\alpha }{m}\leq \Gamma +\frac{\gamma +\delta }{m}.
\end{equation*}%
Similarly we have%
\begin{equation*}
\frac{\beta }{n}\leq \Gamma +\frac{\gamma +\delta }{n},
\end{equation*}%
from which we conclude the diagonal inequality%
\begin{equation}
\left\{ 
\begin{array}{c}
\Gamma \leq \frac{\alpha }{m}\leq \Gamma +\frac{\gamma +\delta }{m} \\ 
\Gamma \leq \frac{\beta }{n}\leq \Gamma +\frac{\gamma +\delta }{n}%
\end{array}%
\right. .  \label{diagonal}
\end{equation}%
It is now an easy matter to verify that the supremum taken over $0<s,t\leq a$
of the sans serif characteristics arising in (3), is finite if and only if
both the power weight equality (\ref{bal mn}) and diagonal inequalities (\ref%
{diagonal}) hold, i.e.%
\begin{eqnarray*}
&&\sup_{0<s,t\leq a}\mathsf{A}_{p,q}^{\left( \alpha ,\beta \right) ,\left(
m,n\right) }\left( v,w\right) \left[ \left[ a,a+s\right] \times \left[ 0,s%
\right] ^{m-1},\left[ 0,t\right] ^{n}\right] <\infty \\
&\Longleftrightarrow &\sup_{0<s,t<a}s^{\alpha -m\Gamma }\ t^{\beta -m\Gamma
}\ a^{-\left( \gamma +\delta \right) }<\infty \\
&\Longleftrightarrow &\left\{ 
\begin{array}{c}
\Gamma =\frac{\left( \alpha +\beta \right) -\left( \gamma +\delta \right) }{%
m+n} \\ 
\Gamma \leq \frac{\alpha }{m}\leq \Gamma +\frac{\gamma +\delta }{m} \\ 
\Gamma \leq \frac{\beta }{n}\leq \Gamma +\frac{\gamma +\delta }{n}%
\end{array}%
\right. .
\end{eqnarray*}

If we combine this with the conclusions of (1) and (2), along with a similar
analysis of the case when $I\times J$ has the form in (\ref{form b}), we
obtain the following proposition.

\begin{proposition}
Suppose that $\alpha ,\beta >0$, $1<p,q<\infty $, and $w\left( x,y\right)
=\left\vert \left( x,y\right) \right\vert ^{-\gamma }$ and $v\left(
x,y\right) =\left\vert \left( x,y\right) \right\vert ^{\delta }$ with $%
-\infty <\gamma ,\delta <\infty $ that satisfy (\ref{local integ}). Then 
\begin{eqnarray*}
&&A_{p,q}^{\left( \alpha ,\beta \right) ,\left( m,n\right) }\left( v,w\right)
\\
&\approx &\sup_{0<s,t<\infty }\mathsf{A}_{p,q}^{\left( \alpha ,\beta \right)
,\left( m,n\right) }\left( v,w\right) \left[ \left[ 0,s\right] ^{m},\left[
0,t\right] ^{n}\right] \\
&&+\sup_{0<s,t\leq a<\infty }\mathsf{A}_{p,q}^{\left( \alpha ,\beta \right)
,\left( m,n\right) }\left( v,w\right) \left[ \left[ a,a+s\right] ^{m},\left[
a,a+t\right] ^{n}\right] ,
\end{eqnarray*}%
where the second summand is finite if and only if $\left\{ 
\begin{array}{c}
\Gamma =\frac{\left( \alpha +\beta \right) -\left( \gamma +\delta \right) }{%
m+n} \\ 
\Gamma \leq \frac{\alpha }{m}\leq \Gamma +\frac{\gamma +\delta }{m} \\ 
\Gamma \leq \frac{\beta }{n}\leq \Gamma +\frac{\gamma +\delta }{n}%
\end{array}%
\right. $.
\end{proposition}

\subsubsection{Rectangles at the origin}

It now remains to determine under what conditions on the indices the first
summand is finite, i.e. when%
\begin{equation*}
\sup_{0<s,t<\infty }\mathsf{A}_{p,q}^{\left( \alpha ,\beta \right) ,\left(
m,n\right) }\left( v,w\right) \left[ \left[ 0,s\right] ^{m},\left[ 0,t\right]
^{n}\right] <\infty .
\end{equation*}%
For this we start with%
\begin{eqnarray*}
&&\mathsf{A}_{p,q}^{\left( \alpha ,\beta \right) ,\left( m,n\right) }\left(
v,w\right) \left[ \left[ 0,s\right] ^{m},\left[ 0,t\right] ^{n}\right] \\
&=&s^{\alpha -m}\ t^{\beta -n}\ \left[ \mathsf{I}_{\left\{ \mathbf{0},%
\mathbf{0}\right\} ;\left( s,t\right) }^{m,n}\left( -\gamma q\right) \right]
^{\frac{1}{q}}\ \left[ \mathsf{I}_{\left\{ \mathbf{0},\mathbf{0}\right\}
;\left( s,t\right) }^{m,n}\left( -\delta p^{\prime }\right) \right] ^{\frac{1%
}{p^{\prime }}}
\end{eqnarray*}%
where if $0<s\leq t<\infty $,%
\begin{eqnarray*}
\mathsf{I}_{\left\{ \mathbf{0},\mathbf{0}\right\} ;\left( s,t\right)
}^{m,n}\left( \eta \right) &=&\int_{x\in \left[ 0,s\right] ^{m}}\left\{
\int_{y\in \left[ 0,t\right] ^{n}}\left( \sigma _{1}+...+\sigma _{m}+\tau
_{1}+...+\tau _{n}\right) ^{\eta }d\tau _{1}...d\tau _{n}\right\} d\sigma
_{1}...d\sigma _{m} \\
&\approx &\int_{0}^{s}\left\{ \int_{0}^{t}\left( \sigma +\tau \right) ^{\eta
}\ \tau ^{n-1}d\tau \right\} \sigma ^{m-1}d\sigma .
\end{eqnarray*}

We note that%
\begin{equation*}
\int_{0}^{\lambda }\left( \rho +1\right) ^{\eta }\rho ^{m-1}d\rho \approx
\left\{ 
\begin{array}{ccc}
\lambda ^{m} & \text{ if } & 0<\lambda \leq 1 \\ 
\lambda ^{\left( m+\eta \right) _{+}} & \text{ if } & 1\leq \lambda <\infty 
\text{ and }m+\eta \neq 0 \\ 
1+\ln \lambda & \text{ if } & 1\leq \lambda <\infty \text{ and }m+\eta =0%
\end{array}%
\right. .
\end{equation*}%
Now suppose that $0<t\leq s<\infty $. Then we have with $x=x_{+}-x_{-}$ that 
\begin{eqnarray*}
\mathcal{I}_{s,t}^{m,n}\left( \eta \right) &\approx &\int_{0}^{t}\left\{
\int_{0}^{s}\left( \sigma +\tau \right) ^{\eta }\sigma ^{m-1}d\sigma
\right\} \tau ^{n-1}d\tau \\
&=&\int_{0}^{t}\left\{ \int_{\sigma =0}^{s}\left( \frac{\sigma }{\tau }%
+1\right) ^{\eta }\left( \frac{\sigma }{\tau }\right) ^{m-1}d\left( \frac{%
\sigma }{\tau }\right) \right\} \tau ^{\eta +m+n-1}d\tau \\
&=&\int_{0}^{t}\left\{ \int_{\sigma ^{\prime }=0}^{\frac{s}{\tau }}\left(
\sigma ^{\prime }+1\right) ^{\eta }\left( \sigma ^{\prime }\right)
^{m-1}d\sigma ^{\prime }\right\} \tau ^{\eta +m+n-1}d\tau \\
&\approx &\int_{0}^{t}\left( \frac{s}{\tau }\right) ^{\left( m+\eta \right)
_{+}}\tau ^{\eta +m+n-1}d\tau =s^{\left( m+\eta \right) _{+}}\ t^{n-\left(
m+\eta \right) _{-}}\ ,
\end{eqnarray*}%
provided $m+\eta \neq 0$ and $m+\eta >-n$, and%
\begin{eqnarray*}
\mathcal{I}_{s,t}^{m,n}\left( -m\right) &\approx &\int_{0}^{t}\left\{
\int_{\sigma ^{\prime }=0}^{\frac{s}{\tau }}\left( \sigma ^{\prime
}+1\right) ^{-m}\left( \sigma ^{\prime }\right) ^{m-1}d\sigma ^{\prime
}\right\} \tau ^{n-1}d\tau \approx \int_{0}^{t}\left( \ln \frac{s}{\tau }%
\right) \tau ^{n-1}d\tau \\
&=&\frac{1}{n}t^{n}\ln s-\int_{0}^{t}\left( \ln \tau \right) \tau
^{n-1}d\tau =\frac{1}{n}t^{n}\ln \frac{s}{t}+\frac{1}{n^{2}}t^{n}=\frac{t^{n}%
}{n}\left( \frac{1}{n}+\ln \frac{s}{t}\right) ,
\end{eqnarray*}%
since $\int x^{n-1}\ln x\,dx=\frac{1}{n}x^{n}\ln x-\frac{1}{n^{2}}x^{n}$. It
will be convenient to simply understand that 
\begin{equation*}
s^{0_{+}}\ t^{-0_{-}}=\ln \frac{s}{t}\text{ when }0<t\leq s<\infty .
\end{equation*}

Similarly, if $0<s\leq t<\infty $, then we have $\mathcal{I}%
_{s,t}^{m,n}\left( \eta \right) \approx s^{m-\left( n+\eta \right) _{-}}\
t^{\left( n+\eta \right) _{+}}$ provided $n+\eta >-m$, with the analogous
understanding that $s^{-0_{-}}\ t^{0_{+}}=\ln \frac{t}{s}$, and so
altogether we obtain%
\begin{equation*}
\mathcal{I}_{s,t}^{m,n}\left( \eta \right) \approx \left\{ 
\begin{array}{ccc}
s^{m-\left( n+\eta \right) _{-}}\ t^{\left( n+\eta \right) _{+}} & \text{ if 
} & 0<s\leq t<\infty \\ 
s^{\left( m+\eta \right) _{+}}\ t^{n-\left( m+\eta \right) _{-}} & \text{ if 
} & 0<t\leq s<\infty%
\end{array}%
\right. ,\ \ \ \ \ m+n+\eta >0.
\end{equation*}%
Thus if $0<s\leq t<\infty $ we have, using $m+n-\gamma q>0$ and $m+n-\delta
p^{\prime }>0$, that 
\begin{eqnarray*}
&&A_{p,q}^{\left( \alpha ,\beta \right) ,\left( m,n\right) }\left(
v,w\right) \left[ I,J\right] =s^{\alpha -m}\ t^{\beta -n}\ \mathcal{I}%
_{s,t}^{m,n}\left( -\gamma q\right) ^{\frac{1}{q}}\ \mathcal{I}%
_{s,t}^{m,n}\left( -\delta p^{\prime }\right) ^{\frac{1}{p^{\prime }}} \\
&=&s^{\alpha -m}\ t^{\beta -n}\left( s^{m-\left( n-\gamma q\right) _{-}}\
t^{\left( n-\gamma q\right) _{+}}\right) ^{\frac{1}{q}}\left( s^{m-\left(
n-\delta p^{\prime }\right) _{-}}\ t^{\left( n-\delta p^{\prime }\right)
_{+}}\right) ^{\frac{1}{p^{\prime }}} \\
&=&s^{\alpha -m+\frac{m-\left( n-\gamma q\right) _{-}}{q}+\frac{m-\left(
n-\delta p^{\prime }\right) _{-}}{p^{\prime }}}\ t^{\beta -n+\frac{\left(
n-\gamma q\right) _{+}}{q}+\frac{\left( n-\delta p^{\prime }\right) _{+}}{%
p^{\prime }}} \\
&=&s^{\alpha -m+\frac{m}{q}-\left( \frac{n}{q}-\gamma \right) _{-}+\frac{m}{%
p^{\prime }}-\left( \frac{n}{p^{\prime }}-\delta \right) _{-}}\ t^{\beta
-n+\left( \frac{n}{q}-\gamma \right) _{+}+\left( \frac{n}{p^{\prime }}%
-\delta \right) _{+}}\ ,
\end{eqnarray*}%
and similarly if $0<t\leq s<\infty $, we have%
\begin{eqnarray*}
&&A_{p,q}^{\left( \alpha ,\beta \right) ,\left( m,n\right) }\left(
v,w\right) \left[ I,J\right] =s^{\alpha -m}\ t^{\beta -n}\ \mathcal{I}%
_{s,t}^{m,n}\left( -\gamma q\right) ^{\frac{1}{q}}\ \mathcal{I}%
_{s,t}^{m,n}\left( -\delta p^{\prime }\right) ^{\frac{1}{p^{\prime }}} \\
&=&s^{\alpha -m}\ t^{\beta -n}\ \left( s^{\left( m-\gamma q\right) _{+}}\
t^{n-\left( m-\gamma q\right) _{-}}\right) ^{\frac{1}{q}}\ \left( s^{\left(
m-\delta p^{\prime }\right) _{+}}\ t^{n-\left( m-\delta p^{\prime }\right)
_{-}}\right) ^{\frac{1}{p^{\prime }}} \\
&=&s^{\alpha -m+\frac{\left( m-\gamma q\right) _{+}}{q}+\frac{\left(
m-\delta p^{\prime }\right) _{+}}{p^{\prime }}}\ t^{\beta -n+\frac{n-\left(
m-\gamma q\right) _{-}}{q}+\frac{n-\left( m-\delta p^{\prime }\right) _{-}}{%
p^{\prime }}n} \\
&=&s^{\alpha -m+\left( \frac{m}{q}-\gamma \right) _{+}+\left( \frac{m}{%
p^{\prime }}-\delta \right) _{+}}\ t^{\beta -n+\frac{n}{q}-\left( \frac{m}{q}%
-\gamma \right) _{-}+\frac{n}{p^{\prime }}-\left( \frac{m}{p^{\prime }}%
-\delta \right) _{-}}\ ,
\end{eqnarray*}%
with the understanding that%
\begin{eqnarray}
s^{0_{+}}\ t^{-0_{-}} &=&\ln \frac{s}{t}\text{ when }0<t\leq s<\infty ,
\label{understanding} \\
s^{-0_{-}}\ t^{0_{+}} &=&\ln \frac{t}{s}\text{ when }0<s\leq t<\infty . 
\notag
\end{eqnarray}%
In order to efficiently calculate the remaining conditions, we assume for
the moment that $m\leq n$, and remove this restriction at the end.

Now we consider separately the (at most) five cases determined by $\delta $
(since the open interval $\left( \frac{m}{q},\frac{n}{q}\right) $ is empty
if $m=n$):%
\begin{equation*}
\delta \in \left( -\infty ,\frac{m}{p^{\prime }}\right) ,\ \ \ \delta =\frac{%
m}{p^{\prime }},\ \ \ \delta \in \left( \frac{m}{p^{\prime }},\frac{n}{%
p^{\prime }}\right) ,\ \ \ \delta =\frac{n}{p^{\prime }},\ \ \ \delta \in
\left( \frac{n}{p^{\prime }},\infty \right) .
\end{equation*}%
\textbf{Case A}: $\delta <\frac{m}{p^{\prime }}$. Then for $0<s\leq t<\infty 
$, we have%
\begin{equation*}
A_{p,q}^{\left( \alpha ,\beta \right) ,\left( m,n\right) }\left( v,w\right) 
\left[ I,J\right] =s^{\alpha -m+\frac{m}{q}-\left( \frac{n}{q}-\gamma
\right) _{-}+\frac{m}{p^{\prime }}}\ t^{\beta -n+\left( \frac{n}{q}-\gamma
\right) _{+}+\frac{n}{p^{\prime }}-\delta }\ ,
\end{equation*}%
and for $0<t\leq s<\infty $, we have%
\begin{equation*}
A_{p,q}^{\left( \alpha ,\beta \right) ,\left( m,n\right) }\left( v,w\right) 
\left[ I,J\right] =s^{\alpha -m+\left( \frac{m}{q}-\gamma \right) _{+}+\frac{%
m}{p^{\prime }}-\delta }\ t^{\beta -n+\frac{n}{q}-\left( \frac{m}{q}-\gamma
\right) _{-}+\frac{n}{p^{\prime }}}\ .
\end{equation*}%
Now we consider separately within Case A the (at most) five subcases
determined by $\gamma $ (since the open interval $\left( \frac{m}{q},\frac{n%
}{q}\right) $ is empty if $m=n$):%
\begin{equation*}
\gamma \in \left( -\infty ,\frac{m}{q}\right) ,\ \ \ \gamma =\frac{m}{q},\ \
\ \gamma \in \left( \frac{m}{q},\frac{n}{q}\right) ,\ \ \ \gamma =\frac{n}{q}%
,\ \ \ \gamma \in \left( \frac{n}{q},\infty \right) .
\end{equation*}

\textbf{Subcase A1:} $\gamma \in \left( -\infty ,\frac{m}{q}\right) $. In
this subcase $\frac{m}{q}-\gamma >0$ and $\frac{n}{q}-\gamma >0$, and so for 
$0<s\leq t<\infty $, we have%
\begin{equation*}
A_{p,q}^{\left( \alpha ,\beta \right) ,\left( m,n\right) }\left( v,w\right) 
\left[ I,J\right] =s^{\alpha -m+\frac{m}{q}+\frac{m}{p^{\prime }}}\ t^{\beta
-n+\frac{n}{q}-\gamma +\frac{n}{p^{\prime }}-\delta }\ ,
\end{equation*}%
and for $0<t\leq s<\infty $, we have%
\begin{equation*}
A_{p,q}^{\left( \alpha ,\beta \right) ,\left( m,n\right) }\left( v,w\right) 
\left[ I,J\right] =s^{\alpha -m+\frac{m}{q}-\gamma +\frac{m}{p^{\prime }}%
-\delta }\ t^{\beta -n+\frac{n}{q}+\frac{n}{p^{\prime }}}\ .
\end{equation*}%
Thus in the presence of the power weight equality (\ref{bal mn}), the
boundedness of these two local characteristics $A_{p,q}^{\left( \alpha
,\beta \right) ,\left( m,n\right) }\left( v,w\right) $ in the indicated
ranges is equivalent to the four conditions:%
\begin{eqnarray*}
\alpha -m+\frac{m}{q}+\frac{m}{p^{\prime }} &\geq &0, \\
\alpha -m+\frac{m}{q}-\gamma +\frac{m}{p^{\prime }}-\delta &\leq &0, \\
\beta -n+\frac{n}{q}+\frac{n}{p^{\prime }} &\geq &0, \\
\beta -n+\frac{n}{q}-\gamma +\frac{n}{p^{\prime }}-\delta &\leq &0,
\end{eqnarray*}%
i.e.%
\begin{eqnarray*}
\Gamma &\leq &\frac{\alpha }{m}\leq \Gamma +\frac{\gamma +\delta }{m}, \\
\Gamma &\leq &\frac{\beta }{n}\leq \Gamma +\frac{\gamma +\delta }{n},
\end{eqnarray*}%
which are the second and third lines in (\ref{3 lines}) in this case.

\textbf{Subcase A2:} $\gamma =\frac{m}{q}$ and $m<n$. In this subcase $\frac{%
m}{q}-\gamma =0$ and $\frac{n}{q}-\gamma >0$, and so for $0<s\leq t<\infty $%
, we have%
\begin{equation*}
A_{p,q}^{\left( \alpha ,\beta \right) ,\left( m,n\right) }\left( v,w\right) 
\left[ I,J\right] =s^{\alpha -m+\frac{m}{q}+\frac{m}{p^{\prime }}}\ t^{\beta
-n+\frac{n}{q}-\gamma +\frac{n}{p^{\prime }}-\delta }\ ,
\end{equation*}%
and for $0<t\leq s<\infty $, we have from (\ref{understanding}) that%
\begin{equation*}
A_{p,q}^{\left( \alpha ,\beta \right) ,\left( m,n\right) }\left( v,w\right) 
\left[ I,J\right] =s^{\alpha -m+\frac{m}{p^{\prime }}-\delta }\ t^{\beta -n+%
\frac{n}{q}+\frac{n}{p^{\prime }}}\ln \frac{s}{t}\ .
\end{equation*}%
Thus the boundedness of these two local characteristics $A_{p,q}^{\left(
\alpha ,\beta \right) ,\left( m,n\right) }\left( v,w\right) $ in the
indicated ranges is equivalent to the four conditions:%
\begin{eqnarray*}
\alpha -m+\frac{m}{q}+\frac{m}{p^{\prime }} &\geq &0, \\
\alpha -m+\frac{m}{p^{\prime }}-\delta &<&0, \\
\beta -n+\frac{n}{q}+\frac{n}{p^{\prime }} &>&0, \\
\beta -n+\frac{n}{q}-\gamma +\frac{n}{p^{\prime }}-\delta &\leq &0,
\end{eqnarray*}%
i.e.%
\begin{eqnarray*}
\Gamma &\leq &\frac{\alpha }{m}<\Gamma +\frac{\gamma +\delta }{m}, \\
\Gamma &<&\frac{\beta }{n}\leq \Gamma +\frac{\gamma +\delta }{n},
\end{eqnarray*}%
which are the second and third lines in (\ref{3 lines}) in this subcase.

\textbf{Subcase A3:} $\gamma \in \left( \frac{m}{q},\frac{n}{q}\right) $. In
this subcase $\frac{m}{q}-\gamma <0$ and $\frac{n}{q}-\gamma >0$, and so for 
$0<s\leq t<\infty $, we have%
\begin{equation*}
A_{p,q}^{\left( \alpha ,\beta \right) ,\left( m,n\right) }\left( v,w\right) 
\left[ I,J\right] =s^{\alpha -m+\frac{m}{q}+\frac{m}{p^{\prime }}}\ t^{\beta
-n+\frac{n}{q}-\gamma +\frac{n}{p^{\prime }}-\delta }\ ,
\end{equation*}%
and for $0<t\leq s<\infty $, we have%
\begin{equation*}
A_{p,q}^{\left( \alpha ,\beta \right) ,\left( m,n\right) }\left( v,w\right) 
\left[ I,J\right] =s^{\alpha -m+\frac{m}{p^{\prime }}-\delta }\ t^{\beta -n+%
\frac{n}{q}+\frac{m}{q}-\gamma +\frac{n}{p^{\prime }}}\ .
\end{equation*}%
Thus the boundedness of these two local characteristics $A_{p,q}^{\left(
\alpha ,\beta \right) ,\left( m,n\right) }\left( v,w\right) $ in the
indicated ranges is equivalent to the four conditions:%
\begin{eqnarray*}
\alpha -m+\frac{m}{q}+\frac{m}{p^{\prime }} &\geq &0, \\
\alpha -m+\frac{m}{p^{\prime }}-\delta &\leq &0, \\
\beta -n+\frac{n}{q}+\frac{m}{q}-\gamma +\frac{n}{p^{\prime }} &\geq &0, \\
\beta -n+\frac{n}{q}-\gamma +\frac{n}{p^{\prime }}-\delta &\leq &0,
\end{eqnarray*}%
i.e.%
\begin{eqnarray*}
\Gamma &\leq &\frac{\alpha }{m}\leq \Gamma +\frac{\gamma +\delta }{m}-\frac{%
\gamma -\frac{m}{q}}{m}, \\
\Gamma +\frac{\gamma -\frac{m}{q}}{n} &\leq &\frac{\beta }{n}\leq \Gamma +%
\frac{\gamma +\delta }{n},
\end{eqnarray*}%
which are the second and third lines in (\ref{3 lines}) in this subcase.

\textbf{Subcase A4:} $\gamma =\frac{n}{q}$ and $m<n$. In this subcase $\frac{%
m}{q}-\gamma <0$ and $\frac{n}{q}-\gamma =0$, and so for $0<s\leq t<\infty $%
, we have from (\ref{understanding}) that%
\begin{equation*}
A_{p,q}^{\left( \alpha ,\beta \right) ,\left( m,n\right) }\left( v,w\right) 
\left[ I,J\right] =s^{\alpha -m+\frac{m}{q}+\frac{m}{p^{\prime }}}\ t^{\beta
-n+\frac{n}{p^{\prime }}-\delta }\ln \frac{t}{s}\ ,
\end{equation*}%
and for $0<t\leq s<\infty $, we have%
\begin{equation*}
A_{p,q}^{\left( \alpha ,\beta \right) ,\left( m,n\right) }\left( v,w\right) 
\left[ I,J\right] =s^{\alpha -m+\frac{m}{p^{\prime }}-\delta }\ t^{\beta -n+%
\frac{n}{q}+\frac{m}{q}-\gamma +\frac{n}{p^{\prime }}}\ .
\end{equation*}%
Thus the boundedness of these two local characteristics $A_{p,q}^{\left(
\alpha ,\beta \right) ,\left( m,n\right) }\left( v,w\right) $ in the
indicated ranges is equivalent to the four conditions:%
\begin{eqnarray*}
\alpha -m+\frac{m}{q}+\frac{m}{p^{\prime }} &>&0, \\
\alpha -m+\frac{m}{p^{\prime }}-\delta &\leq &0, \\
\beta -n+\frac{n}{q}+\frac{m}{q}-\gamma +\frac{n}{p^{\prime }} &\geq &0, \\
\beta -n+\frac{n}{p^{\prime }}-\delta &<&0,
\end{eqnarray*}%
i.e.%
\begin{eqnarray*}
\Gamma &<&\frac{\alpha }{m}\leq \Gamma +\frac{\frac{m}{q}+\delta }{m}, \\
\Gamma +\frac{\gamma -\frac{m}{q}}{n} &\leq &\frac{\beta }{n}<\Gamma +\frac{%
\gamma +\delta }{n},
\end{eqnarray*}%
which are the second and third lines in (\ref{3 lines}) in this subcase.

\textbf{Subcase A5:} $\gamma \in \left( \frac{n}{q},\infty \right) $. In
this subcase $\frac{m}{q}-\gamma <0$ and $\frac{n}{q}-\gamma <0$, and so for 
$0<s\leq t<\infty $, we have%
\begin{equation*}
A_{p,q}^{\left( \alpha ,\beta \right) ,\left( m,n\right) }\left( v,w\right) 
\left[ I,J\right] =s^{\alpha -m+\frac{m}{q}+\frac{n}{q}-\gamma +\frac{m}{%
p^{\prime }}}\ t^{\beta -n+\frac{n}{p^{\prime }}-\delta }\ ,
\end{equation*}%
and for $0<t\leq s<\infty $, we have%
\begin{equation*}
A_{p,q}^{\left( \alpha ,\beta \right) ,\left( m,n\right) }\left( v,w\right) 
\left[ I,J\right] =s^{\alpha -m+\frac{m}{p^{\prime }}-\delta }\ t^{\beta -n+%
\frac{n}{q}+\frac{m}{q}-\gamma +\frac{n}{p^{\prime }}}\ .
\end{equation*}%
Thus the boundedness of these two local characteristics $A_{p,q}^{\left(
\alpha ,\beta \right) ,\left( m,n\right) }\left( v,w\right) $ in the
indicated ranges is equivalent to the four conditions:%
\begin{eqnarray*}
\alpha -m+\frac{m}{q}+\frac{n}{q}-\gamma +\frac{m}{p^{\prime }} &\geq &0, \\
\alpha -m+\frac{m}{p^{\prime }}-\delta &\leq &0, \\
\beta -n+\frac{n}{q}+\frac{m}{q}-\gamma +\frac{n}{p^{\prime }} &\geq &0, \\
\beta -n+\frac{n}{p^{\prime }}-\delta &\leq &0,
\end{eqnarray*}%
i.e.%
\begin{eqnarray*}
\Gamma +\frac{\gamma -\frac{n}{q}}{m} &\leq &\frac{\alpha }{m}\leq \Gamma +%
\frac{\gamma +\delta }{m}-\frac{\gamma -\frac{m}{q}}{m}, \\
\Gamma +\frac{\gamma -\frac{m}{q}}{n} &\leq &\frac{\beta }{n}\leq \Gamma +%
\frac{\gamma +\delta }{n}-\frac{\gamma -\frac{n}{q}}{n},
\end{eqnarray*}%
which are the second and third lines in (\ref{3 lines}) in this subcase.

\textbf{Subcase A6:} $\gamma =\frac{n}{q}$ and $m=n$. In this case $\frac{m}{%
q}-\gamma =0$ and $\frac{n}{q}-\gamma =0$, and so for $0<s\leq t<\infty $,
we have from (\ref{understanding}) that%
\begin{equation*}
A_{p,q}^{\left( \alpha ,\beta \right) ,\left( m,n\right) }\left( v,w\right) 
\left[ I,J\right] =s^{\alpha -m+\frac{m}{q}+\frac{m}{p^{\prime }}}\ t^{\beta
-n+\frac{n}{p^{\prime }}-\delta }\ln \frac{t}{s}\ ,
\end{equation*}%
and for $0<t\leq s<\infty $, we have%
\begin{equation*}
A_{p,q}^{\left( \alpha ,\beta \right) ,\left( m,n\right) }\left( v,w\right) 
\left[ I,J\right] =s^{\alpha -m+\frac{m}{p^{\prime }}-\delta }\ t^{\beta -n+%
\frac{n}{q}+\frac{n}{p^{\prime }}}\ln \frac{s}{t}\ .
\end{equation*}%
Thus the boundedness of these two local characteristics $A_{p,q}^{\left(
\alpha ,\beta \right) ,\left( m,n\right) }\left( v,w\right) $ in the
indicated ranges is equivalent to the four conditions:%
\begin{eqnarray*}
\alpha -m+\frac{m}{q}+\frac{m}{p^{\prime }} &>&0, \\
\alpha -m+\frac{m}{p^{\prime }}-\delta &<&0, \\
\beta -n+\frac{n}{q}+\frac{n}{p^{\prime }} &>&0, \\
\beta -n+\frac{n}{p^{\prime }}-\delta &<&0,
\end{eqnarray*}%
i.e.%
\begin{eqnarray*}
\Gamma &<&\frac{\alpha }{m}<\Gamma +\frac{\gamma +\delta }{m}, \\
\Gamma &<&\frac{\beta }{n}<\Gamma +\frac{\gamma +\delta }{n},
\end{eqnarray*}%
which are the second and third lines in (\ref{3 lines}) in this case.

Now we turn to the next major case.

\textbf{Case B}: $\delta =\frac{m}{p^{\prime }}$. Then for $0<s\leq t<\infty 
$, we have%
\begin{equation*}
A_{p,q}^{\left( \alpha ,\beta \right) ,\left( m,n\right) }\left( v,w\right) 
\left[ I,J\right] =s^{\alpha -m+\frac{m}{q}-\left( \frac{n}{q}-\gamma
\right) _{-}+\frac{m}{p^{\prime }}}\ t^{\beta -n+\left( \frac{n}{q}-\gamma
\right) _{+}+\frac{n}{p^{\prime }}-\delta }\ ,
\end{equation*}%
and for $0<t\leq s<\infty $, we have from (\ref{understanding}) that%
\begin{equation*}
A_{p,q}^{\left( \alpha ,\beta \right) ,\left( m,n\right) }\left( v,w\right) 
\left[ I,J\right] =s^{\alpha -m+\left( \frac{m}{q}-\gamma \right) _{+}+\frac{%
m}{p^{\prime }}-\delta }\ t^{\beta -n+\frac{n}{q}-\left( \frac{m}{q}-\gamma
\right) _{-}+\frac{n}{p^{\prime }}}\ln \frac{s}{t}\ .
\end{equation*}%
Now we consider separately within Case B the (at most) five subcases
determined by $\gamma $ (since the open interval $\left( \frac{m}{q},\frac{n%
}{q}\right) $ is empty if $m=n$):%
\begin{equation*}
\gamma \in \left( -\infty ,\frac{m}{q}\right) ,\ \ \ \gamma =\frac{m}{q},\ \
\ \gamma \in \left( \frac{m}{q},\frac{n}{q}\right) ,\ \ \ \gamma =\frac{n}{q}%
,\ \ \ \gamma \in \left( \frac{n}{q},\infty \right) .
\end{equation*}

\textbf{Subcase B1:} $\gamma \in \left( -\infty ,\frac{m}{q}\right) $. In
this subcase $\frac{m}{q}-\gamma >0$ and $\frac{n}{q}-\gamma >0$, and so for 
$0<s\leq t<\infty $, we have%
\begin{equation*}
A_{p,q}^{\left( \alpha ,\beta \right) ,\left( m,n\right) }\left( v,w\right) 
\left[ I,J\right] =s^{\alpha -m+\frac{m}{q}+\frac{m}{p^{\prime }}}\ t^{\beta
-n+\frac{n}{q}-\gamma +\frac{n}{p^{\prime }}-\delta }\ ,
\end{equation*}%
and for $0<t\leq s<\infty $, we have%
\begin{equation*}
A_{p,q}^{\left( \alpha ,\beta \right) ,\left( m,n\right) }\left( v,w\right) 
\left[ I,J\right] =s^{\alpha -m+\frac{m}{q}-\gamma +\frac{m}{p^{\prime }}%
-\delta }\ t^{\beta -n+\frac{n}{q}+\frac{n}{p^{\prime }}}\ln \frac{s}{t}\ .
\end{equation*}%
Thus in the presence of the power weight equality (\ref{bal mn}), the
boundedness of these two local characteristics $A_{p,q}^{\left( \alpha
,\beta \right) ,\left( m,n\right) }\left( v,w\right) $ in the indicated
ranges is equivalent to the four conditions:%
\begin{eqnarray*}
\alpha -m+\frac{m}{q}+\frac{m}{p^{\prime }} &\geq &0, \\
\alpha -m+\frac{m}{q}-\gamma +\frac{m}{p^{\prime }}-\delta &<&0, \\
\beta -n+\frac{n}{q}+\frac{n}{p^{\prime }} &>&0, \\
\beta -n+\frac{n}{q}-\gamma +\frac{n}{p^{\prime }}-\delta &\leq &0,
\end{eqnarray*}%
i.e.%
\begin{eqnarray*}
\Gamma &\leq &\frac{\alpha }{m}<\Gamma +\frac{\gamma +\delta }{m}, \\
\Gamma &<&\frac{\beta }{n}\leq \Gamma +\frac{\gamma +\delta }{n},
\end{eqnarray*}%
which are the second and third lines in (\ref{3 lines}) in this subcase.

\textbf{Case B2:} $\gamma =\frac{m}{q}$. In this case $\frac{m}{q}-\gamma =0$
and $\frac{n}{q}-\gamma >0$, and so for $0<s\leq t<\infty $, we have%
\begin{equation*}
A_{p,q}^{\left( \alpha ,\beta \right) ,\left( m,n\right) }\left( v,w\right) 
\left[ I,J\right] =s^{\alpha -m+\frac{m}{q}+\frac{m}{p^{\prime }}}\ t^{\beta
-n+\frac{n}{q}-\gamma +\frac{n}{p^{\prime }}-\delta }\ ,
\end{equation*}%
and for $0<t\leq s<\infty $, we have%
\begin{equation*}
A_{p,q}^{\left( \alpha ,\beta \right) ,\left( m,n\right) }\left( v,w\right) 
\left[ I,J\right] =s^{\alpha -m+\frac{m}{p^{\prime }}-\delta }\ t^{\beta -n+%
\frac{n}{q}+\frac{n}{p^{\prime }}}\ln \frac{s}{t}\ .
\end{equation*}%
This coincides with Subcase A2 above and so is equivalent to the four
conditions%
\begin{eqnarray*}
\Gamma &\leq &\frac{\alpha }{m}<\Gamma +\frac{\gamma +\delta }{m}, \\
\Gamma &<&\frac{\beta }{n}\leq \Gamma +\frac{\gamma +\delta }{n},
\end{eqnarray*}%
which are the second and third lines in (\ref{3 lines}) in this subcase.

The remaining subcases \textbf{B3}, \textbf{B4}, \textbf{B5} and \textbf{B6}
are handled in similar fashion.

\textbf{Case C}: $\frac{m}{p^{\prime }}<\delta <\frac{n}{p^{\prime }}$. This
is handled in similar fashion.

\textbf{Case D}: $\delta =\frac{n}{p^{\prime }}$. This is handled in similar
fashion.

\textbf{Case E}: $\frac{n}{p^{\prime }}<\delta <\infty $. The subcases 
\textbf{E1}, \textbf{E2}, \textbf{E3}, \textbf{E4} and \textbf{E6}, \ are
handled in similar fashion, and we end with the remaining and final subcase.

\textbf{Subcase E5}: $\frac{n}{q}<\gamma <\infty $. In this case we have
both $\frac{n}{p^{\prime }}-\delta <0$ and $\frac{n}{q}-\gamma <0$ and so if 
$0<s\leq t<\infty $, we have 
\begin{equation*}
A_{p,q}^{\left( \alpha ,\beta \right) ,\left( m,n\right) }\left( v,w\right) 
\left[ I,J\right] =s^{\alpha -m+\frac{m}{q}+\frac{n}{q}-\gamma +\frac{m}{%
p^{\prime }}+\frac{n}{p^{\prime }}-\delta }\ t^{\beta -n}\ ,
\end{equation*}%
and similarly if $0<t\leq s<\infty $, we have%
\begin{equation*}
A_{p,q}^{\left( \alpha ,\beta \right) ,\left( m,n\right) }\left( v,w\right) 
\left[ I,J\right] =s^{\alpha -m}\ t^{\beta -n+\frac{n}{q}+\frac{m}{q}-\gamma
+\frac{n}{p^{\prime }}+\frac{m}{p^{\prime }}-\delta }\ .
\end{equation*}%
Thus in the presence of the power weight equality (\ref{bal mn}), the
boundedness of these two local characteristics $A_{p,q}^{\left( \alpha
,\beta \right) ,\left( m,n\right) }\left( v,w\right) $ in the indicated
ranges is equivalent to the four conditions:%
\begin{eqnarray*}
\alpha -m+\frac{m}{q}+\frac{n}{q}-\gamma +\frac{m}{p^{\prime }}+\frac{n}{%
p^{\prime }}-\delta &\geq &0, \\
\alpha -m &\leq &0, \\
\beta -n+\frac{n}{q}+\frac{m}{q}-\gamma +\frac{n}{p^{\prime }}+\frac{m}{%
p^{\prime }}-\delta &\geq &0, \\
\beta -n &\leq &0,
\end{eqnarray*}%
i.e.%
\begin{eqnarray*}
\Gamma -\frac{n}{m}\left( \frac{1}{q}+\frac{1}{p^{\prime }}\right) +\frac{%
\gamma +\delta }{m} &\leq &\frac{\alpha }{m}\leq 1, \\
\Gamma -\frac{m}{n}\left( \frac{1}{q}+\frac{1}{p^{\prime }}\right) +\frac{%
\gamma +\delta }{n} &\leq &\frac{\beta }{n}\leq 1.
\end{eqnarray*}%
which are the second and third lines in (\ref{3 lines}) in this subcase.
Indeed, in this subcase, the second and third lines in (\ref{3 lines}) are 
\begin{eqnarray}
\Gamma +\frac{\bigtriangleup _{p,q}^{\gamma ,\delta }\left( n\right) }{m}
&\leq &\frac{\alpha }{m}\leq \Gamma +\frac{\gamma +\delta }{m}-\frac{%
\bigtriangleup _{p,q}^{\gamma ,\delta }\left( m\right) }{m},  \notag \\
\Gamma +\frac{\bigtriangleup _{p,q}^{\gamma ,\delta }\left( m\right) }{n}
&\leq &\frac{\beta }{n}\leq \Gamma +\frac{\gamma +\delta }{n}-\frac{%
\bigtriangleup _{p,q}^{\gamma ,\delta }\left( n\right) }{n},  \notag
\end{eqnarray}%
where%
\begin{eqnarray*}
\bigtriangleup _{p,q}^{\gamma ,\delta }\left( m\right) &\equiv &\left(
\gamma -\frac{m}{q}\right) _{+}+\left( \delta -\frac{m}{p^{\prime }}\right)
_{+}\ , \\
\bigtriangleup _{p,q}^{\gamma ,\delta }\left( n\right) &\equiv &\left(
\gamma -\frac{n}{q}\right) _{+}+\left( \delta -\frac{n}{p^{\prime }}\right)
_{+}\ .
\end{eqnarray*}%
This is equivalent to 
\begin{eqnarray}
\Gamma +\frac{\gamma -\frac{n}{q}+\delta -\frac{n}{p^{\prime }}}{m} &\leq &%
\frac{\alpha }{m}\leq \Gamma +\frac{\gamma +\delta }{m}-\frac{\gamma -\frac{m%
}{q}+\delta -\frac{m}{p^{\prime }}}{m},  \notag \\
\Gamma +\frac{\gamma -\frac{m}{q}+\delta -\frac{m}{p^{\prime }}}{n} &\leq &%
\frac{\beta }{n}\leq \Gamma +\frac{\gamma +\delta }{n}-\frac{\gamma -\frac{n%
}{q}+\delta -\frac{n}{p^{\prime }}}{n},  \notag
\end{eqnarray}%
i.e.%
\begin{eqnarray}
\Gamma -\frac{n}{m}\left( \frac{1}{q}+\frac{1}{p^{\prime }}\right) +\frac{%
\gamma +\delta }{m} &\leq &\frac{\alpha }{m}\leq \Gamma +\frac{1}{q}+\frac{1%
}{p^{\prime }}=1,  \notag \\
\Gamma -\frac{m}{n}\left( \frac{1}{q}+\frac{1}{p^{\prime }}\right) +\frac{%
\gamma +\delta }{n} &\leq &\frac{\beta }{n}\leq \Gamma +\frac{1}{q}+\frac{1}{%
p^{\prime }}=1.  \notag
\end{eqnarray}

Finally we note that these families of inequalities for $\alpha ,\beta $
remain the same when $m\geq n$. This concludes the proof of Theorem \ref{A
char mn}.

\begin{corollary}
\label{cor to power}If $A_{p,q}^{\left( \alpha ,\beta \right) ,\left(
m,n\right) }\left( v_{\delta },w_{\gamma }\right) <\infty $, then the
following inequalities hold:%
\begin{eqnarray*}
\alpha &\leq &m\text{ and }\beta \leq n, \\
\frac{\alpha }{m} &\leq &\min \left\{ \frac{\gamma }{m}+\frac{1}{q^{\prime }}%
,\ \frac{\delta }{m}+\frac{1}{p}\right\} , \\
\frac{\beta }{n} &\leq &\min \left\{ \frac{\gamma }{n}+\frac{1}{q^{\prime }}%
,\ \frac{\delta }{n}+\frac{1}{p}\right\} .
\end{eqnarray*}%
In addition, we have the corresponding \emph{strict} inequalities in the
following cases:%
\begin{eqnarray}
\frac{\alpha }{m} &<&1\text{ when either }\frac{\gamma }{m}-\frac{1}{q}\leq 0%
\text{ or }\frac{\delta }{m}-\frac{1}{p^{\prime }}\leq 0,  \label{corres} \\
\frac{\beta }{n} &<&1\text{ when either }\frac{\gamma }{n}-\frac{1}{q}\leq 0%
\text{ or }\frac{\delta }{n}-\frac{1}{p^{\prime }}\leq 0,  \notag \\
\frac{\alpha }{m} &<&\frac{\gamma }{m}+\frac{1}{q^{\prime }}\text{ when
either }\frac{\gamma }{m}-\frac{1}{q}\geq 0\text{ or }\frac{\delta }{m}-%
\frac{1}{p^{\prime }}\leq 0,  \notag \\
\frac{\alpha }{m} &<&\frac{\delta }{m}+\frac{1}{p}\text{ when either }\frac{%
\gamma }{m}-\frac{1}{q}\leq 0\text{ or }\frac{\delta }{m}-\frac{1}{p^{\prime
}}\geq 0,  \notag \\
\frac{\beta }{n} &<&\frac{\gamma }{n}+\frac{1}{q^{\prime }}\text{ when
either }\frac{\gamma }{n}-\frac{1}{q}\geq 0\text{ or }\frac{\delta }{n}-%
\frac{1}{p^{\prime }}\leq 0,  \notag \\
\frac{\beta }{n} &<&\frac{\delta }{n}+\frac{1}{p}\text{ when either }\frac{%
\gamma }{n}-\frac{1}{q}\leq 0\text{ or }\frac{\delta }{n}-\frac{1}{p^{\prime
}}\geq 0.  \notag
\end{eqnarray}%
In particular, these strict inequalities for $\frac{\alpha }{m}$ and $\frac{%
\beta }{n}$ all hold if both $\gamma \leq \frac{\min \left\{ m,n\right\} }{q}
$ and $\delta \leq \frac{\min \left\{ m,n\right\} }{p^{\prime }}$ hold (c.f.
(\ref{local integ}) which shows that $\gamma <\frac{m+n}{q}$ and $\delta <%
\frac{m+n}{p^{\prime }}$ are required by finiteness of $A_{p,q}^{\left(
\alpha ,\beta \right) ,\left( m,n\right) }\left( v_{\delta },w_{\gamma
}\right) $).
\end{corollary}

\begin{proof}
We compute%
\begin{eqnarray*}
\frac{\alpha }{m} &\leq &\Gamma +\frac{\gamma +\delta }{m}-\frac{%
\bigtriangleup _{p,q}^{\gamma ,\delta }\left( m\right) }{m} \\
&=&\frac{1}{p}-\frac{1}{q}+\frac{\gamma +\delta }{m}-\left\{ \left( \frac{%
\gamma }{m}-\frac{1}{q}\right) _{+}+\left( \frac{\delta }{m}-\frac{1}{%
p^{\prime }}\right) _{+}\right\} \\
&=&1-\left\{ \left( \frac{\gamma }{m}-\frac{1}{q}\right) _{-}+\left( \frac{%
\delta }{m}-\frac{1}{p^{\prime }}\right) _{-}\right\} ,
\end{eqnarray*}%
and similarly%
\begin{equation*}
\frac{\beta }{n}\leq 1-\left\{ \left( \frac{\gamma }{n}-\frac{1}{q}\right)
_{-}+\left( \frac{\delta }{n}-\frac{1}{p^{\prime }}\right) _{-}\right\} .
\end{equation*}%
Using $-x_{-}=\min \left\{ 0,x\right\} $ we also compute that%
\begin{eqnarray*}
\frac{\alpha }{m} &\leq &1-\left\{ \left( \frac{\gamma }{m}-\frac{1}{q}%
\right) _{-}+\left( \frac{\delta }{m}-\frac{1}{p^{\prime }}\right)
_{-}\right\} \\
&\leq &\min \left\{ 1-\left( \frac{\gamma }{m}-\frac{1}{q}\right)
_{-},1-\left( \frac{\delta }{m}-\frac{1}{p^{\prime }}\right) _{-}\right\} \\
&\leq &\min \left\{ 1+\frac{\gamma }{m}-\frac{1}{q},1+\frac{\delta }{m}-%
\frac{1}{p^{\prime }}\right\} \\
&=&\min \left\{ \frac{\gamma }{m}+\frac{1}{q^{\prime }},\frac{\delta }{m}+%
\frac{1}{p}\right\} ,
\end{eqnarray*}%
and moreover that 
\begin{eqnarray*}
\frac{\alpha }{m} &<&\frac{\gamma }{m}+\frac{1}{q^{\prime }}\text{ when
either }\frac{\delta }{m}-\frac{1}{p^{\prime }}\leq 0\text{ or }\frac{\gamma 
}{m}-\frac{1}{q}\geq 0, \\
\frac{\alpha }{m} &<&\frac{\delta }{m}+\frac{1}{p}\text{ when either }\frac{%
\gamma }{m}-\frac{1}{q}\leq 0\text{ or }\frac{\delta }{m}-\frac{1}{p^{\prime
}}\geq 0.
\end{eqnarray*}%
Similarly we have%
\begin{equation*}
\frac{\beta }{n}\leq \min \left\{ \frac{\gamma }{n}+\frac{1}{q^{\prime }},%
\frac{\delta }{n}+\frac{1}{p}\right\} ,
\end{equation*}%
and moreover that%
\begin{eqnarray*}
\frac{\beta }{n} &<&\frac{\gamma }{n}+\frac{1}{q^{\prime }}\text{ when
either }\frac{\delta }{n}-\frac{1}{p^{\prime }}\leq 0\text{ or }\frac{\gamma 
}{n}-\frac{1}{q}\geq 0, \\
\frac{\beta }{n} &<&\frac{\delta }{n}+\frac{1}{p}\text{ when either }\frac{%
\gamma }{n}-\frac{1}{q}\leq 0\text{ or }\frac{\delta }{n}-\frac{1}{p^{\prime
}}\geq 0.
\end{eqnarray*}
\end{proof}

We end the appendix with a variant of Lemma \ref{tails} in which the
restrictions $0<\alpha \leq m$ and $0<\beta \leq n$ are no longer needed if
one of the weights is a power weight.

\begin{proposition}
\label{tails'}Suppose at least one of the weights $\sigma ,\omega $ is a
power weight on $\mathbb{R}^{m+n}$. Then for $1<p,q<\infty $ and $\alpha
,\beta >0$, we have 
\begin{equation*}
\widehat{\mathbb{A}}_{p,q}^{\left( \alpha ,\beta \right) ,\left( m,n\right)
}\left( \sigma ,\omega \right) \leq C\ \mathbb{N}_{p,q}^{\left( \alpha
,\beta \right) ,\left( m,n\right) }\left( \sigma ,\omega \right) ,
\end{equation*}%
where $C$ is a positive constant depending on $m,n,\alpha ,\beta ,p,q$ and
the power weight.
\end{proposition}

By duality we may suppose without loss of generality that $\sigma $ is a
locally integrable power weight. Indeed, modifying slightly the proof of
Lemma \ref{tails}, we have with 
\begin{equation*}
\mathcal{E}_{x,y}\left( I,J\right) \equiv \left\{ \left( u,t\right) \in 
\mathbb{R}^{m}\times \mathbb{R}^{n}:\left\vert x-u\right\vert \geq
\left\vert I\right\vert ^{\frac{1}{m}}\text{ and }\left\vert y-t\right\vert
\geq \left\vert J\right\vert ^{\frac{1}{n}}\right\} ,
\end{equation*}%
that%
\begin{equation*}
\left\vert x-u\right\vert ^{\alpha -m}\ \left\vert y-t\right\vert ^{\beta
-n}\geq \mathbf{1}_{\mathcal{E}_{x,y}\left( I,J\right) }\left( \left(
u,t\right) \right) \ \left\vert I\right\vert ^{\frac{\alpha }{m}%
-1}\left\vert J\right\vert ^{\frac{\beta }{n}-1}\widehat{s}_{I\times
J}\left( x,y\right) \widehat{s}_{I\times J}\left( u,t\right) \ ,
\end{equation*}%
and so for $R>0$ and $f_{R}\left( u,t\right) \equiv \mathbf{1}_{B\left(
0,R\right) \times B\left( 0,R\right) }\left( u,t\right) \widehat{s}%
_{Q}\left( u,t\right) ^{p^{\prime }-1}$, we have%
\begin{eqnarray*}
I_{\alpha ,\beta }^{m,n}\left( f_{R}\sigma \right) \left( x,y\right)
&=&\diint\limits_{B\left( 0,R\right) \times B\left( 0,R\right) }\left\vert
x-u\right\vert ^{\alpha -n}\left\vert y-t\right\vert ^{\beta -n}\widehat{s}%
_{Q}\left( u,t\right) ^{p^{\prime }-1}d\sigma \left( u,t\right) \\
&\geq &\diint\limits_{B\left( 0,R\right) \times B\left( 0,R\right) }\mathbf{1%
}_{\mathcal{E}_{x,y}\left( I,J\right) }\left( \left( u,t\right) \right) \
\left\vert I\right\vert ^{\frac{\alpha }{m}-1}\left\vert J\right\vert ^{%
\frac{\beta }{n}-1}\widehat{s}_{I\times J}\left( x,y\right) \widehat{s}%
_{I\times J}\left( u,t\right) \widehat{s}_{Q}\left( u,t\right) ^{p^{\prime
}-1}d\sigma \left( u,t\right) \\
&=&\left\vert I\right\vert ^{\frac{\alpha }{m}-1}\left\vert J\right\vert ^{%
\frac{\beta }{n}-1}\widehat{s}_{I\times J}\left( x,y\right)
\diint\limits_{B\left( 0,R\right) \times B\left( 0,R\right) }\mathbf{1}_{%
\mathcal{E}_{x,y}\left( I,J\right) }\left( \left( u,t\right) \right) \ 
\widehat{s}_{I\times J}\left( u,t\right) ^{p^{\prime }}d\sigma \left(
u,t\right) .
\end{eqnarray*}%
It is at this point that we use our assumption that $\sigma \left(
u,t\right) =\left\vert \left( u,t\right) \right\vert ^{\rho }$, $\rho
>-\left( m+n\right) $, is a locally integrable power weight in order to
conclude that%
\begin{equation*}
\diint\limits_{B\left( 0,R\right) \times B\left( 0,R\right) }\mathbf{1}_{%
\mathcal{E}_{x,y}\left( I,J\right) }\left( \left( u,t\right) \right) \ 
\widehat{s}_{I\times J}\left( u,t\right) ^{p^{\prime }}d\sigma \left(
u,t\right) \geq c_{m,n,\alpha ,\beta ,p,q,\rho }\diint\limits_{B\left(
0,R\right) \times B\left( 0,R\right) }\widehat{s}_{I\times J}\left(
u,t\right) ^{p^{\prime }}d\sigma \left( u,t\right)
\end{equation*}%
holds for $\left( x,y\right) \in \mathbb{R}^{m}\times \mathbb{R}^{n}$ with a
constant $c_{m,n,\alpha ,\beta ,p,q,\rho }$ independent of $\left(
x,y\right) \in \mathbb{R}^{m}\times \mathbb{R}^{n}$. Now we continue with
the proof of Lemma \ref{tails} as given earlier to conclude that%
\begin{equation*}
\widehat{\mathbb{A}}_{p,q}^{\left( \alpha ,\beta \right) ,\left( m,n\right)
}\left( \sigma ,\omega \right) \leq \left( \frac{1}{c_{m,n,\alpha ,\beta
,p,q,\rho }}\right) ^{\frac{1}{q}}\mathbb{N}_{p,q}^{\left( \alpha ,\beta
\right) ,\left( m,n\right) }\left( \sigma ,\omega \right) .
\end{equation*}

\end{document}